\documentclass[12pt,onecolumn]{article}
\usepackage[margin=2.5cm]{geometry}
\usepackage[utf8]{inputenc}
\usepackage{amsmath,amssymb,amsthm} 
\usepackage{color}
\usepackage[font=small,labelfont=bf]{caption}
\usepackage{wrapfig}
\usepackage{enumitem}
\usepackage{authblk}
\usepackage{graphicx}
\usepackage{hyperref} 
\usepackage{empheq}
\usepackage{mdframed}
\usepackage{wrapfig}
\usepackage{ulem}
\usepackage{empheq}
\usepackage[table]{xcolor}
\usepackage{tikz}
\usepackage{amssymb}
\usepackage{epstopdf}
\usepackage{array,amsfonts,amsmath,amssymb,graphicx,relsize,url,array,indentfirst,booktabs,geometry,subfig}
\usepackage{wasysym}
\usepackage{listings}
\usepackage{amsthm}
\usepackage{mathtools}
\usepackage{esint}
\def \ve {\varepsilon}

\def \cL {\mathcal{L}}
\def \cN {\mathcal{N}}

\def \tN {\text{N}}

\def\Q{{\boldsymbol{Q}}}
\def\u{{\boldsymbol{u}}}
\def\x{{\boldsymbol{x}}}
\def\y{{\boldsymbol{y}}}
\def\V{{\boldsymbol{V}}}
\def\P{{\mathcal{P}}}
\def\H{{\bf H}}
\def\OMEGA{{\boldsymbol{\omega}}}
\def\D{{\boldsymbol{D}}}
\def\A{{\boldsymbol{A}}}

\theoremstyle{definition}

\newtheorem{thm}{Theorem}[section]
\newtheorem{proposition}{Proposition}[section]
\newtheorem{lemma}{Lemma}[section]

\newtheorem{remark}{Remark}[section]

\def\XXint#1#2#3{{\setbox0=\hbox{$#1{#2#3}{\int}$ }
		\vcenter{\hbox{$#2#3$ }}\kern-.6\wd0}}

\begin{document} 

\title{Dynamics and steady state of squirmer motion\\
in liquid crystal}


\author[1]{Leonid Berlyand}
\author[1]{Hai Chi}
\author[2]{Mykhailo Potomkin}
\author[3]{Nung Kwan Yip}

\affil[1]{Department of Mathematics, Pennsylvania State University, University Park, PA 16802, USA}

\affil[2]{Department of Mathematics, University of California at Riverside, Riverside, CA 92521, USA}

\affil[3]{Department of Mathematics, Purdue University, 150 N. University Street
	West Lafayette, IN 47907-2067, USA}

\date{}

\maketitle

\begin{abstract}
We analyze a nonlinear PDE system describing the motion of a microswimmer in a nematic liquid crystal environment. For the microswimmer's motility, the squirmer model is used in which self-propulsion enters the model through the slip velocity on the microswimmer's surface. The liquid crystal is described using the well-established Beris-Edwards formulation. In previous computational studies, it was shown that the squirmer, regardless of its initial configuration, eventually orients itself either parallel or perpendicular to the preferred orientation dictated by the liquid crystal. Furthermore, the corresponding solution of the coupled nonlinear system converges to a steady state. 
In this work, we rigourously establish the existence of steady state and also the finite-time existence for the time-dependent problem in a periodic domain. 
Finally, we will use a two-scale asymptotic expansion to derive a homogenized 
model for the collective swimming of squirmers as they reach their steady state orientation and speed.               
\end{abstract}


\tableofcontents


\section{Introduction}
Microswimmers are objects of micron size which are immersed in a fluid and capable of autonomous motion. They are ubiquitous in nature,
as examplified by bacteria and eukaryotic cells. 
Recently, synthetic microswimmers for applications in medicine and material repair has been introduced in \cite{paxton2004catalytic}, see also reviews \cite{wong2016synthetic,dey2017chemical}. Transport of microswimmers, both living and synthetic, as well as effective properties of suspensions populated by many such microswimmers largely depends on how they respond to surrounding environment. Modeling microswimmers has become a growing area of research. The case when microswimmers are immersed in a Newtonian fluid has been intensively studied -- see \cite{sokolov2009reduction,LopGacDouAurCle2015,HaiAraBerKar08,HaiSokAraBer09,HaiAroBerKar10,PotKaiBerAra2017,rubio2021alignment} and reviews \cite{MarJoaRamLivProRaoSim2013,ElgWinGom2015,gompper20202020,igor2022bacteria}. 
However, bacteria often swim in biofluids which demonstrate viscoelastic or anisotropic properties very different from those of isotropic Newtonian fluids. 
For example, {\it Helicobacter pylori} bacteria are present in stomach and are
associated with diseases such as chronic atrophic gastritis and ulcer \cite{parsonnet1991gastric,bravo2018pylori}. The ``success" in the inflammation of stomach walls by {\it H. Pylori} depends on how the bacterium reorients itself in the mucous protective layer. Note that mucus is a viscoelastic fluid which exhibits properties of a liquid crystal for a certain range of macroscopic parameters \cite{viney1999mucus,nuris2019mucus}. In addition to medical relevance, experimental realization which combines bacteria with a nematic water-based and non-toxic (to bacteria) liquid crystal led to a wealth of intriguing observations such as collective phenomena for small bacterial concentrations, moving topological defects, and visualization of flagella beating \cite{ZhoSokLavAra2014,Zhou_2017,genkin2017defects,Ara2018,Zho2018,turiv2020polar}.     
 
One of the most well-established model of microswimmer is the so called {\it squirmer}. The model was initially introduced in \cite{lighthil1952squirmer} for {\it Paramecium}, a micro-organism which swims with the help of small elastic 
appendages called cilia. The main modeling assumption for squirmers is that the body is non-deformable and the swimming effect is introduced via a given slip velocity profile on the body surface that models the cilia's activity. Analysis of squirmers immersed in a Newtonian fluid, from the well-posedness to the relation between the slip profile and the resulting velocity has been the focus of many authors,  \cite{galdi1999steady,galdi2007navier,michelin2010efficiency,michelin2013strokes,guo2014cilia,chisholm2016squirmer,guo2021optimal}.  

To describe a nematic liquid crystal we use the well-established Beris-Edwards model \cite{BerEdw1994}, a highly nonlinear PDE model coupling Navier-Stokes (or Stokes) equation with a PDE for the tensor order parameter which indicates the preferred local orientation as well as the strength of the local alignment of the
liquid crystal molecules. 
Well-posedness of the Beris-Edwards model in $\mathbb R^2$ and $\mathbb R^3$ were first studied in \cite{PaiZar2012,paicu2011}. Existence of weak and strong 
solutions in a bounded domain with a fixed boundary and both homogeneous and inhomogeneous boundary conditions for the tensor order parameter were established in \cite{ade/1448323166,AbeDolLiu2013}.  


In our work, we consider a model which combines a Beris-Edwards liquid crystal with a squirmer. Such a system was, for example, used as a computational model in \cite{LinWurStr2017} to study orientation dynamics of the spherical squirmer with respect to the preferred orientation of the liquid crystal. In \cite{chi2020surface}, we extended this model to elongated squirmers and studied how the long-term orientation dynamics of the squirmer depends on physical and geometrical parameters. To the best of our knowledge, there are no analytical results, 
such as well-posedness or model reductions via multi-scale limits for squirmers immersed in a liquid crystal environment. On the other hand, there 
is a range of results for particles in classical isotropic fluids.
We refer to \cite{duerinckx2021corrector, duerinckx2022effective, duerinckx2021continuum, duerinckx2020einstein, duerinckx2022sedimentation} for some recent works.

The structure of this paper is as follows. 
In Section~\ref{sec:model}, we present the Beris-Edwards model coupled with a squirmer for both the time-dependent 
and steady state problems. The latter corresponds to a squirmer moving with a 
constant velocity. 
In Section~\ref{sec:main-result}, we formulate our main results on the 
existence of solutions to both the steady state and time-dependent problems as well as a two-scale homogenization limit resulting in an effective model for a suspension of squirmers swimming parallel to each other. The last statement
can be considered as a steady motion of a bacterial colony. 
Proof of the main results are presented in Sections~\ref{sec:steady-state}, \ref{sec:time-dependent}, and \ref{sec:homo_expansion}. Some calculations and non-dimensionalization are relegated to Appendix.   


\section{Model}
\label{sec:model}

\subsection{Time dependent PDE system}
Consider a rigid squirmer swimming in a liquid crystal with translational and angular velocities $\V(t)$ and $\OMEGA(t)$, respectively. 
In the context of the Beris-Edwards model, the liquid crystal is described by 
a velocity field $\u(\x,t):\mathbb R^d\times \mathbb R_+\mapsto \mathbb R^d$ 
and a tensor order field $\Q(\x,t):\mathbb R^d\times \mathbb R_+\mapsto \mathbb R^{d\times d}$ taking values in symmetric traceless $d\times d$ matrices.
Here $d=2,3$ is the spatial dimension.
The functions $\u=(u_j)_{j=1}^{d}$ and $\Q=(Q_{ij})_{i,j=1}^{d}$ satisfy the following system of partial differential equations and boundary conditions, written in the frame moving with velocity $\V(t)$, so the squirmer is always centered at $0$:  
\begin{eqnarray}
	&\rho(\partial_t  + \u\cdot \nabla)\u + \rho \frac{d\V}{dt}  = \nabla \cdot (\sigma_{\text{hydro}}+\sigma_{\text{ela}}),\text{ in } \Pi \setminus \P(t) &\label{eq:time_lc_1}\\
	&\nabla\cdot \u = 0,\text{ in }\Pi \setminus \P(t),&\label{eq:time_lc_2}\\
	&\u = u_{\text{sq}}(\boldsymbol{\alpha}(t),\x)\boldsymbol{\tau} + \OMEGA(t) \times \x,\text{ on }\partial \P(t),&\label{eq:time_lc_3}\\
	&\partial_t \Q + (\u\cdot \nabla) \Q -S(\nabla \u, \Q)= \Gamma \left( K\Delta \Q + \hat{\H}(\Q) \right) + F_{\text{ext}}(\Q, \Q_\infty),\text{ in }\Pi \setminus \P(t),&\label{eq:time_lc_4}\\
	&\Q, \u, \nabla p \quad \text{periodic in} \, \Pi \label{eq:times_lc_5}\\
	&K\partial_\nu \Q =W(\Q_{\text{pref}}-\Q) \text{ on }\partial \P(t).&\label{eq:time_lc_6}
\end{eqnarray}
Here $\Pi = (-L, L)^{d}$ is a periodic box, 
$\P(t)$ and $\partial \P(t)$ are the domain occupied by the squirmer 
and its surface in the moving frame.
We assume that $\overline{\mathcal{P}(t)}\subset \Pi$ for all $t\geq 0$. We will also use the notation $\Omega(t):=\Pi\setminus\mathcal{P}(t)$
\textcolor{black}{to denote the fluid region}. 

Equation \eqref{eq:time_lc_1} is a modified Navier-Stokes equation for the 
velocity $\u(\x,t)$ which satisfies the divergence-free condition
\eqref{eq:time_lc_2}. To this effect, we have
$\sigma_{\text{hydro}}=\eta(\nabla \u+(\nabla \u)^{\text{T}})-p\mathbb{I}$ 
to be the standard isotropic stress tensor
where $p(\x,t)$ is the pressure of the liquid crystal with 
uniform density $\rho$ and viscosity $\eta$. 
The internal structure of the liquid crystal, i.e., local preferred direction and order, affects the flow through an additional elastic stress $\sigma_{\text{ela}}$ given by 
\begin{eqnarray}
\sigma_\text{ela}&=&K\left[(\Q\,\Delta \Q- \Delta \Q\, \Q)-\nabla \Q\astrosun \nabla \Q\right]\nonumber\\
&&-\xi\left[\H(\Q+\dfrac{\mathbb I}{d})+(\Q+\dfrac{\mathbb{I}}{d})\H-2(\Q+\frac{\mathbb{I}}{d})\text{Tr}(\Q\H)\right]. 
\label{def_sigma_ela}  
\end{eqnarray}   
Here $K$ is the elastic constant and $\nabla \Q\astrosun \nabla \Q$ is a $d\times d$ matrix with the $(k,l)$ component $\sum\limits_{i,j}\partial_{x_k}Q_{ij}\partial_{x_l}Q_{ij}$. The parameter $\xi$ measures the ratio between 
tumbling and aligning that a shear flow exerts on the liquid crystal molecules. The matrix-valued function $\H=\H(\Q)$ is defined as $\H(\Q)=\hat{\H}(\Q)+K\Delta \Q$ where $\hat{\H}(\Q)$ is
\begin{equation}
\hat{\H}(\Q):=a \Q-c \Q \text{Tr}(\Q^2)=-\nabla_\Q\left(\dfrac{c}{4}\|\Q\|^4-\dfrac{a}{2}\|\Q\|^2 \right)=-\nabla_\Q\hat{\mathcal{F}}(\Q).\label{def_of_H_hat}
\end{equation} 
The scalar potential $\hat{\mathcal{F}}(\Q)$ is the polynomial part of the Landau-de Gennes free energy whose coefficients $a$ and $c$ depend on macroscopic parameters of the liquid crystal such as temperature. The potential $\hat{\mathcal{F}}(\Q)$ attains minima at $\Q={\bf 0}$, corresponding to the isotropic state when the liquid crystal flows as a Newtonian fluid, and at tensor order parameters $\Q$ with $q_{\infty}:=\|\Q\|=\sqrt{\dfrac{a}{c}}$, corresponding to the equilibrium liquid crystalline states. 

Boundary conditions \eqref{eq:time_lc_3} describes how the squirmer interacts with the flow $\u$ of the liquid crystal. 
The orientation of the squrimer is described by a vector 
$\boldsymbol{\alpha}(t)\in \mathcal{S}^{d-1}$.
We also let $\boldsymbol{\tau}$ to be a tangent vector field to the surface of the squirmer which can be chosen to be
$\boldsymbol{\tau}:=(\boldsymbol{\alpha}\times\boldsymbol{\nu})\times\boldsymbol{\nu}$ where $\boldsymbol{\nu}$ is the inward normal vector on the squirmer's surface $\partial \mathcal{P}(t)$. A typical example of the slip velocity $u_{\text{sq}}$ is given by \cite{lighthil1952squirmer} (which is also used in the 
computational work \cite{LinWurStr2017,chi2020surface}) 
\begin{equation*}
u_{\text{sq}}(\boldsymbol{\alpha}(t),\boldsymbol{x})=v_{\text{prop}}\sin \theta (1+\beta \cos(\theta)), \text{ where }\theta=\cos^{-1}\left[\dfrac{\boldsymbol{x}\cdot\boldsymbol{\alpha}}{\|\boldsymbol{x}\|}\right].
\end{equation*}
Here $\theta$ is the azimuthal angle on the squirmer, the parameter $v_\text{prop}$ is proportional to propulsion strength, and $\beta$ quantifies the type of the squirmer (puller vs pusher; see \cite{chi2020surface} for details). In this work, we consider $u_{\text{sq}}(\boldsymbol{\alpha}(t),\boldsymbol{x})=\sin(\theta) g(\theta)$ with a smooth function $g(\theta)$. Note that such $u_{\text{sq}}$ vanishes at points of singularity of the vector field $\boldsymbol{\tau}$.   

The instantaneous angular velocity of the squirmer is denoted by $\OMEGA(t)$.
Then any material point $\x$ on the squirmer surface $\partial \mathcal{P}(t)$ 
will move with velocity $\OMEGA(t)\times \x$ in the moving frame for which the system \eqref{eq:time_lc_1}-\eqref{eq:time_lc_6} is written. Then the boundary condition \eqref{eq:time_lc_3} states that there is a given slip velocity $u_{\text{sq}}(\boldsymbol{\alpha}(t),\x)\boldsymbol{\tau}$ with the no-penetration condition:
\begin{eqnarray}
\left[\u-\OMEGA\times \x\right]\times \boldsymbol{\nu}&=&u_{\text{sq}}(\boldsymbol{\alpha}(t),\x)\boldsymbol{\tau}\times \boldsymbol{\nu},\nonumber\\
\left[\u-\OMEGA\times \x\right]\cdot \boldsymbol{\nu}&=&0.\nonumber
\end{eqnarray}    
The given non-zero slip velocity models self-propulsion of the squirmer. Such a condition was originally derived for micro-organisms swimming with the help of small elastic appendages (cilia) distributed on the surface \cite{lighthil1952squirmer}.      

The matrix-valued equation \eqref{eq:time_lc_4} describes the dynamics of 
$\Q(\x,t)$. While the two first terms in the left-hand side of \eqref{eq:time_lc_4} is the advection derivative, the third term $S(\nabla \u,\Q)$ describes how the flow gradient $\nabla\u$ rotates and stretches the order-parameter $\Q$, and is given by 
\begin{eqnarray}\label{def_of_S}
S(\nabla \u,\Q)=(\xi \D + \A)\left(\Q+\frac{\mathbb I}{d}\right)+\left(\Q+\frac{\mathbb I}{d}\right)(\xi \D- \A)-2\xi\left(\Q+\frac{\mathbb I}{d}\right)\text{tr}(\Q\nabla \u),
\end{eqnarray}
where $\D=\dfrac{1}{2}\left[\nabla \u +(\nabla \u)^{\text{T}}\right]$ and $\A=\dfrac{1}{2}\left[\nabla \u -(\nabla \u)^{\text{T}}\right]$ are symmetric and anti-symmetric parts of $\nabla \u$, respectively. The right-hand side of \eqref{eq:time_lc_4} consists of the term leading to the minimization of the total Landau-de Gennes energy 
\begin{equation}
\mathcal{E}_{\text{LdG}}(\Q)=\int\limits_{\Omega(t)}\hat{\mathcal{F}}(\Q)+\frac{K}{2}|\nabla \Q|^2\,\text{d}\x
\end{equation}
with the relaxation parameter $\Gamma>0$ and the term $F_{\text{ext}}(\Q, \Q_\infty)$ describing the aligning effect with an external field. This term imposes the equilibrium condition for liquid crystal, that is, in the absence of squirmer we have $\Q\equiv \Q_\infty$. We chose  $\Q_{\infty}=q_{\infty}(\boldsymbol{e}_x\otimes\boldsymbol{e}_x-\frac{\mathbb I}{d})$ which means that if the liquid crystal is not perturbed by a squirmer then its molecules are oriented parallel to $\boldsymbol{e}_x$ (the unit basis vector parallel to $x$-axis). In this work,  we will use the example of $F_{\text{ext}}(\Q, \Q_\infty)$ from \cite{genkin2017defects} for $d=2$, given by
\begin{equation}\label{2d_f_ext}
F_{\text{ext}}(\Q, \Q_{\infty}) = -\zeta \Q R_{\pi/2} \text{Tr}[\Q\Q_{\infty}R_{\pi/2}],
\end{equation}
where $\zeta\geq0$ is the aligning parameter and $R_{\pi/2}$ is the matrix of counterclockwise rotation by $\pi/2$. For $d=3$, the formula for $F_{\text{ext}}(\Q, \Q_\infty)$ is  
\begin{equation}\label{3d_f_ext}
F_{\text{ext}}(\Q, \Q_{\infty}) =\zeta \left( \text{tr}(\Q^2) \Q_\infty - \text{tr}(\Q\Q_{\infty}) \Q \right)
\end{equation}
We note that if one considers dynamics $\dot{\Q}=F_{\text{ext}}(\Q, \Q_{\infty})$ then the Euclidean norm of $\Q$ is preserved, {\it i.e.} $\text{tr}(\Q^2)\equiv \text{const}$, and $\Q(t)$ converges to a multiple of $\Q_{\infty}$ as $t$ increases, so that $\Q\cdot \Q_{\infty}=\text{tr}(\Q\Q_{\infty})>0$. One can also show that \eqref{3d_f_ext} is equivalent to \eqref{2d_f_ext} in the case of two-dimensional $\Q$ and $\Q_\infty$ (with zero third row and column).
 
We impose anchoring boundary condition \eqref{eq:time_lc_6} on $\Q$ along the squirmer surface $\partial \mathcal{P}(t)$ which forces $\Q$ to be close to a given tensor $\Q_{\text{pref}}=q_\infty(\boldsymbol{n}_{\text{pref}}\otimes\boldsymbol{n}_{\text{pref}}-\frac{{\mathbb I}}{d})$. Here, $\boldsymbol{n}_{\text{pref}}=\boldsymbol{\nu}$ in the case of homeotropic anchoring when the surface orients liquid crystal molecules perpendicular to it or equivalently, parallel to the normal vector $\boldsymbol{\nu}$. On the other hand,
$\boldsymbol{n}_{\text{pref}}=\boldsymbol{\tau}$ in the case of the planar anchoring when molecules are aligned with the tangential vector field $\boldsymbol{\tau}$. The boundary condition \eqref{eq:time_lc_6} indeed penalizes the difference $\Q_{\text{pref}}-\Q$ in the sense that if we drop all terms in \eqref{eq:time_lc_4} except $\Gamma \left( K\Delta \Q + \hat{\H}(\Q) \right)$, then the solution $\Q$ to this truncated version of \eqref{eq:time_lc_4} with boundary condition \eqref{eq:time_lc_6} minimizes
the energy 
\begin{equation}\label{eq:as_if_energy}
\mathcal{E}_{\text{LdG}}(\Q)+W\int\limits_{\partial\mathcal{P}(t)}|\Q_{\text{pref}}-\Q|^2\,\text{d}S_\x.
\end{equation}
The coefficient $W$ in front of the penalization term in the energy functional \eqref{eq:as_if_energy} and also the right-hand side of 
\eqref{eq:time_lc_6} measures the anchoring strength. Mathematically,
depending on if $W\to \infty$ or $0$, \eqref{eq:time_lc_6} reduces to 
Dirichlet or Neumann boundary condition for $\Q$.

To determine the trajectory of the squirmer, that is, its velocity $\boldsymbol{V}(t)$ and angular velocity $\OMEGA(t)$, we consider force and torque balances for the squirmer: 
\begin{eqnarray}
&&m\frac{\text{d}\V}{\text{d}t} = \int\limits_{\partial \P(t)} \sigma\boldsymbol{\nu}\, \text{d}S_\x, \label{eq:time_ft_1}\\
&&\frac{\text{d}\left[I(t)\boldsymbol{\omega}\right]}{\text{d}t} = \int\limits_{\partial \P(t)} \x \times \sigma\boldsymbol{\nu} +\boldsymbol{\ell} \,\text{d}S_\x. \label{eq:time_ft_2}
\end{eqnarray} 
 Here $\sigma=\sigma_{\text{hydro}}+\sigma_{\text{ela}}$ is the total stress whereas $m$ and $I(t)=\left\{I_{ij}\right\}_{i,j=1}^{d}$ are the mass and inertia tensor of the squirmer, defined via
 \begin{eqnarray*}
 m&=& \rho_{\mathcal{P}}|\mathcal{P}(t)|,\\
 I_{ij}(t)&=&\rho_{\mathcal{P}}\int\limits_{\mathcal{P}(t)}\left[\boldsymbol{e}_i\times \x \right]\cdot \left[\boldsymbol{e}_j\times \x\right]\,\text{d}\x.
 \end{eqnarray*}
Here $\rho_{\mathcal{P}}$ is the squirmer's density. The additional torque $\boldsymbol{\ell}$ comes from the internal structure of the liquid crystal, namely, from that there is a preferred direction. It translates into the non-zero asymmetric part of the stress tensor $\sigma$. The formula for this additional torque is \cite{chi2020surface}
\begin{equation}
\boldsymbol{\ell}=\boldsymbol{\mu}\boldsymbol{\nu}, \text{ where }\boldsymbol{\mu}=(\mu_{ij})_{i,j=1}^{d} \text{ and }
\mu_{ij}=-2K\sum\limits_{m,l,k=1}^{d}\epsilon_{ilk}Q_{lm}Q_{mk,j}.\label{def_of_ell}
\end{equation} 
Here $\epsilon_{ilk}$ is the Levi-Civita symbol. Finally,  we note that the orientation $\boldsymbol{\alpha}(t)$ and the angular velocity $\OMEGA$ are related via 
\begin{equation}
\dot{\boldsymbol{\alpha}}=\OMEGA\times \boldsymbol{\alpha}.\label{omega_vs_alpha}
\end{equation}
\begin{remark} 
Note that the term $\ell$ admits a simplified form:
\begin{eqnarray}
\int\limits_{\partial \mathcal{P}}\ell_i\,\text{d}S_x&=&-2K\sum\limits_{l,m,k=1}^{d}\int_{\partial \mathcal{P}}\epsilon_{ilk}Q_{lm}Q_{mk,j}\nu_j\,\text{d}S_x\nonumber\\
&=&-2W\sum\limits_{l,m,k=1}^{d}\int_{\partial \mathcal{P}}\epsilon_{ilk}Q_{lm}(Q_{\text{pref},mk}-Q_{mk})\,\text{d}S_x.\nonumber
\end{eqnarray}
Here we used boundary conditions \eqref{eq:time_lc_6}. Next, for any symmetric matrix $\boldsymbol{B}=(B_{ij})_{i,j=1}^{d}$ we have 
\begin{equation}
\sum\limits_{l,m,k}\epsilon_{ilk}B_{lm}B_{mk}=0.\label{prop_levi-civita}
\end{equation}
Indeed, from properties of the Levi-Civita symbol we have 
\begin{equation*}
\sum\limits_{l,m,k}\epsilon_{ilk}B_{lm}B_{mk}=-\sum\limits_{l,m,k}\epsilon_{ilk}B_{km}B_{ml}.
\end{equation*}
On the other hand, due to symmetry of $\boldsymbol{B}$ we have  
\begin{equation*}
\sum\limits_{l,m,k}\epsilon_{ilk}B_{lm}B_{mk}=\sum\limits_{l,m,k}\epsilon_{ilk}B_{km}B_{ml}.
\end{equation*}
Thus, we have \eqref{prop_levi-civita}, from which we have the simplified form expression (simplified because it is linear in $\Q$ as opposed to \eqref{def_of_ell} which is quadratic in $\Q$):
\begin{eqnarray}
\int\limits_{\partial \mathcal{P}}\ell_i\,\text{d}S_x&=&-2W\sum\limits_{l,m,k=1}^{d}\int\limits_{\partial \mathcal{P}}\epsilon_{ilk}Q_{lm}Q_{\text{pref},mk}\,\text{d}S_x.\label{simplified_ell}
\end{eqnarray}
\end{remark}

\begin{remark}\label{remark:dissipation}
	We end the introduction of the time-dependent problem with the energy identity satisfied by solutions of this problem.   
First, consider the energy functional: 
	\begin{eqnarray}
	\mathcal{E}(t)
	&=&\dfrac{m\boldsymbol{V}^2}{2}+\dfrac{I\boldsymbol{\OMEGA}\cdot\OMEGA}{2}+\dfrac{\rho}{2}\int\limits_{\Omega(t)}|\boldsymbol{u}+\boldsymbol{V}|^2\,\text{d}\x\nonumber \\
	&&+\int\limits_{\Omega(t)}\hat{\mathcal{F}}(\Q)+\frac{K}{2}|\nabla \Q|^2\,\text{d}\x+
	\dfrac{ W}{2}\int\limits_{\partial\mathcal{P}(t)}|\Q_{\mathrm{pref}}-\Q|^2\,\text{d}S_\x.
	\end{eqnarray} 
	Note that in the absence of the squirmer $\mathcal{P}(t)=\emptyset$ 
(or when the squirmer is passive, {\it i.e.,} $u_{\mathrm{sq}}=0$) 
and if the external field $F_{\mathrm{ext}}$ equals zero,
then the system is dissipative, that is, the energy is non-increasing:
	\begin{equation}
	\dfrac{\text{d}}{\text{d}t}\mathcal{E}(t)=-\mathcal{D}(t)\leq 0, \quad \text{where }\mathcal{D}(t):= \eta \int\limits_{\Omega(t)}|\nabla \u|^2 \,\text{d}\x+\Gamma\int\limits_{\Omega(t)}|\H|^2\,\text{d}\x.
	\end{equation} 
	On the other hand, when the system experiences the energy input from the self-propulsion mechanism and external field $F_{\mathrm{ext}}$, 
the energy identity takes the following form:
	\begin{equation}
	\dfrac{\text{d}}{\text{d}t}\mathcal{E}(t) =-\mathcal{D}(t)+\int\limits_{\partial \mathcal{P}(t)}\sigma\boldsymbol{\nu}\cdot u_{\text{sq}}\boldsymbol{\tau}\,\text{d}S_\x+\int\limits_{\Omega(t)}\H:F_{\mathrm{ext}}\,\text{d}\x.
	\end{equation}
	Note that boundary integral $\int_{\partial \mathcal{P}(t)}\sigma\boldsymbol{\nu}\cdot u_{\text{sq}}\boldsymbol{\tau}\,\text{d}S_\x$ 
contains nonlinear terms in $\Q$ which in turn depends on the higher order
regularity property of $\Q$.
This causes difficulty in the analysis of the time-dependent problem.
Hence in this paper, we will only present a short time existence result and
leave the long time behavior to future work.

\end{remark}

\subsection{Steady state PDE system}

In this work, we are also interested in the steady translational motion of the squirmer in the liquid crystal. In the context of the model \eqref{eq:time_lc_1}-\eqref{eq:time_lc_6} \& \eqref{eq:time_ft_1}-\eqref{eq:time_ft_2}, the steady motion is described by the stationary solution of this system:  
\begin{eqnarray}
&\rho\u\cdot \nabla\u  = \nabla \cdot (\sigma_{\text{hydro}}+\sigma_{\text{ela}}) ,\text{ in } \Pi \setminus \P_{\text{st}}, &\label{eq:stst_lc_1}\\
&\nabla\cdot \u = 0,\text{ in }\Pi \setminus \P_{\text{st}},&\label{eq:stst_lc_2}\\
&\u = u_{\text{sq}}(\boldsymbol{\alpha}_{\text{st}},\x)\boldsymbol{\tau},\text{ on }\partial \P_{\text{st}},&\label{eq:stst_lc_3}\\
&(\u\cdot \nabla) \Q -S(\nabla \u, \Q)= \Gamma \left( K\Delta \Q + \hat{\H}(\Q) \right) + F_{\text{ext}}(\Q, \Q_\infty),\text{ in }\Pi \setminus \P_{\text{st}},&\label{eq:stst_lc_4}\\
&\Q, \u, \nabla p \quad \text{periodic in} \, \Pi \label{eq:stst_lc_5}\\
&K\partial_\nu \Q =W(\Q_{\text{pref}}-\Q) \text{ on }\partial \P_{\text{st}}.&\label{eq:stst_lc_6}
\end{eqnarray}
Here, we assume that the squirmer moves with the velocity $V_{\text{st}}$ 
with the orientation angle $\boldsymbol{\alpha}_{\text{st}}$, 
both of which are independent of time. As equations 
\eqref{eq:stst_lc_1}-\eqref{eq:stst_lc_6} are written in the squirmer's frame,
the domain $\mathcal{P}_{\text{st}}$ occupied by the squirmer will then 
be stationary. Similar to the time dependent case, we use 
$\Omega=\Pi\setminus \mathcal{P}_{\text{st}}$ to denote the fluid region in the
steady state case. 

In this setting, the force and torque balances \eqref{eq:time_ft_1}, 
\eqref{eq:time_ft_2} become
\begin{eqnarray}
&&\boldsymbol{0}=\int\limits_{\partial \P_{\text{st}}} \sigma\boldsymbol{\nu}\, \text{d}S_\x,\label{force-balance-st}\\
&&\boldsymbol{0} = \int\limits_{\partial \P_{\text{st}}} \x \times \sigma\boldsymbol{\nu} +\boldsymbol{\ell} \,\text{d}S_\x. \label{torque-balance-st}
\end{eqnarray} 
The force balance \eqref{force-balance-st}, in view of periodic boundary conditions for $\u$ and $\Q$ together with $\u\cdot \boldsymbol{\nu}=0$ on $\partial P_{\text{st}}$ (follows from \eqref{eq:stst_lc_3}), leads to
\begin{eqnarray}
0=\int\limits_{\partial \P_{\text{st}}} \sigma\boldsymbol{\nu}\, \text{d}S_\x=-\int\limits_{\partial \Pi}\sigma\boldsymbol{\nu}\, \text{d}S_\x +\rho\int\limits_{\Omega}\u\cdot \nabla \u \,\text{d}x=-\int\limits_{\partial \Pi} p\boldsymbol{\nu}\,\text{d}S_x.\label{eq:force-balance-becomes-periodic-pressure}
\end{eqnarray}
Therefore, since $\nabla p$ is periodic, as imposed in \eqref{eq:stst_lc_5}, it follows from \eqref{eq:force-balance-becomes-periodic-pressure} that $p(\x)$ is periodic in $\Pi$. Indeed, the fact that $\nabla p$ is periodic implies that   
\begin{equation}
p (\x)= \boldsymbol{m} \cdot \x+ p^{\text{per}}(\x),  \label{eq:lin+per}
\end{equation}
where $p^{\text{per}}(\x)$ is a function which is periodic in $\Pi$ and $\boldsymbol{m}\in \mathbb R^{d}$. Substitution of \eqref{eq:lin+per} into \eqref{eq:force-balance-becomes-periodic-pressure} implies that $\boldsymbol{m}=\boldsymbol{0}$ and $p(\x)=p^{\text{per}}(\x)$.
In this case, the force balance \eqref{force-balance-st} is satisfied regardless of squirmer's velocity $V_{\text{st}}$.  

We note that if an external force $\textbf{F}^{(e)}=\left\{F^{(e)}_i\right\}_{i=1}^{d}$ is applied on the squirmer, then the force balance in stationary case becomes  
\begin{equation}
\int\limits_{\partial \P_{\text{st}}} \sigma\boldsymbol{\nu}\, \text{d}S_\x+\textbf{F}^{(e)}=0
\end{equation}
which due to the same arguments as in derivation of \eqref{eq:force-balance-becomes-periodic-pressure} is equivalent to 
\begin{equation}
-\int\limits_{\partial \Pi} p\boldsymbol{\nu}\, \text{d}S_\x+\textbf{F}^{(e)}=0.
\end{equation}
Using \eqref{eq:lin+per} and the divergence theorem for the first term in the equation above we get 
\begin{equation}
|\Pi|\boldsymbol{m}=\textbf{F}^{(e)}.
\end{equation}
Therefore, an external force results in the pressure difference
\begin{equation}
F^{(e)}_i=\dfrac{L}{2}[p]_i, \, i=1,...,d, \quad \text{where }[p]_i=p|_{x_i=L}-p|_{x_i=-L}.\label{pressure-difference}
\end{equation}
In terms of the force balance, the periodic problem \eqref{eq:stst_lc_1}-\eqref{eq:stst_lc_6} is in contrast with the analogous problem in the exterior domain 
$\Pi=\mathbb R^{d}$. Namely, for the latter, we need to impose additional 
boundary conditions at $x =\infty$: 
$\u=-\boldsymbol{V}_{\text{st}}$ and $\Q=\Q_{\infty}$. Then
we would have obtained a Stokes-law-like force-velocity relation 
instead of the force-pressure relation \eqref{pressure-difference}. 

In this work, the squirmer swims due to self-propulsion only, without an external force, $\textbf{F}^{(e)}=\boldsymbol{0}$. Thus, we impose periodicity for the
 pressure $p$. Taking this into account, we define a weak solution of \eqref{eq:stst_lc_1}-\eqref{eq:stst_lc_6} as a couple $(\u,\Q)\in H^1_{\text{per}}(\Omega;\mathbb R^{d})\times H^2_{\text{per}}(\Omega;\mathbb R^{d\times d})$ such that equations \eqref{eq:stst_lc_2}, \eqref{eq:stst_lc_3} as well as the following two equalities hold for all $\boldsymbol{\psi}\in H^1_{\text{per}}(\Omega;\mathbb R^d)\cap\left\{\boldsymbol{\psi}|_{\partial\mathcal{P}_{\text{st}}}=0 \text{ and }\nabla \cdot \boldsymbol{\psi}=0\right\}$ and $\boldsymbol{\Phi}\in H^1_{\text{per}}(\Omega;\mathbb R^{d\times d})$ and every integral term is finite : 
\begin{eqnarray}
&&\eta \int\limits_{\Omega}\nabla \u:\nabla \boldsymbol{\psi}\,\text{d}x + \rho\int\limits_{\Omega}(\u\cdot \nabla)\u\cdot \boldsymbol{\psi}\,\text{d}x+\int\limits_{\Omega} \sigma_{\text{ela}}: \nabla \boldsymbol{\psi}\,\text{d}x=0.\label{weak_def_u}
\end{eqnarray}
\begin{eqnarray}
&&\Gamma\left[-K\int\limits_{\Omega}\nabla \Q:\nabla \boldsymbol{\Phi} \,\text{d}x + W\int\limits_{\partial\mathcal{P}_{\text{st}}}(\Q_{\text{pref}}-\Q):\boldsymbol{\Phi}\,\text{d}S_x\right]\nonumber\\&&
\hspace{80pt}-
\int\limits_{\Omega}\left((\u\cdot \nabla)\Q-S(\nabla \u,\Q)\right):\boldsymbol{\Phi}\,\text{d}x+ \int_{\Omega} F_{\text{ext}}:\boldsymbol{\Phi} \,\text{d}x=0.\label{weak_def_Q}
\end{eqnarray}


\section{Main results}
\label{sec:main-result}
Here we present our three main results.

Our first main result is the existence of a weak solution of the steady state problem \eqref{eq:stst_lc_1}-\eqref{eq:stst_lc_6}.
For the sake of clarity, we restrict our attention to the case when the shape parameter $\xi$ is $0$. Under this simplification, we can represent $\sigma_{\text{ela}}$ as 
\begin{equation}
\sigma_{\text{ela}}=-K\nabla \Q \astrosun \nabla \Q +\sigma_{a}, \text{ where }\sigma_a(\Q,\H)=\Q\H-\H\Q=K(\Q\Delta\Q-\Delta \Q\Q)
\end{equation}
and the term $S$ given by \eqref{def_of_S} satisfies the following equality
\begin{equation}\label{property_of_S}
S(\nabla \u, \Q):\Q =\text{Tr}(S(\nabla \u,\Q)\Q)=0. 
\end{equation}    

We also impose
\begin{equation}\label{property_of_F_ext}
F_{\text{ext}}(\Q,\Q_{\infty}):\Q=0.
\end{equation}
This condition holds for our specific choices of  $F_{\text{ext}}(\Q,\Q_{\infty})$ given by \eqref{2d_f_ext} or \eqref{3d_f_ext}.

\begin{thm}\it
Suppose $\xi=0$.
	There is a constant $C>0$ independent of $K$, $W$, $\Q_{\text{pref}}$, $\Q_{\infty}$, $\eta$, $\rho$, $u_{\text{sq}}$, $\Gamma$, $\boldsymbol{\alpha}_{\text{st}}$ such that if  
	\begin{equation} 
	\eta>2C\rho\|u_{\mathrm{sq}}\|_{L^{\infty}(\Omega)}\text{ and }\Gamma>2C\left(\dfrac{1}{\sqrt{K}}\|u_{\mathrm{sq}}\|_{L^{\infty}(\Omega)}+\|u_{\mathrm{sq}}\|_{W^{1,\infty}(\Omega)}\right),\label{large_enough}
	\end{equation} 
	then there is a weak solution $(\u,\Q) \in (H^{1}(\Omega), H^{2}(\Omega))$ of \eqref{eq:stst_lc_1}-\eqref{eq:stst_lc_6}.
	\label{theorem_existence}
\end{thm}  

\begin{remark}	
	The condition \eqref{large_enough} holds when parameters $\eta$ and $\Gamma$ are sufficiently large enough, given all other parameters. The condition \eqref{large_enough} also holds when $u_{\mathrm{sq}}$ is sufficiently small which means that self-propulsion is small. In the limit $u_{\mathrm{sq}}\to 0$ we recover existence of steady state for a passive swimmer without a condition on parameters.   
\end{remark}

\begin{remark}
Theorem~\ref{theorem_existence} states the existence of a weak solution of \eqref{eq:stst_lc_1}-\eqref{eq:stst_lc_6} for all orientation angles $\boldsymbol{\alpha}_{\text{st}}$. As discussed in Section~\ref{sec:steady-state}, the force balance \eqref{force-balance-st} is satisfied since weak solutions of \eqref{eq:stst_lc_1}-\eqref{eq:stst_lc_6}  have periodic pressure $p$. To determine the steady orientation $\boldsymbol{\alpha}_{\text{st}}$, one needs to consider additionally the torque balance \eqref{torque-balance-st} which is satisfied for $\boldsymbol{\alpha}_{\text{st}}=\dfrac{k\pi}{2}$ ($k$ is an integer). We note that it follows from our numerical studies in \cite{chi2020surface} that a squirmer can swim steadily only if it is oriented parallel, $\boldsymbol{\alpha}_{\text{st}}=k\pi$, or perpendicularly, $\boldsymbol{\alpha}_{\text{st}}=(2k-1)\frac{\pi}{2}$, to the vector $\boldsymbol{e}_x$, the liquid crystal orientation in the absence of the squirmer.   
\end{remark}

\medskip 

Our second main result is the local-in-time existence for the time dependent problem \eqref{eq:time_lc_1}-\eqref{eq:time_lc_6} with \eqref{eq:time_ft_1} and \eqref{eq:time_ft_2}. Here, we simplify the system by considering a spherical squirmer $\mathcal{P}(t)$ in its own moving frame so that $\Omega$ and $\mathcal{P}$ are independent of time. Under this assumption, the torque balance equation can be simplified into
\begin{eqnarray}
&&I\frac{\text{d}\boldsymbol{\omega}}{\text{d}t} = \int\limits_{\partial \P} \x \times \sigma\boldsymbol{\nu}  + \boldsymbol{\ell}\,\text{d}S_\x, \label{eq:time_ft_2_rewrite}
\end{eqnarray}
where the rotating inertia $I(t) = I\mathbb{I}$ becomes also independent of time and isotropic. 

\begin{thm}\it
	Suppose that $(\u_{\mathrm{sq}}, \Q_{\mathrm{pref}}) \in H^{5/2}(\partial \P) \times H^{5/2}(\partial \P)$ and the initial data $(\u_0, \Q_0) \in H^{2}_{\sigma}(\Omega) \times H^{3}(\Omega)$, where $H_{\sigma}^2(\Omega)=H^2(\Omega)\cap\left\{\nabla \cdot \u=0 \right\}$. Then there exists $T > 0$ and a unique solution $(\u, \Q)$ to the system \eqref{eq:time_lc_1}-\eqref{eq:time_lc_6} with \eqref{eq:time_ft_1} and \eqref{eq:time_ft_2_rewrite} such that 
	\begin{eqnarray*}
	&&\u \in H^{1}(0, T;H^{2}_{\sigma}(\Omega)) \cap H^{2}(0, T;L^{2}_{\sigma}(\Omega)), \\
	&&\Q \in H^{1}(0, T;H^{3}(\Omega)) \cap H^{2}(0, T;H^{1}(\Omega)).
	\end{eqnarray*} 
\label{theorem-finite-time}
\end{thm}

\begin{remark}
We adapt techniques from \cite{AbeDolLiu2013} to prove this result in Section \ref{sec:time-dependent}. The main idea is to rewrite the problem in a suitable Banach space and then use the Banach's fixed point theorem. However, the difference from \cite{AbeDolLiu2013} is an additional difficulty coming from  presence of the squirmer which requires to consider inhomogeneous boundary conditions as well as force and torque balances \eqref{eq:time_ft_1} and \eqref{eq:time_ft_2_rewrite}. The terms in balance equations involve boundary integrals with derivatives in integrands. It led to that the spatial regularity of the solution couple $(\u,\Q)$ is higher than it is required by a weak solution of the PDE problem \eqref{eq:time_lc_1}-\eqref{eq:time_lc_6}.       	
\end{remark}

\medskip 

Our third main result is a formal homogenization limit in the system \eqref{eq:stst_lc_1}-\eqref{eq:stst_lc_6}. This result can be considered as the derivation of a simplified model describing motion of a colony with periodically distributed squirmers (e.g., bacterial colony) in the liquid crystal.  

Specifically, we introduce a small parameter 
$
\varepsilon:=\frac{L}{\delta_L},
$
where $L$ is the linear size of a periodic box containing a single squirmer and $\delta_L$ is the observation scale. Next, we consider the problem \eqref{eq:stst_lc_1}-\eqref{eq:stst_lc_6} where all the parameters are written in 
physical dimensions. Details of non-dimensionalization are relegated to 
Appendix~\ref{sec:homo_rescaling}.  After the non-dimensionalization, we consider the steady state problem \eqref{eq:stst_lc_1}-\eqref{eq:stst_lc_6} in a periodic box $\Pi_{\ve}=[-\varepsilon,\varepsilon]^d$. The squirmer occupies domain $\mathcal{P}_{\ve}$ whose linear size is $\sim \varepsilon$. Consider the domain $U$ which is $\mathbb R^{d}$ or a sub-domain of $\mathbb R^d$ composed of many periodic boxes $\Pi_{\ve}$ such that the linear size of $U$ is of the order $1$ 
with respect to $\varepsilon$. Then \eqref{eq:stst_lc_1}-\eqref{eq:stst_lc_6} becomes (see Appendix~\ref{sec:homo_rescaling} for details):
\begin{eqnarray}
&& \varepsilon\gamma\Delta\Q+\tilde{a}\, \Q-\tilde{c}\, \Q\text{Tr}(\Q^2)+S(\nabla \tilde{\u}, \Q) - \tilde{\u} \cdot \nabla \Q+\tilde{\zeta}\, \tilde{F}_\text{ext}={\bf G}(\x)\text{ in }\Omega_{\ve},\label{eq:pre_homo_Q_eq_re}\\
&& \partial_{\boldsymbol{\nu}} \Q = \tilde{W}(\Q_\text{pref} - \Q)  \text{ on } \partial \P_{\ve}, \label{eq:pre_homo_Q_bc_re}\\
&& \varepsilon\tilde{\rho}(\tilde\u\cdot\nabla)\tilde{\u}-\varepsilon\tilde{\eta}\Delta \tilde{\u} + \nabla\tilde{p} = \varepsilon^2 \kappa\nabla\cdot(\nabla {\Q} \astrosun \nabla {\Q} + {\Q}\Delta {\Q} - \Delta {\Q} {\Q})+{\bf F}(\x) \text{ in }\Omega_{\ve}, \label{eq:pre_homo_u_eq_re} \\
&& \tilde{\u} = \ve\tilde{u}_{\text{sq}}\boldsymbol{\tau} \text{ on } \partial \P_{\ve} \label{eq:pre_homo_u_bc_re}.
\end{eqnarray} 
Here $\Omega_{\ve} = \Pi_{\ve} \setminus \mathcal{P}_{\ve}$ and $\nabla \cdot \tilde{\u}=0$. 
Here for simplicity, we assume the corresponding force and torque balance \eqref{force-balance-st} and \eqref{torque-balance-st} are automatically satisfied. This physically means that the bacterial colony has reached the steady state when every squirmer swims along a stable direction. ${\bf G}$ and ${\bf F}$ are given  external fields, varying spatially at the scale $1$ (independent of $\ve$). Parameters $\gamma,\tilde{a},\tilde{c},\tilde{\zeta},\tilde{W},\tilde{\rho},\tilde{\eta},\kappa$ are explained in Appendix~\ref{sec:homo_rescaling}.   

Our contribution in this regard is the identification of the homogenized limit
$(\u^{(h)},\Q^{(h)})$ of $(\ve^{-1}\tilde{\u},\Q)$.
We relegate the presentation of the limiting equations as well as their 
derivation via formal two-scale asymptotic expansions to Section~\ref{sec:homo_expansion}. We comment here that $\Q^{(h)}$ solves an \emph{algebraic equation} \eqref{eq_for_Q_h_appendix}, whereas $\u^{(h)}$ admits the representation \eqref{u_represent_Darcy} similar to that in the \emph{Darcy's law}.

\section{Existence of steady state $-$ proof of Theorem~\ref{theorem_existence}}
\label{sec:steady-state}

In this section, we address solvability of steady state PDE system \eqref{eq:stst_lc_1}-\eqref{eq:stst_lc_6}.  
To this end, we first show in Subsection~\ref{subsec:max} that if a solution of the system \eqref{eq:stst_lc_1}-\eqref{eq:stst_lc_6} exists (with $\xi=0$), 
then it satisfies a maximum principle for $\|\Q\|$. Next, in Subsections \ref{subsec:galerkin-intro}, \ref{subsec:galerkin-energy-estimate}, \ref{subsec:galerkin-existence}, and \ref{subsec:galerkin-limit} we prove the existence for the system \eqref{eq:stst_lc_1}-\eqref{eq:stst_lc_6} where nonlinearities $\hat{\H}$ and $F_{\text{ext}}$ are truncated for large values of $\Q$. Finally, combination of the maximum principle and solvability of the truncated system implies the 
existence of a solution to the original system \eqref{eq:stst_lc_1}-\eqref{eq:stst_lc_6}. 

\subsection{$L^{\infty}$-bound on $\Q$} \label{subsec:max}
Here we adapt the strategy from \cite{GuiRod2015}. 
First, we introduce the number $q_*>0$ such that
\begin{equation}\label{def_of_q_star}
\Gamma\hat{\H}(\Q):\Q\leq 0 \text{ for all }|\Q|\geq q_*.
\end{equation}
Such a finite number $q_{*}$ exists since $F_{\text{ext}}(\Q,\Q_{\infty})$ is a quadratic polynomial of $\Q$ whereas $\hat{\H}$ is the third order polynomial 
with \textcolor{black}{a definite negative sign} in front the highest power.

\begin{lemma}\it
	Let $(\u,\Q)$ be a solution of \eqref{eq:stst_lc_1}-\eqref{eq:stst_lc_6}. 
	Then $\|\Q\|_{L^{\infty}} \leq\alpha$, where 
	\begin{equation}
	\alpha:= \max\left\{|\Q_{\mathrm{pref}}|,q_*\right\}.\label{def_of_alpha}
	\end{equation}
	\label{L_inf_Q}
\end{lemma} 
\begin{proof} 
	Recall the equation for $\Q$: 
	\begin{equation}
	(\u\cdot \nabla)\Q - S(\nabla \u,\Q)-K\Gamma \Delta \Q-\Gamma\hat{\H}(\Q)- F_{\text{ext}}(\Q, \Q_\infty)=0. 
	\end{equation}
By multiplying the above by $\Q$, taking the trace of the resulting expression
and using \eqref{property_of_S} and \eqref{property_of_F_ext}, we get
	\begin{equation*}
	\dfrac{1}{2}\u\cdot \nabla (|\Q|^2)-\dfrac{\Gamma K}{2}\left(\Delta (|\Q|^2) - 2 |\nabla \Q|^2\right)-\Gamma\hat{\H}(\Q):\Q=0. 
	\end{equation*}
	As $|\nabla \Q|^2$ is non-negative, we obtain the inequality:
	\begin{equation}\label{eq_for_Q_square}
	\u\cdot \nabla (|\Q|^2)-\Gamma K\Delta (|\Q|^2) -2\Gamma\hat{\H}(\Q):\Q\leq 0. 
	\end{equation}
	Now introduce $\psi(\Q):= (|\Q|^2- \alpha^2)_{+}$ ($\alpha$ is from \eqref{def_of_alpha}). Note that 
	\begin{equation*}
	\psi(\Q)\,\mathcal{D}(|\Q|^2)=\psi(\Q)\,\mathcal{D}(|\Q|^2-\alpha^2)=\psi(\Q)\,\mathcal{D}\psi(\Q),
	\end{equation*}
	where $\mathcal{D}$ is either $\Delta$ or $\nabla$.
	
	Next, we multiply \eqref{eq_for_Q_square} by $\psi(\Q)$ and integrate 
over $\Pi\setminus \mathcal{P}_{\text{st}}$. Then, we have
	\begin{eqnarray}\label{ineq_above}
	&&\dfrac{1}{2}\int\limits_{\partial \P} \u\cdot \nu \psi^2\,\text{d}s_x-\Gamma K \int\limits_{\partial \P}\dfrac{\partial\psi(\Q)}{\partial \nu}\psi(\Q)\,\text{d}s_x\nonumber
	\\
	&&\hspace{60pt}+\Gamma K\|\nabla \psi (\Q)\|^2_{L^2(\Omega)}-2\Gamma \int\limits_{\Omega}(\hat{\H}(\Q):\Q)\psi(\Q)\,\text{d}x\leq 0.  
	\end{eqnarray}
	The first term in the left hand side of the above inequality vanishes 
due to \eqref{eq:stst_lc_3} while the second term is negative:
	\begin{eqnarray*}
	K\int_{\partial \P} \dfrac{\partial\psi(\Q)}{\partial \nu}\psi(\Q)\,\text{d}s_x&=&2K\int_{\partial \P\cap\{|\Q|>\alpha\}} (\dfrac{\partial \Q}{\partial \nu}:\Q)(|\Q|^2-\alpha^2)\,\text{d}s_x\nonumber\\
	&=&2W\int_{\partial \P\cap\{|\Q|>\alpha\}} ((\Q_{\text{pref}}-\Q):\Q)(|\Q|^2-\alpha^2)\,\text{d}s_x\nonumber\\
	&=&2W\int_{\partial \P\cap\{|\Q|>\alpha\}} ((\Q_{\text{pref}}:\Q)-|\Q|^2)(|\Q|^2-\alpha^2)\,\text{d}s_x\nonumber\\
	&\leq&W\int_{\partial \P\cap\{|\Q|>\alpha\}} (|\Q_{\text{pref}}|^2-|\Q|^2)(|\Q|^2-\alpha^2)\,\text{d}s_x\nonumber \\
	&\leq & 0.
	\end{eqnarray*} 
	Hence, 
	\begin{equation*}
	K\|\nabla \psi (\Q)\|^2_{L^2(\Omega)}\leq 2 \int\limits_{\Omega}\left(\hat{\H}(\Q):\Q\right)\psi(\Q)\,\text{d}x. 
	\end{equation*}
	Next, by \eqref{def_of_q_star}, we have
	\begin{eqnarray*}
	K\|\nabla \psi (\Q)\|^2_{L^2(\Omega)}&\leq& 2 \int\limits_{\Omega}\left(\hat{\H}(\Q):\Q\right)\psi(\Q)\,\text{d}x\leq  0.
	\end{eqnarray*}
	so that $\|\nabla \psi (\Q)\|^2_{L^2(\Omega)}=0$. The Lemma is thus
proved.
\end{proof}

\subsection{Galerkin approximation for pair $(\u, \H)$}
\label{subsec:galerkin-intro}

We introduce here Galerkin approximations for the system 
\eqref{eq:stst_lc_1}-\eqref{eq:stst_lc_6}. For each $m\in \mathbb N$, 
we define: 
\begin{equation}
\u_m =\u_{\text{os}}+\hat{\u}_m=\u_{\text{os}}+\sum\limits_{k=1}^{m}u_{km}\Psi_{k}\text{ and }  \H_{m}=\sum\limits_{k=1}^{m}h_{km}\Phi_{k}.\label{expansions_H_idea}
\end{equation}
Note that the domain $\Pi$ is a bounded periodic box and both Laplacian and Stokes operators have a discrete spectrum implying existence of bases:
\begin{equation}
\text{$\left\{\Psi_{k}\left|\, \Psi_k|_{\partial \P_{\text{st}}}=0,\, \nabla \cdot \Psi_k=0,\,\Psi_k\,\text{is $\Pi$-periodic} \right.\right\}_{k=1}^{\infty}$ and $\left\{\Phi_{k}\left|\, \Phi_k\,\text{is $\Pi$-periodic} \right.\right\}_{k=1}^{\infty}$}
\end{equation}  
in $L^2_\sigma(\Omega;\mathbb R^{d})$ and $L^2(\Omega;\mathbb R^{d\times d})$, respectively. (Recall that $\Omega = \Pi \setminus \mathcal{P}_{\text{st}}$ and $L^2_\sigma$ means $L^2$-space with divergence-free condition.)

The function $\u_{\text{os}}$ above is an offset function used to 
to take care of non-zero boundary conditions for $\u$. It solves Stokes equation:
\begin{empheq}[left=\empheqlbrace]{align}
&\eta \Delta \u_{\text{os}}+\nabla p_{\text{os}} = 0, \text{ in }\Omega \\
& \nabla \cdot \u_{\text{os}}=0,  \text{ in }\Omega,\\ 
&\u_{\text{os}}=u_{\text{sq}}(\boldsymbol{\alpha}_{\text{st}},\x)\boldsymbol{\tau}  \text{ on }\partial \mathcal{P}_{\text{st}},\label{eq:os_bc}\\
&\u_{\text{os}}, p_{\text{os}} \, \text{periodic in} \, \Pi.
\end{empheq}
Anticipating that $\u=\u_{\text{os}} + \hat{\u}$, from \eqref{eq:os_bc} and \eqref{eq:stst_lc_3},
we have 
\begin{equation}
\hat{u}=0\text{ on }\partial\mathcal{P}_{\text{st}}.
\label{bc_for_hat_u} 
\end{equation}

To continue, for an appropriately large constant $M>0$,
we introduce a truncated potential $\hat{\mathcal{F}}_M\geq 0$ as follows:
\begin{equation}
\hat{\mathcal{F}}_{M}(\Q)=\left\{
\begin{array}{ll}
\hat{\mathcal{F}}(\Q), & \text{for $\|\Q\|\leq M$}\\
\left.\dfrac{\partial\hat{\mathcal{F}}}{\partial\|\Q\|}\right|_{\|\Q\|=M}\|\Q\|
+
\left(\left.\hat{\mathcal{F}}\right|_{\|\Q\|=M}-\left.\dfrac{\partial\hat{\mathcal{F}}}{\partial\|\Q\|}\right|_{\|\Q\|=M}M\right)
&\text{for }\|\Q\|>M
\end{array}
\right.
.
\end{equation}
The functional derivative of $\hat{\mathcal{F}}_{M}$ is given by
\begin{equation}
\hat{\H}_M(\Q)=-\nabla_{\Q}\hat{\mathcal{F}}_M=\left\{
\begin{array}{lr}
\hat{\H}(\Q),&\|\Q\|\leq M,\\
\gamma_M\dfrac{\Q}{\|\Q\|},&\|\Q\|>M,
\end{array}\right.
\quad
\text{where $\gamma_M = \left.\dfrac{\partial\hat{\mathcal{F}}}{\partial\|\Q\|}\right|_{\|\Q\|=M}$.}
\label{def_of_h_m}
\end{equation}
We have the following bound on $\hat{H}_M(\Q)$:
\begin{equation}
\|\hat{H}_M(\Q)\|_{L^{\infty}(\Omega)}\leq \Gamma_M, \quad \text{where }\Gamma_M=\max\{\|\hat{H}\|_{L^{\infty}(B_M(0))},\gamma_M\}.
\label{max_bound_on_H_M}
\end{equation}

We now define the function $\Q_m$, corresponding to the Galerkin approximation 
$\H_m$ as the solution to the following system:
\begin{equation}
\left\{
\begin{array}{l}
K\Delta \Q_m+\hat{\H}_M(\Q_m)=\H_m, \text{ in }\Omega\\
K\partial_\nu \Q_m= W (\Q_{\text{pref}}-\Q_m)\text{ on }\partial \mathcal{P}_{\text{st}},\\
\quad
\Q_m \, \text{periodic in} \, \Pi
\end{array}
\right.\label{aux_elliptic}
\end{equation}
Below, we will need {\it a priori} estimates for the solution to the problem \eqref{aux_elliptic}, formulated in the following lemma. Its proof is given in Appendix~\ref{sec:aux_estimate}.
\begin{lemma}
	\label{lemma_existence_of_q_m}
{\it Let $\H_m\in L^2(\Omega)$. Then there exists a solution $\Q_m$ for   \eqref{aux_elliptic}. Moreover, there exists a constant $C>0$ such that 
\begin{eqnarray}
&&\sqrt{K}\|\nabla \Q_m\|_{L^2(\Omega)}+\|\Q_m\|_{L^2(\Omega)}+\sqrt{W}\|\Q_m\|_{L^2(\partial\mathcal{P}_{\mathrm{st}})} \nonumber\\&& 
\hspace{150pt}\leq C\left(\|\H_m\|_{L^2(\Omega)}+\sqrt{W}\|\Q_{\mathrm{pref}}\|_{L^2(\partial\mathcal{P}_{\mathrm{st}})}+1\right) \label{lemma32_H1} \\
&&\|\Q\|_{H^2(\Omega)}\leq C \left(\gamma_1\|\H_m\|_{L^2(\Omega)}+\gamma_2\|\Q_{\mathrm{pref}}\|_{C^1}+\gamma_3\Gamma_M\right), \label{lemma32_H2}
\end{eqnarray}
where 
\begin{equation}
\gamma_1=\gamma_3=\dfrac{W+K}{W}\quad \text{and}\quad \gamma_2=\sqrt{\dfrac{W+K}{K}}.\nonumber
\end{equation}
}
\end{lemma}

With the above, then the Galerkin approximations $(\u_m,\H_m)$ are defined so as to satisfy 
for each $k=1,...,m$, the following conditions: 
follows:
\begin{eqnarray}
&&\eta \int\limits_{\Omega}\nabla \hat{\u}_m:\nabla \Psi_k\,\text{d}x  +\eta \int\limits_{\Omega}\nabla \u_{\text{os}}:\nabla \Psi_k\,\text{d}x+ \rho\int\limits_{\Omega}(\u_{m}\cdot \nabla)\u_m\cdot \Psi_k\,\text{d}x\nonumber \\&& \hspace{120 pt}-\int\limits_{\Omega} \left(K\nabla \Q_m\astrosun \nabla \Q_m-\sigma_a(\Q_m,\H_m)\right): \nabla \Psi_{k}\,\text{d}x=0,\label{weak_galerkin_u}
\end{eqnarray}
\begin{eqnarray}
&&\Gamma\int\limits_{\Omega}\H_{m}:\Phi_k \,\text{d}x -
\int\limits_{\Omega}\left((\u_m\cdot \nabla)\Q_{m}-S(\nabla \u_m,\Q_m)\right):\Phi_k\,\text{d}x\nonumber\\
&&\hspace{220pt}+ \int_{\Omega} F_{\text{ext},M} : \Phi_{k} \,\text{d}x=0.\label{weak_galerkin_Q}
\end{eqnarray}
Here, $F_{\text{ext},M}$ is defined as a continuous function such that: 
\begin{equation}
F_{\text{ext},M}=\left\{
\begin{array}{rl}
F_{\text{ext}},& \|F_{\text{ext}}\|\leq M, \\ & \\
M\dfrac{F_{\text{ext}}}{\|F_{\text{ext}}\|},& \|F_{\text{ext}}\|>M. 
\end{array}
\right.\label{def_of_f_m}
\end{equation} 

Next, we will prove the existence and apriori estimates for $(\u_m, \H_m$).

\subsection{Energy estimate for Galerkin approximations}
\label{subsec:galerkin-energy-estimate}

\begin{lemma} \it
	Provided that $\eta$ and $\Gamma$ are large enough, there exists a constant $C$ independent of $m$ such that 
	\begin{equation}
	\|\nabla \hat{\u}_m\|_{L^2(\Omega)}^2+\|\H_m\|_{L^2(\Omega)}^2 < C. \label{lemma33_ineq}
	\end{equation}
	\label{lemma:apriori-estimate}
\end{lemma}

\begin{proof} 
	Using test function $\hat{\u}_m$ and $\H_m$ instead of $\Psi_{k}$ and $\Phi_{k}$ in \eqref{weak_galerkin_u}-\eqref{weak_galerkin_Q} and taking the sum of two equalities, we obtain the following energy equality:
	\begin{eqnarray}
	&&\eta \int\limits_{\Omega}|\nabla \hat{\u}_m|^2\,\text{d}x+\Gamma\int\limits_{\Omega}|\H_m|^2\,\text{d}x\nonumber\\
	&& \hspace{40 pt}= - \rho\int\limits_{\Omega}(\u_{m}\cdot \nabla)\u_m\cdot \hat{\u}_m\,\text{d}x-\eta \int\limits_{\Omega}\nabla \u_{\text{os}}:\nabla \hat{\u}_m\,\text{d}x\nonumber\\
	&& \hspace{50 pt}-\int\limits_{\Omega}\left(\sigma_a(\Q_m,\H_m): \nabla \hat{\u}_m+S(\nabla \hat{\u}_m,\Q_m):\H_m\right)\,\text{d}x\nonumber
	\\
	&& \hspace{50 pt}+K\int\limits_{\Omega}\left(\nabla \Q_m\astrosun \nabla \Q_m:\nabla \hat{\u}_m +  (\hat{\u}_m\cdot \nabla)\Q_{m}:\Delta {\Q}_m\right) \,\text{d}x\nonumber\\
	&& \hspace{50 pt} +K\int\limits_{\Omega} (\u_{\text{os}}\cdot \nabla)\Q_{m}:\Delta {\Q}_m \,\text{d}x  +\int\limits_{\Omega}(\u_m\cdot \nabla)\Q_m:\hat{\H}_M(\Q_m)\,\text{d}x\nonumber\\
	&& \hspace{50 pt} - \int\limits_{\Omega} S(\nabla \u_\text{os},\Q_m):\H_m \text{d}x - \int\limits_{\Omega} F_{\text{ext},M}:\H_m \,\text{d}x.\label{raw_energy_est}
	\end{eqnarray}
	Next, we estimate each integral. Below, $C$ denotes a generic constant independent of $K,W,\Q_{\text{pref}},\eta,\rho,\u_{\text{os}},\Gamma,m$ which may change from line to line whereas $C^{*}$ is a generic constant which is independent of $m$ only and may also change from line to line.   
	
	\medskip 
	
	\noindent{\bf 1.} We use the representation $\u_m=\u_{\mathrm{os}} + \hat{\u}_m$ 
to write
	\begin{eqnarray}
	&&-\rho\int\limits_{\Omega}(\u_{m}\cdot \nabla)\u_m\cdot \hat{\u}_m\,\text{d}x=-\rho\int\limits_{\Omega}({\u}_m\cdot \nabla)\hat{\u}_m\cdot \hat{\u}_m\,\text{d}x-\rho\int\limits_{\Omega}({\u}_m\cdot \nabla){\u}_{\text{os}}\cdot \hat{\u}_m\,\text{d}x.\quad\,\,\label{b1_num_1}
	\end{eqnarray}
	Next, using integration by parts, the non-penetration condition $\hat{\u}_m\cdot \boldsymbol{\nu}=0$ on $\partial \mathcal{P}_{\mathrm{st}}$ and the divergence-free condition 
$\nabla \cdot \u_m=0$, we get:
	\begin{equation}
	\int\limits_{\Omega}({\u}_m\cdot \nabla)\hat{\u}_m\cdot \hat{\u}_m\,\text{d}x=\dfrac{1}{2}\int\limits_{\Omega}{\u}_m\cdot \nabla |\hat{\u}_m|^2\,\text{d}x = 0.
	\end{equation}
	To estimate the second integral in the right-hand side of \eqref{b1_num_1} we integrate by parts and use the non-penetration condition again on 
$\partial \mathcal{P}_{\mathrm{st}}$ to we get:	
	\begin{eqnarray}
	-\rho\int\limits_{\Omega} (\u_m\cdot \nabla)\u_{\text{os}}\cdot \hat{\u}_m\,\text{d}x&=&\rho\int\limits_{\Omega}\u_m \cdot \nabla \hat{\u}_m\u_{\text{os}}\,\text{d}x.
	\end{eqnarray}
	Finally, we use the Poincar\'e estimate for $\hat{\u}_m$ (one can also use \eqref{poincare_without_coefs} with $\hat{\u}_m$ instead of $\Q$) as well as that the offset function $\u_{\text{os}}$ is a smooth function with bounded derivatives: 
	\begin{eqnarray}
	&&-\rho\int\limits_{\Omega}(\u_{m}\cdot \nabla)\u_m\cdot \hat{\u}_m\,\text{d}x=\rho\int\limits_{\Omega}\hat{\u}_m \cdot \nabla \hat{\u}_m\u_{\text{os}}\,\text{d}x+\rho\int\limits_{\Omega}\u_{\text{os}} \cdot \nabla \hat{\u}_m\u_{\text{os}}\,\text{d}x\nonumber\\
	&&\hspace{40pt}\leq \rho \int\limits_{\Omega}|\nabla \hat{\u}_{m}|(|\hat{\u}_m||\u_{\text{os}}|+|\u_{\text{os}}|^2)\,\text{d}x\nonumber \\
	&&\hspace{40pt}\leq C\rho\|{\u}_{\text{os}}\|_{L^\infty(\Omega)} \left(\int\limits_{\Omega}|\nabla \hat{\u}_m|^2\,\text{d}x\right)+C^{*}. \label{a_priori_bp1} 
	\end{eqnarray}

\smallskip 

	\noindent {\bf 2.} Here we bound the second integral in the right-hand side of \eqref{raw_energy_est} by the Cauchy-Schwarz inequality:
	\begin{eqnarray}
	-\eta \int\limits_{\Omega}\nabla \u_{\text{os}}:\nabla \hat{\u}_m\,\text{d}x\leq \dfrac{\eta}{4}\int\limits_{\Omega}|\nabla \hat{\u}_m|^2\,\text{d}x+C^{*}. \label{a_priori_bp2}
	\end{eqnarray}

	\noindent{\bf 3.} We have the equality $
	\sigma_a (A,B):D+S(D,A):B=0$ which holds for all matrices $A$, $B$, and $D$ such that $A$ and $B$ are symmetric: 
	\begin{equation}
	\int\limits_{\Omega}\left(\sigma_a(\Q_m,\H_m): \nabla \hat{\u}_m+S(\nabla \hat{\u}_m,\Q_m):\H_m\right)\,\text{d}x=0. \label{a_priori_bp3}
	\end{equation}

	\smallskip 
	
	\noindent{\bf 4.} Note that the integral in the 4th line of \eqref{raw_energy_est} vanishes. Indeed, using integration by parts, $\nabla \cdot \hat{\u}_m=0$ and \eqref{bc_for_hat_u}, we get 
	\begin{eqnarray}
	&&\int\limits_{\Omega}\nabla \Q_m\astrosun \nabla \Q_m:\nabla \hat{\u}_m \,\text{d}x\nonumber \nonumber\\&&\hspace{20 pt} = -\dfrac{1}{2}\int\limits_{\Omega}(\hat{\u}_m\cdot \nabla)|\nabla \Q_m|^2\,\text{d}x-\int\limits_{\Omega}(\hat{\u}_m\cdot \nabla)\Q_m:\Delta \Q_m \,\text{d}x +\int\limits_{\partial \mathcal{P}_{\mathrm{st}} }\left[\left(\nabla \Q_m\astrosun \nabla \Q_m\right) \nu\right]\cdot \hat{\u}_m\,\text{d}S_x\nonumber \\
	&&\hspace{30 pt}=-\dfrac{1}{2}\int\limits_{\partial\mathcal{P}_{\text{st}}}(\hat{\u}_m\cdot \nu)|\nabla \Q_m|^2\,\text{d}S_x+\dfrac{1}{2}\int\limits_{\Omega}(\nabla\cdot \hat{\u}_m)|\nabla\Q_m|^2\,\text{d}x-\int\limits_{\Omega}(\hat{\u}_m\cdot \nabla)\Q_m:\Delta \Q_m \,\text{d}x\nonumber\\
	&&\hspace{30pt}=-\int\limits_{\Omega}(\hat{\u}_m\cdot \nabla)\Q_m:\Delta \Q_m \,\text{d}x.\label{a_priori_bp4}
	\end{eqnarray}
	
	\smallskip 
	
	\noindent{\bf 5.} We use the Cauchy-Schwarz inequality, \eqref{bc_for_hat_u}, and the {\it a priori} bound \eqref{lemma32_H1} to estimate the 5th line of \eqref{raw_energy_est}:
	\begin{eqnarray}
	&&K\int\limits_{\Omega} (\u_{\text{os}}\cdot \nabla)\Q_{m}:\Delta {\Q}_m \,\text{d}x  +\int\limits_{\Omega}(\u_m\cdot \nabla)\Q_m:\hat{\H}_M(\Q_m)\,\text{d}x\nonumber\\
	&&\hspace{50pt}=\int\limits_{\Omega} (\u_{\text{os}}\cdot \nabla)\Q_{m}:\H_m \,\text{d}x+\int\limits_{\Omega}(\hat{\u}_m\cdot \nabla)\hat{\mathcal{F}}_M(\Q_m)\,\text{d}x\nonumber\\
	&&\hspace{55 pt}\leq C\|\u_{\text{os}}\|_{L^{\infty}(\Omega)}(\sqrt{K}\|\nabla \Q_m\|_{L^2(\Omega)}^2+\dfrac{1}{\sqrt{K}}\|\H_m\|^2_{L^2(\Omega)})+\int\limits_{\partial \P_{\text{st}}}(\hat{\u}_m\cdot\nu)\hat{\mathcal{F}}_M(\Q_m)\,\text{d}S_x \nonumber \\
	&& \hspace{55 pt} 
	\leq C\|\u_{\text{os}}\|_{L^{\infty}(\Omega)}(\sqrt{K}\|\nabla \Q_m\|_{L^2(\Omega)}^2+\dfrac{1}{\sqrt{K}}\|\H_m\|^2_{L^2(\Omega)})\nonumber \\
	&& \hspace{55 pt}\leq \dfrac{C}{\sqrt{K}}\|\u_{\text{os}}\|_{L^{\infty}(\Omega)}\|\H_m\|_{L^2(\Omega)}^2+C^{*}.\label{a_priori_bp5}
	\end{eqnarray}

\smallskip 
	
	\noindent{\bf 6.} We use again the {\it a priori} bound \eqref{lemma32_H1} and the Cauchy-Schwarz inequality to estimate the first term in the 6th line of \eqref{raw_energy_est}:
	\begin{eqnarray}
	\int_{\Omega} S(\nabla \u_\text{os},\Q_m):\H_m \text{d}x &<& C\|\u_{\text{os}}\|_{W^{1,\infty}(\Omega)}(\| \Q_m\|_{L^2(\Omega)}^2 + \|\H_m\|_{L^{2}(\Omega)}^2)
	\nonumber \\
	&<& C \|\u_{\text{os}}\|_{W^{1,\infty}(\Omega)} \|\H_m\|_{L^2(\Omega)}^2 + C^{*}.\label{a_priori_bp6}
	\end{eqnarray}
	
	\smallskip

	\noindent{\bf 7.} Finally, the last term in \eqref{raw_energy_est} is estimated as follows 
	\begin{eqnarray}
	\int_{\Omega} F_{\text{ext},M}:\H_{m} \text{d}x < \dfrac{\Gamma}{4}\|\H_m\|_{L^{2}(\Omega)}^2+C^{*}.\label{a_priori_bp7}
	\end{eqnarray}

	\medskip 
	Collect \eqref{a_priori_bp1}-\eqref{a_priori_bp7} and substitute them in \eqref{raw_energy_est}:
	\begin{eqnarray*}
	&&\left(\dfrac{3\eta}{4}-C\rho\|\u_{\text{os}}\|_{L^{\infty}(\Omega)}\right)\|\hat{\u}_{m}\|_{L^2(\Omega)}^2\\
	&&\hspace{60pt}+\left(\dfrac{3\Gamma}{4}-\dfrac{C}{\sqrt{K}}\|\u_{\text{os}}\|_{L^{\infty}(\Omega)}-C\|\u_{\text{os}}\|_{W^{1,\infty}(\Omega)}\right)\|\H_m\|_{L^2(\Omega)}^2\leq C^{*}.
	\end{eqnarray*}
	Under the restrictions \eqref{large_enough}, the inequality \eqref{lemma33_ineq} 
holds proving the Lemma.
\end{proof}

\subsection{Existence of Galerkin approximations}
\label{subsec:galerkin-existence}

We will use the following result \cite[Lemma IX.3.1, p. 597]{Galdi2011}: 
\begin{thm}\it 
	Let $\P:\mathbb R^{p}\to\mathbb R^{p}$ be a continuous mapping such that for some $R>0$:
	\begin{equation}
	\P(\boldsymbol{\xi})\cdot \boldsymbol{\xi} \geq 0 \quad \text{ for all }\boldsymbol{\xi} \in \mathbb R^p \text{ with }|\boldsymbol{\xi}|=R.\label{hedgehog}
	\end{equation} 	
	Then there exists $\boldsymbol{\xi}_0\in\mathbb R^p$ with $|\boldsymbol{\xi}_0|\leq R$ such that $\P(\boldsymbol{\xi}_0)=0$.
\end{thm}

Next, we introduce the mapping $\P$ for our problem. 
Given $m\geq 1$, let $$\boldsymbol{\xi}=(u_{1m},...,u_{mm},h_{1m},...,h_{mm})\in\mathbb R^{2m}$$ and the $k$th component mapping $\P:\mathbb R^{2m}\to \mathbb R^{2m}$ $(p=2m)$ is the left-hand side of \eqref{weak_galerkin_u} for $1\leq k\leq m$ and the left-hand side of \eqref{weak_galerkin_Q} for $m+1\leq k\leq 2m$. We obtain that 
\begin{equation*}
\P(\boldsymbol{\xi})\cdot \boldsymbol{\xi}=\eta\|\nabla \hat{\u}_m\|_{L^2(\Omega)}^2+\Gamma\|\H_m\|_{L^2(\Omega)}^2-\mathcal{R}(\u_m,\H_m),
\end{equation*}
where $\mathcal{R}(\u_m,\H_m)$ is the right-hand side of \eqref{raw_energy_est}. In the proof of Lemma~\ref{lemma:apriori-estimate} we showed that 
\begin{eqnarray*}
&&|\mathcal{R}(\u_m,\H_m)|\leq \left(\dfrac{\eta}{4}+C\rho\|\u_{\text{os}}\|_{L^{\infty}(\Omega)}\right)\|\hat{\u}_{m}\|_{L^2(\Omega)}^2\\
&&\hspace{80pt}+\left(\dfrac{\Gamma}{4}+\dfrac{C}{\sqrt{K}}\|\u_{\text{os}}\|_{L^{\infty}(\Omega)}+C\|\u_{\text{os}}\|_{W^{1,\infty}(\Omega)}\right)\|\H_m\|_{L^2(\Omega)}^2+ C^{*}.
\end{eqnarray*}
Therefore, using this inequality we obtain the following:
\begin{lemma}\it
	Assume $\eta$ and $\Gamma$ are large enough so Lemma \ref{lemma:apriori-estimate} holds. Then there exists constants $C_1,C_2>0$ independent of $m$ such that 
	\begin{equation*}
	\P(\boldsymbol{\xi})\cdot \boldsymbol{\xi}\geq C_1 \left(\|\nabla \hat{\u}_m\|_{L^2(\Omega)}^2+\|\H_m\|_{L^2(\Omega)}^2\right)-C_2. 
	\end{equation*}
\end{lemma}

The condition \eqref{hedgehog} is satisfied for large $R>0$ and thus we have the following
existence result for our Galerkin approximations:

\begin{thm}\it
	Assume $\eta$ and $\Gamma$ are the same as in Lemma \ref{lemma:apriori-estimate}. 
Then there exists a solution $(\hat{\u}_m,\H_m)$ of \eqref{weak_galerkin_u}-\eqref{weak_galerkin_Q}. Moreover, if $\Q_m$ is defined via \eqref{aux_elliptic}, then the solution satisfies
	\begin{equation}
	\|\hat{\u}_m\|_{H^1(\Omega)}^2+\|\Q_{m}\|_{H^2(\Omega)}^2+\hat{\mathcal{F}}_M(\Q_m)+\|\Q_{\mathrm{pref}}-\Q_m\|^2_{L^2(\partial \P_{\text{st}})}+\|\H_m\|_{L^2(\Omega)}^2 < C.\label{a-priori-bound-2}
	\end{equation}
\end{thm}

\subsection{Passing to limit $m\to \infty$}
\label{subsec:galerkin-limit}

From \eqref{a-priori-bound-2} we get that there is a subsequence of $\left\{(\hat{\u}_m,\H_m)\right\}$ such that 
\begin{eqnarray}
\hat{\u}_m & \rightharpoonup & \hat{\u}\text{ in }H^1(\Omega) \label{weak_u}\\
\H_m & \rightharpoonup & \H\text{ in }L^2(\Omega)\label{weak_H}\\
\Q_m & \rightharpoonup & \Q \text{ in }H^2(\Omega). \label{weak_Q}
\end{eqnarray}
Next, we will use the following auxiliary lemma \cite[Lemma 1.3]{lions1972}: 
\begin{lemma}\it 
	Let $\mathcal{O}$ be a bounded domain. Let $p_m(x)$ and $p(x)$ be such functions from $L^q(\mathcal{O})$, $1<q<\infty$, such that 
	\begin{equation}
	\|p_m\|_{L^q(\mathcal{O})}\leq C\text{ and }p_m\to p \text{ a.e. in }\mathcal{O}. 
	\end{equation}
	Then $p_m\rightharpoonup p$ in $L^q(\mathcal{O})$.
	\label{lemma-lions}
\end{lemma}
From \eqref{a-priori-bound-2} and Lemma~\ref{lemma-lions} we get 
\begin{equation}
\hat{\H}_M(\Q_m) \rightharpoonup \hat{\H}_M(\Q)\text{ in } L^2(\Omega). \label{weak_hat_H} 
\end{equation} 
Using \eqref{weak_H}, \eqref{weak_Q}, and \eqref{weak_hat_H} as well as the trace theorem,
we can pass to the limits $m\to\infty$ in the weak formulation of \eqref{aux_elliptic}: 
\begin{equation}
K\int\limits_{\Omega}\nabla \Q \cdot \nabla G\,\text{d}x+W\int\limits_{\partial \P_{\text{st}}}(\Q_{\text{pref}}-\Q):G\,\text{d}S_x+\int\limits_{\Omega}\hat{\H}_M(\Q):G\,\text{d}x=\int\limits_{\Omega} \H:G \,\text{d}x \label{weak_limit_in_aux}
\end{equation}
for all smooth test functions $G$.


Next, we pass to the limit in \eqref{weak_galerkin_u}-\eqref{weak_galerkin_Q} using \eqref{weak_u}, \eqref{weak_Q}, $H^2(\Omega)\hookrightarrow H^1(\Omega)\hookrightarrow L^2(\Omega)$, and the property that product of strongly and weakly converging sequences weakly converges to the product of corresponding limits. We get ($\u=\u_{\text{os}}  +\hat{\u}$):
\begin{eqnarray}
&&\eta \int\limits_{\Omega}\nabla \hat{\u}:\nabla \Psi_k\,\text{d}x  +\eta \int\limits_{\Omega}\nabla \u_{\text{os}}:\nabla \Psi_k\,\text{d}x+ \rho\int\limits_{\Omega}(\u\cdot \nabla)\u\cdot \Psi_k\,\text{d}x\nonumber \\&& \hspace{120 pt}+\int\limits_{\Omega} \sigma_a(\Q,\H): \nabla \Psi_{k}\,\text{d}x=K\int\limits_{\Omega} \nabla \Q\astrosun \nabla \Q: \nabla \Psi_{k}\,\text{d}x.\label{weak_limit_u}
\end{eqnarray}
\begin{eqnarray}
&&\Gamma\int\limits_{\Omega}\H:\Phi_k \,\text{d}x -
\int\limits_{\Omega}\left((\u\cdot \nabla)\Q-S(\nabla \u,\Q)\right):\Phi_k\,\text{d}x+ \int_{\Omega} F_{\text{ext},M} : \Phi_{k} \,\text{d}x=0.\label{weak_limit_Q}
\end{eqnarray}
Finally, we can drop subscript $M$ in $\hat{\H}_M(\Q)$ and $F_{\text{ext},M}(\Q,\Q_\infty)$ due to the $L^{\infty}${\it -a-priori} bound on solution of \eqref{eq:stst_lc_1}-\eqref{eq:stst_lc_6} in Lemma~\ref{L_inf_Q}.


\section{Well-posedness of time dependent problem $-$ Proof of Theorem~\ref{theorem-finite-time}}
\label{sec:time-dependent}

In this Section, we prove the local-in-time existence of the unique solution with additional regularity by using Banach fixed point theorem. In Section \ref{operator_space}, we will write the time dependent problem in operator form. 
In Section \ref{lipshitz_non_linear} and \ref{local_in_time_theorem}, we will address the Lipshitz properties of the non-linear part and the solvability of the linear part of PDE system. In Section \ref{local_in_time_theorem}, we will prove the local-in-time existence and uniqueness by Banach fixed point theorem.

\subsection{Operators and function spaces} \label{operator_space}
We first define the projection operator $P_{\sigma}: H^{-1}(\Omega) \rightarrow H^{-1}_{\sigma}(\Omega)$ onto the space of divergence-free functions so that if we apply $P_{\sigma}$ to \eqref{eq:time_lc_1}, the pressure $p$  is eliminated. Specifically, the equation \eqref{eq:time_lc_1} becomes
\begin{eqnarray}
	\partial_t \u + P_{\sigma}( \nabla \cdot (\u \otimes \u))  +  \frac{d\V}{dt} - \rho^{-1}\eta  P_{\sigma}(\Delta \u)= \rho^{-1} P_{\sigma}(\nabla \cdot \sigma_{\text{ela}}(\Q)). \label{eq:time_lc_1_rewrite}
\end{eqnarray}

Now we consider the problem consisting of \eqref{eq:time_lc_1_rewrite}, \eqref{eq:time_lc_2}-\eqref{eq:time_lc_6} with force and torque balances \eqref{eq:time_ft_1}, \eqref{eq:time_ft_2_rewrite}. The tuple of unknowns is $\mathcal{U}=(\u,\Q,\boldsymbol{\omega},\V)^{\text{T}}$. We rewrite the problem as 
\begin{equation}
\mathcal{L}\mathcal{U}=\mathcal{N}(\mathcal{U}),\label{lu_eq_to_nu}
\end{equation}
where we define linear operator $\cL$ and non-linear operator $\cN$ as
	\begin{align}
		\cL
		\left(
		\begin{matrix}
			\u \\
			\Q \\
			\boldsymbol{\omega} \\
			\V
		\end{matrix}
		\right) = \partial_t 
		\left(
		\begin{matrix}
			\u \\
			\Q \\
			\boldsymbol{\omega} \\
			\V
		\end{matrix}
		\right)
		- \left(
	\begin{matrix}
		\rho^{-1}\eta P_{\sigma} ( \Delta \u) \\
		\Gamma K\Delta \Q \\
		0 \\
		0
	\end{matrix}
	\right) \label{def_of_L}
\end{align}
and
\begin{align}
	\cN
	\left(
	\begin{matrix}
		\u \\
		\Q \\
		\boldsymbol{\omega} \\
		\V
	\end{matrix}
	\right) = 
	\left(
	\begin{matrix}
		\rho^{-1} P_{\sigma} \left(\nabla \cdot (\sigma_{\text{ela}}(\Q) -  \rho\u \otimes \u) \right) - \dfrac{d\V}{dt} \\
		-\u \cdot \nabla \Q + \Gamma \hat{\H}(\Q) + S(\nabla \u, \Q) + F_{\text{ext}}(\Q, \Q_{\infty}) \\
		\frac{1}{I} \int_{\partial \P} \x \times \sigma \boldsymbol{\nu}  + \boldsymbol{\ell}\,\text{d}S_{\x} \\
		\frac{1}{m} \int_{\partial \P} \sigma \boldsymbol{\nu} \, \text{d}S_\x
	\end{matrix}
	\right).
\end{align}

To handle the nonlinear and inhomogeneous boundary conditions, 
we represent unknown functions $\u$ and $\Q$ as 
\begin{equation*}
\u = \u_h + \u_{\text{os}}\text{ and }\Q = \Q_h + \Q_{\text{os}}.
\end{equation*}
The offset function $\u_\text{os}$ is given by 
\begin{eqnarray}
	&&-\eta \Delta \u_{\text{os}} + \nabla p_{\text{os}} = 0 \text{ in } \Omega \label{def_off_set_u_1}\\
	&&\u_{\text{os}} = u_{\text{sq}}(\boldsymbol{\alpha}(t), \x)\boldsymbol{\tau} + \boldsymbol{\omega}(t) \times \x \text{ on } \partial \P \label{def_off_set_u_2} \\
	&& \u_{\text{os}} \text{ periodic in } \Pi \label{def_off_set_Q_3}
\end{eqnarray}
 
\noindent The offset function $\Q_{\text{os}}$ is defined such that 
\begin{eqnarray}
	&& K\partial_\nu \Q_{\text{os}} = W(\Q_{\text{pref}}-\Q_{\text{os}}) \text{ on }\partial \P \label{def_off_set_Q_2}\\
	&& \Q_{\text{os}} \text{ periodic in } \Pi \label{def_off_set_Q_4}
\end{eqnarray}
Specifically, we define
\begin{eqnarray}
	\Q_{\text{os}}(\x) = \Q_{\text{pref}}\left( \dfrac{\x}{\|\x\|}\right) \psi(\|\x\|), \quad \x\in \Pi\setminus \mathcal{P}\label{def_of_Q_os}
\end{eqnarray}
Here $\psi(\|\x\|)\geq 0$ is a smooth function such that $\psi(\|\x\|) = 1$ for $\x \in  \left(\partial\mathcal{P}+B_{r_{*}}(\boldsymbol{0})\right)\cap \Pi\setminus \mathcal{P}$ with $r_{*}=\text{dist}(\partial \Pi, \partial \mathcal{P})/4$, and  $\psi(\|\x\|) = 0$ when $\|\x\| > 2r_{*}$.  
Boundary condition \eqref{def_off_set_Q_2} is satisfied since $\partial_{\nu} \Q_{\text{os}}=\partial_{\|\x\|}\Q_{\text{os}}=\boldsymbol{0}$ and $\left.\Q_{\text{os}}\right|_{\partial \mathcal{P}}=\Q_{\text{pref}}.$ The offset function $\Q_{\text{os}}$ can be extended periodically so it satisfies \eqref{def_off_set_Q_4} since $\Q_{\text{os}}\equiv 0$ on $\partial \Pi$. We point out that $\Q_{\text{os}}$ is the solution of the Poisson problem with boundary conditions \eqref{def_off_set_Q_2}-\eqref{def_off_set_Q_4} and the partial differential equation $-\Delta \Q_{\text{os}}=\boldsymbol{f}$ with $\boldsymbol{f}=-\Delta\left(\Q_{\text{pref}}\left( \dfrac{\x}{\|\x\|}\right) \psi(\|\x\|)\right)$.
Note that the offset function $\u_\text{os}$ depends on unknown orientation angle $\boldsymbol{\alpha}(t)$ and angular velocity $\boldsymbol{\omega}(t)$ whereas $\Q_{\text{os}}$ does not. 
Therefore, $\u_{\text{os}}$ changes in time $t$ while $\Q_{\text{os}}$ is independent of time~$t$.


With the above, the functions $\u_h$ and $\Q_h$ satisfy homogeneous boundary conditions.
Their equations in $\Omega$ are similar to the original \eqref{eq:time_lc_1_rewrite} and 
\eqref{eq:time_lc_4}. More precisely, these equations with force and torque balances in the form of \eqref{lu_eq_to_nu} look as follows: 
\begin{align}
	\mathcal{L}  
	\left(
	\begin{matrix}
		\u_h \\
		\Q_h \\
		\boldsymbol{\omega} \\
		\V
	\end{matrix}
	\right) =
\mathcal{J} 
	\left(
	\begin{matrix}
		\u_h \\
		\Q_h \\
		\boldsymbol{\omega} \\
		\V
	\end{matrix}
	\right) =:
	 \mathcal{N}
	\left(
	\begin{matrix}
		\u_h + \u_{\text{os}} \\
		\Q_h + \Q_{\text{os}} \\
		\boldsymbol{\omega} \\
		\V
	\end{matrix}
	\right) - 
	\mathcal{L}
	\left(
	\begin{matrix}
		\u_{\text{os}} \\
		\Q_{\text{os}} \\
		0 \\
		0
	\end{matrix}
	\right)
\label{lu_equal_to_ju}
\end{align}
 
To describe the domains of the operator $\mathcal{L}$,
 we introduce the following Banach spaces:
\begin{align}
	X_{\u} &= \left\{\u \in H^2(0, T;L_{\sigma}^{2}(\Omega)) \cap H^{1}(0, T;H_{0, \sigma}^{2}(\Omega)) \left|\begin{array}{l} \u = 0  \text{ on } \partial \P,\\ \u \text{ periodic in } \Pi\end{array}\right. \right\} \\
X_\Q &= \left\{\Q \in H^{2}(0, T;H^1(\Omega)) \cap H^{1}(0, T;H^{3}(\Omega)) \left| \begin{array}{l}\partial_{\nu}\Q = -W\Q \text{ on } \partial \P,\\ \Q \text{ periodic in } \Pi\end{array}\right. \right\}
\end{align}
with corresponding norms
\begin{eqnarray}
	&&\|\u\|_{X_\u} = \left(\|\u\|_{H^2(0, T;L_{\sigma}^{2}(\Omega))}^2 + \|\u\|_{H^{1}(0, T;H_{0, \sigma}^{2}(\Omega))}^2 + \|\u\big|_{t = 0}\|_{H^{2}_{0, \sigma}(\Omega)}^2 + \|\partial_t \u\big|_{t = 0}\|_{H^{1}_{0, \sigma}(\Omega)}^2\right)^{\frac{1}{2}} \nonumber \\
	&&\|\Q\|_{X_\Q} = \left(\|\Q\|_{H^2(0, T;H^{1}(\Omega))}^2 + \|\Q\|_{H^{1}(0, T;H^{3}(\Omega))}^2 + \|\Q\big|_{t = 0}\|_{H^{3}(\Omega)}^2 + \|\partial_t \Q\big|_{t = 0}\|_{H^{2}(\Omega)}^2\right)^{\frac{1}{2}}. \nonumber 
\end{eqnarray}
Introduce also 
\begin{align*}
	Y_{\u} = H^{1}(0, T;L^{2}_{\sigma}(\Omega)),\quad Y_\Q = H^{1}(0, T;H^{1}(\Omega)). 
\end{align*}
Then $X = X_{\u} \times X_\Q \times H^{2}(0, T) \times H^{2}(0, T)$ and $Y = Y_{\u} \times Y_\Q \times H^{1}(0, T) \times H^{1}(0, T)$ are the domain and the range of the operator $\mathcal{L}$. The corresponding norms are:
\begin{eqnarray*}
	&&\|(\u, \Q, \OMEGA, \V)\|_X = (\|\u\|_{X_\u}^2 + \|\Q\|_{X_\Q}^2 + \|\OMEGA\|_{H^{2}(0, T)}^2 + \|\V\|_{H^{2}(0, T)}^2)^\frac{1}{2}, \\
	&&\|(\u, \Q, \OMEGA, \V)\|_Y = (\|\u\|_{Y_\u}^2 + \|\Q\|_{Y_\Q}^2 + \|\OMEGA\|_{H^{1}(0, T)}^2 + \|\V\|_{H^{1}(0, T)}^2)^\frac{1}{2}.
\end{eqnarray*}

\subsection{Lipschitz property of the non-linear part} \label{lipshitz_non_linear}
\noindent In this section we show the Lipschitz property of the non-linear operator $\mathcal{J}$ with respect to the norms of the spaces $X$ and $Y$. Below, we will use short notations for spaces of functions depending on both $t$ and $\x$, for example, $H^{2}(0, T;L^{2}(\Omega))$ will be denoted by $H^2L^2$. We start with the following estimates with constants vanishing as $T\to 0$.
\begin{proposition}\it
	There exists a constant $C(T)$ such that $C(T)\to 0$ as $T\to 0$ and each of inequalities ({\it i})-({\it iv}) below hold for all $f$ and $g$ as long as the left-hand side of the inequality is finite:
	\begin{eqnarray}
		({\it i})&& \left\|f\right\|_{L^{\infty}L^{\infty}}\leq C(T)\|f\|_{H^1H^2},\label{est_inf_inf}\\
		 ({\it ii})&&\left\| f \right\|_{H^{1}L^{\infty}}\leq C(T) (\left\| f \right\|_{H^{1}H^{2}} + \left\| f \right\|_{H^{2}L^{2}}),\label{est_1_inf}\\
		 ({\it iii})&& 	\|fg\|_{H^1H^1}\leq  C(T) (\left\| f \right\|_{H^{1}H^{2}} + \left\| f \right\|_{H^{2}L^{2}}) \cdot (\left\| g \right\|_{H^{1}H^{2}} + \left\| g \right\|_{H^{2}L^{2}}),\label{est_1_1_1}\\
	 ({\it iv})&&	\|fg\|_{H^1H^1} \leq C(T) (\left\| f \right\|_{H^{1}H^{3}} + \left\| f \right\|_{H^{2}H^{1}}) \cdot \left\| g \right\|_{H^{1}H^{1}}.\label{est_1_1_2}
	\end{eqnarray}
\end{proposition}
\begin{proof}
	In the proof, below $C$ is independent from $T$ unless the dependence is indicated via the following notation $C(T)$. All constants $C(T)$ vanish as $T\to 0$. We will use the following inequalities in the proof
	\begin{align}
		& \left\|f\right\|_{L^{\infty}(\Omega)}\leq C\|f\|_{H^2(\Omega)} \text{ (General Sobolev Inequality, \cite{evans1998pde})} \label{GSI_ine}\\
		& \|f\|_{L^{\infty}(0,T)}\leq C\|f\|^{1/2}_{L^2(0,T)}\|f\|^{1/2}_{H^1(0,T)} \text{ (Agmon's inequality in 1D)} \label{AG_1d_ine}\\
		& \|f\|_{L^{\infty}(\Omega)}\leq C\|f\|^{1/4}_{L^2(\Omega)}\|f\|^{3/4}_{H^2(\Omega)} \text{ (Agmon's inequality in 2D and 3D)} \label{AG_3d_ine}\\
		& \|f\|_{L^{2}(0,T)}\leq T^{1/2} \|f\|_{L^{\infty}(0,T)} \label{holder_ine}\\
		& \|f\|_{L^{\infty}(0,T)}\leq C \|f\|_{H^1(0,T)} \label{FC_ine} \\
		& \|fg\|_{H^1(\Omega)}\leq C \left(\|f\|_{H^1(\Omega)}\|g\|_{L^\infty(\Omega)} + \|f\|_{L^\infty(\Omega)}\|g\|_{H^1(\Omega)}\right) \label{holder_ine_2} \\
		& \|fg\|_{H^1(\Omega)}\leq C\left( \|f\|_{L^2(\Omega)}\|g\|_{W^{1,\infty}(\Omega)} + \|f\|_{W^{1,\infty}(\Omega)}\|g\|_{L^2(\Omega)}\right) \label{holder_ine_3}
	\end{align}
	
	\smallskip 
	
	\noindent{\it Proof of ({\it i})}: 
	\begin{eqnarray}
	\left\| f \right\|_{L^{\infty}L^{\infty}} &&\leq\left\| f \right\|_{L^{\infty}H^{2}} \quad \text{(use \eqref{GSI_ine})}\nonumber\\
&&\leq \left\| f \right\|_{L^{2}H^{2}}^{\frac{1}{2}} \left\| f \right\|_{H^{1}H^{2}}^{\frac{1}{2}} \, (\text{use \eqref{AG_1d_ine}})\nonumber\\
	 &&\leq CT^{1/4}\left\| f \right\|_{L^{\infty}H^{2}}^{\frac{1}{2}} \left\| f \right\|_{H^{1}H^{2}}^{\frac{1}{2}} \quad (\text{use \eqref{holder_ine}})
	 \nonumber\\&&
	 \leq CT^{1/4}\left\| f \right\|_{H^{1}H^{2}} \quad  (\text{use \eqref{FC_ine}}). \label{main_ine_1}
	\end{eqnarray}
	
	\smallskip 
	
	\noindent{\it Proof of ({\it ii})}:
	\begin{align}
	\left\| f \right\|_{H^{1}L^{\infty}} & \leq C\left(\left\| f \right\|_{L^{2}L^{\infty}} + \left\| \partial_t f \right\|_{L^{2}L^{\infty}}\right) \nonumber\\
	&\leq C(T^{1/2} \left\| f \right\|_{L^{\infty}H^{2}} + \left\| \partial_t f \right\|_{L^{2}L^{2}}^{1/4} \cdot \left\| \partial_t f \right\|_{L^{2}H^{2}}^{3/4}) \text{ (use \eqref{GSI_ine}, \eqref{holder_ine}, and \eqref{AG_3d_ine})}\nonumber\\ 
	& \leq C(T^{1/2} \left\| f \right\|_{L^{\infty}H^{2}} + T^{1/8} \left\| \partial_t f \right\|_{L^{\infty}L^{2}}^{1/4} \cdot \left\| \partial_t f \right\|_{L^{2}H^{2}}^{3/4}) \text{ (use \eqref{holder_ine})}\nonumber\\
	& \leq C(T^{1/2} \left\| f \right\|_{H^{1}H^{2}} + T^{1/8} \left\| \partial_t f \right\|_{H^{1}L^{2}}^{1/4} \cdot \left\| \partial_t f \right\|_{L^{2}H^{2}}^{3/4}) \text{ (use \eqref{FC_ine})} \nonumber\\
	& \leq C(T^{1/2} \left\| f \right\|_{H^{1}H^{2}} + T^{1/8} \left\| f \right\|_{H^{2}L^{2}}^{1/4} \cdot \left\|  f \right\|_{H^{1}H^{2}}^{3/4}) \nonumber\\
	& \leq C(T) (\left\| f \right\|_{H^{1}H^{2}} + \left\| f \right\|_{H^{2}L^{2}}). \label{main_ine_2}
	\end{align}

	\noindent{\it Proof of ({\it iii})}:
	\begin{align}
		\|fg\|_{H^1H^1} & \leq \|f\|_{H^1L^{\infty}}\|g\|_{L^{\infty}H^{1}} + \|g\|_{H^1L^{\infty}}\|f\|_{L^{\infty}H^{1}} \nonumber\\
		& \hspace{2cm} + \|f\|_{H^1H^1}\|g\|_{L^{\infty}L^{\infty}} + \|g\|_{H^1H^1}\|f\|_{L^{\infty}L^{\infty}} \text{ (use \eqref{holder_ine_2})} \label{step_1}
	\end{align}
Next, estimate each term in the right-hand side of \eqref{step_1}:
	\begin{align}
		& \|f\|_{H^1L^{\infty}}\|g\|_{L^{\infty}H^{1}} \leq C(T)(\left\| f \right\|_{H^{1}H^{2}} + \left\| f \right\|_{H^{2}L^{2}}) \cdot \|g\|_{H^{1}H^{1}} \text{ (use \eqref{main_ine_2} and \eqref{FC_ine})} \\
		& \|g\|_{H^1L^{\infty}}\|f\|_{L^{\infty}H^{1}} \leq C(T)(\left\| g \right\|_{H^{1}H^{2}} + \left\| g \right\|_{H^{2}L^{2}}) \cdot \|f\|_{H^{1}H^{1}} \text{ (use \eqref{main_ine_2} and \eqref{FC_ine})} \\
		& \|f\|_{H^1H^1}\|g\|_{L^{\infty}L^{\infty}} \leq C(T)\|f\|_{H^1H^1}\|g\|_{H^1H^2} \text{ (use \eqref{main_ine_1})} \\
		& \|g\|_{H^1H^1}\|f\|_{L^{\infty}L^{\infty}} \leq C(T)\|g\|_{H^1H^1}\|f\|_{H^1H^2} \text{ (use \eqref{main_ine_1})}. \label{step_5}
	\end{align}
	Combining \eqref{step_1}-\eqref{step_5}, we obtain \eqref{est_1_1_1}.
	
	\smallskip 
	
	\noindent{\it Proof of ({\it iv})}:
		\begin{align}
		\|fg\|_{H^1H^1} & \leq \|f\|_{H^1W^{1, \infty}}\|g\|_{L^{\infty}L^{2}} + \|f\|_{H^1L^{\infty}}\|g\|_{L^{\infty}H^{1}} \nonumber\\
		& \hspace{2cm} + \|f\|_{L^{\infty}W^{1, \infty}}\|g\|_{H^{1}L^{2}} + \|f\|_{L^{\infty}L^{\infty}}\|g\|_{H^{1}H^{1}} \text{ (use \eqref{holder_ine_3})} \label{step_2_1}
	\end{align}
Next, estimate each term in the right-hand side of \eqref{step_2_1}:
	\begin{align}
		 \|f\|_{H^1W^{1, \infty}}\|g\|_{L^{\infty}L^{2}} &\leq C(T) (\left\| f \right\|_{H^{1}H^{3}} + \left\| f \right\|_{H^{2}H^{1}}) \cdot \|g\|_{H^{1}L^{2}} \text{ (use \eqref{main_ine_2} and \eqref{FC_ine})} \\
		 \|f\|_{H^1L^{\infty}}\|g\|_{L^{\infty}H^1} &\leq C(T) (\left\| f \right\|_{H^{1}H^{2}} + \left\| f \right\|_{H^{2}L^2}) \cdot \|g\|_{H^{1}H^{1}} \text{ (use \eqref{main_ine_2} and \eqref{FC_ine})} \\
		 \|f\|_{L^{\infty}W^{1, \infty}}\|g\|_{H^{1}L^{2}} &\leq 
		C(T) \|f\|_{H^1H^3} \cdot  \|g\|_{H^{1}L^{2}} \text{ (use \eqref{main_ine_1})} \\
		\|f\|_{L^{\infty}L^{\infty}}\|g\|_{H^{1}H^{1}} &\leq 
		C(T) \|f\|_{H^1H^2} \cdot  \|g\|_{H^{1}H^{1}} \text{ (use \eqref{main_ine_1})}. \label{step_2_5}
	\end{align}
	Combining \eqref{step_2_1}-\eqref{step_2_5}, we obtain \eqref{est_1_1_2}.

\end{proof}

\begin{remark}
It is useful to rewrite \eqref{est_1_1_2} with  the norm of space $X_\Q$:
\begin{equation}
\|fg\|_{H^1H^1} \leq C(T) \left\| f \right\|_{X_\Q}  \left\| g \right\|_{H^{1}H^{1}}.\label{est_1_1_2_prime}
\end{equation}
\end{remark}

\begin{lemma}\it \label{prop_lipshitz_J}
 For all $R>0$, there exists a time $T > 0$ such that
for all $(\u^{(i)}_h, \Q^{(i)}_h, \boldsymbol{\omega}^{(i)}, \V^{(i)}) \in B_X(0, R)=\{(\u_h, \Q_h, \boldsymbol{\omega}, \V) \in X \big| \|(\u_h, \Q_h, \boldsymbol{\omega}, \V)\|_{X} \leq R\}$, $i=1,2$, then
		\begin{eqnarray}
			&&\|\mathcal{J}(\u^{(1)}_h, \Q^{(1)}_h, \boldsymbol{\omega}^{(1)}, \V^{(1)}) - \mathcal{J}(\u^{(2)}_h, \Q^{(2)}_h, \boldsymbol{\omega}^{(2)}, \V^{(2)})\|_{Y}  \nonumber \\
			&& \hspace{3.0cm} \leq C(T, R)\|(\u^{(1)}_h, \Q^{(1)}_h, \boldsymbol{\omega}^{(1)}, \V^{(1)}) - (\u^{(2)}_h, \Q^{(2)}_h, \boldsymbol{\omega}^{(2)}, \V^{(2)})\|_{X}. \label{lipshitz_S}
		\end{eqnarray}
Moreover, the constant coefficient $C(T, R) \rightarrow 0$ when $T \rightarrow 0$. 
\end{lemma}

\begin{proof}
$~$


	 \noindent{STEP 1.} We first establish the Lipschitz continuity of $\u_{\text{os}}$, the solution of \eqref{def_off_set_u_1},\eqref{def_off_set_u_2},\eqref{def_off_set_Q_3}, with respect to $\boldsymbol{\alpha}$ and $\boldsymbol{\omega}$. For given $\boldsymbol{\alpha}^{(i)}(t)$ and $\boldsymbol{\omega}^{(i)}(t)$, $i=1,2$, such that $\boldsymbol{\alpha}^{(1)}(0)=\boldsymbol{\alpha}^{(2)}(0)$, one has
	\begin{align*}
		&-\eta \Delta \left(\u_{\text{os}}^{(1)} - \u_{\text{os}}^{(2)} \right) + \nabla \left( p_{\text{os}}^{(1)} - p_{\text{os}}^{(2)} \right) = 0 \text{ in } \Omega, \\
		&\u_{\text{os}}^{(1)} - \u_{\text{os}}^{(2)} = \left( u_{\text{sq}}(\boldsymbol{\alpha}^{(1)}(t), \x) - u_{\text{sq}}(\boldsymbol{\alpha}^{(2)}(t) , \x) \right)\boldsymbol{\tau} + (\boldsymbol{\omega}^{(1)}(t) - \boldsymbol{\omega}^{(2)}(t)) \times \x \text{ on } \partial \P,\\
		&
		\u_{\text{os}}^{(1)}, \u_{\text{os}}^{(2)} \text{ periodic in }\Pi.
	\end{align*}
   Due to the stability of the Stokes operator (similar to \cite[Theorem IV.6.1]{Galdi2011}) 
   \begin{equation}
   \|\u_{\text{os}}\|_{H^2(\Omega)}\leq C \eta^{-1} \left(\|u_{\text{sq}}\|_{L^{2}(\partial \mathcal{P})}+|\OMEGA(t)|\right) \label{elliptic_galdi}
   \end{equation}
   and smooth dependence of $u_{\text{sq}}$ in $\alpha(t)$, we have 
	\begin{eqnarray}
		&&\|\u_{\text{os}}^{(1)} - \u_{\text{os}}^{(2)}\|_{H^2L^2} + \|\u_{\text{os}}^{(1)} - \u_{\text{os}}^{(2)}\|_{H^1H^2} \nonumber \\
		&&\hspace{100pt} \leq C\eta^{-1}\left(\|\boldsymbol{\alpha}^{(1)} - \boldsymbol{\alpha}^{(2)}\|_{H^2(0,T)} + \|\boldsymbol{\omega}^{(1)} - \boldsymbol{\omega}^{(2)}\|_{H^2(0,T)}\right). \label{stability_u_os_1} \\
		&&\|\u_{\text{os}}^{(1)} - \u_{\text{os}}^{(2)}\|_{H^1H^2}  \leq C\eta^{-1}\left(\|\boldsymbol{\alpha}^{(1)} - \boldsymbol{\alpha}^{(2)}\|_{H^1(0,T)} + \|\boldsymbol{\omega}^{(1)} - \boldsymbol{\omega}^{(2)}\|_{H^1(0,T)}\right). \label{stability_u_os_1.1}
	\end{eqnarray}
	Since $\boldsymbol{\alpha}^{(i)}(t) = \boldsymbol{\alpha}^{{(i)}}(0)+\int_{0}^{t} \boldsymbol{\omega}^{(i)}(\tau)\times \boldsymbol{\alpha}^{(i)}(\tau) \text{d}\tau$ (see \eqref{omega_vs_alpha}) and $|\boldsymbol{\alpha}^{(i)}(t)|=1$, $i=1,2$ as well as $\int\limits_0^{T}|h(t)|^2\,\text{d}t\leq T^2\int\limits_0^{T}|h_t(0,T)|^2\,\text{d}t$ for all $h\in H^1(0,T)$, one  gets 
	\begin{eqnarray}
	&& \|\boldsymbol{\alpha}^{(1)} - \boldsymbol{\alpha}^{(2)}\|_{H^2(0,T)} \leq T\|\boldsymbol{\omega}^{(1)} - \boldsymbol{\omega}^{(2)}\|_{H^2(0,T)}, \\
	&& \|\boldsymbol{\omega}^{(1)} - \boldsymbol{\omega}^{(2)}\|_{H^1(0,T)} \leq C(T)\|\boldsymbol{\omega}^{(1)} - \boldsymbol{\omega}^{(2)}\|_{H^2(0,T)}.
	\end{eqnarray} Then \eqref{stability_u_os_1} and \eqref{stability_u_os_1.1} become
	\begin{align}
		&\|\u_{\text{os}}^{(1)} - \u_{\text{os}}^{(2)}\|_{H^2L^2} + \|\u_{\text{os}}^{(1)} - \u_{\text{os}}^{(2)}\|_{H^1H^2} \leq C\eta^{-1}(1+T)\|\boldsymbol{\omega}^{(1)} - \boldsymbol{\omega}^{(2)}\|_{H^2(0,T)},\label{stability_u_os_2} \\
		&\|\u_{\text{os}}^{(1)} - \u_{\text{os}}^{(2)}\|_{H^1H^2} \leq C(T)\eta^{-1}\|\boldsymbol{\omega}^{(1)} - \boldsymbol{\omega}^{(2)}\|_{H^2(0,T)}.\label{stability_u_os_2.1}
	\end{align}
	
	Applying \eqref{elliptic_galdi} to $\partial^{k}_t\u^{(i)}$ with
 $k=0,1,2$ and using the definition of time-independent $\Q_{\text{os}}$ \eqref{def_of_Q_os},
there is a $C>0$ depending on $\Omega$ and $q_{\infty}$ such that 
	\begin{align} 
		&\|\u_{\text{os}}^{(i)}\|_{H^2L^2} + \|\u_{\text{os}}^{(i)}\|_{H^1H^2} \leq C\eta^{-1}(\|\OMEGA^{(i)}\|_{H^2(0,T)}+1), \quad i=1,2. \label{uniform_u_offset}
\\
		&\|\Q_{\text{os}}\|_{H^2H^1} + \|\Q_{\text{os}}\|_{H^1H^3} \leq C. \label{uniform_Q_offset}
	\end{align}
(Though $\Q_{\text{os}}$ is independent of time, here we use its $H^2H^1$ and $H^1H^3$ norms for the clarity of arguments below.) We will also need the following inequality:
 \begin{equation}
 \|\Q^{(i)}\|_{H^1H^3}+\|\Q^{(i)}\|_{H^2H^1}\leq C+R, \quad i=1,2.\label{Q_i_bounded}
 \end{equation} 
 Indeed, 
 \begin{eqnarray}
	&&\|\Q^{(i)}\|_{H^1H^3}+\|\Q^{(i)}\|_{H^2H^1}\leq\|\Q^{(i)}-\Q_{\text{os}}\|_{H^1H^3}+\|\Q^{(i)}-\Q_{\text{os}}\|_{H^2H^1}\nonumber\\&&\hspace{240pt}+ \|\Q_{\text{os}}\|_{H^1H^3}+\|\Q_{\text{os}}\|_{H^2H^1} \nonumber\\
	&&\hspace{130pt}\leq \|\Q^{(i)}_{h}\|_{X_\Q}+ \|\Q_{\text{os}}\|_{H^1H^3}+\|\Q_{\text{os}}\|_{H^2H^1}\nonumber \\
	&& \hspace{130pt}\leq C+R.\nonumber
\end{eqnarray}

	\noindent STEP 2. Here we establish the following inequality: 
	\begin{equation}
	\|\rho^{-1}P_\sigma\nabla\cdot\left[\sigma_{\text{ela}}(\Q^{(1)})-\sigma_{\text{ela}}(\Q^{(2)})\right]\|_{H^1L^{2}_{\sigma}}\leq C(T)\|\Q^{(1)}_h-\Q^{(2)}_h\|_{X_\Q}. \label{what_we_want_in_step_2}
	\end{equation} 
	
	To this end, we first note that since $P_{\sigma}\nabla \cdot: H^{1}(\Omega) \rightarrow L^2_{\sigma}(\Omega)$ is a bounded operator \cite[Lemma II.2.5.2]{Sohr2001}, the inequality \eqref{what_we_want_in_step_2} follows from
	\begin{equation}
	\|\sigma_{\text{ela}}(\Q^{(1)})-\sigma_{\text{ela}}(\Q^{(2)})\|_{H^1H^1}\leq C(T)\|\Q^{(1)}_h-\Q^{(2)}_h\|_{X_\Q}. \label{what_we_want_in_step_2_modified}
	\end{equation}
	We decompose $\sigma_{\text{ela}}$ into five parts $\sigma_{\text{ela}} = \sigma_K + \sigma_a + \sigma_{s}^1 + \sigma_{s}^2 + \sigma_{s}^3$, where
	\begin{align*}
		&\sigma_{K} = -K\nabla \Q \astrosun \nabla \Q \\
		&\sigma_{a} = K(\Q\Delta \Q - \Delta \Q \Q) \\
		&\sigma_{s}^1 = -\dfrac{2\xi}{d}\H \\
		&\sigma_{s}^2 = -\xi\left[\H\Q + \Q\H\right] + \dfrac{2\xi}{d}\text{Tr}(\Q\H) \\
		&\sigma_{s}^3 = 2\xi\left[\Q\text{Tr}(\Q\H)\right].
	\end{align*}
	Here $\sigma_{s}^{1}, \sigma_{s}^{2}, \sigma_{s}^{3}$ are the linear, bilinear and 
trilinear part of $\sigma_{s}:=\sigma^{1}_s+\sigma^{2}_s+\sigma^{3}_s$, respectively.

	\medskip
	\noindent \textit{Part 1}: $\sigma_{K}(\Q)$.
	\begin{eqnarray}
	&&	\hspace{-30pt}\| \sigma_{K}(\Q^{(1)}) - \sigma_{K}(\Q^{(2)}) \|_{H^1H^1} =\| (K\nabla\Q^{(1)} \astrosun \nabla\Q^{(1)} - K\nabla \Q^{(2)}\astrosun\nabla\Q^{(2)}) \| _{H^1H^1} \nonumber \\
		&&\hspace{15pt}\leq  CK \left(\| \nabla \Q^{(1)}\astrosun \nabla(\Q^{(1)} - \Q^{(2)})\| _{H^1H^1} + \| \nabla(\Q^{(1)} - \Q^{(2)})\astrosun\nabla \Q^{(2)} \| _{H^1H^1}\right) \nonumber \\
		&&\hspace{15pt}=  CK \left(\| \nabla \Q^{(1)}\astrosun \nabla(\Q^{(1)}_h - \Q^{(2)}_h)\| _{H^1H^1} + \| \nabla(\Q^{(1)}_h - \Q^{(2)}_h)\astrosun\nabla \Q^{(2)} \| _{H^1H^1}\right).
	\end{eqnarray}
	Using \eqref{est_1_1_1}, \eqref{uniform_Q_offset} and that $(\u^{(i)}_h, \Q^{(i)}_h, \boldsymbol{\omega}^{(i)}, \V^{(i)}) \in B_X(0, R)$, we get  
	\begin{align}
		& CK \left(\| \nabla \Q^{(1)}\astrosun \nabla(\Q^{(1)}_h - \Q^{(2)}_h)\| _{H^1H^1} + \| \nabla(\Q^{(1)}_h - \Q^{(2)}_h)\astrosun\nabla \Q^{(2)} \| _{H^1H^1}\right) \nonumber \\
		 &\quad \leq  C(T)K\left( \|\nabla(\Q^{(1)}_h - \Q^{(2)}_h)\|_{H^1H^2} + \|\nabla(\Q^{(1)}_h - \Q^{(2)}_h)\|_{H^2L^2} \right) \nonumber \\
		&\hspace*{2cm}\left[\left( \|\nabla\Q^{(1)}\|_{H^1H^2} + \|\nabla\Q^{(1)}\|_{H^2L^2} \right) + \left( \|\nabla\Q^{(2)}\|_{H^1H^2} + \|\nabla\Q^{(2)}\|_{H^2L^2} \right)\right] \nonumber \\
		 &\quad \leq C(T)K(R + 1)\|\Q^{(1)}_h - \Q^{(2)}_h\|_{X_\Q} \nonumber\\
		 &\quad \leq C(T)\|\Q^{(1)}_h - \Q^{(2)}_h\|_{X_\Q}.
	\end{align}
We note that the generic constant $C(T)$ may change from line to line and may depend on, for example, $K$, $R$, $C$ from \eqref{uniform_Q_offset} and $T$ (but recall that notation $C(T)$ also means that $C(T)\to 0$ as $T\to 0$). We sometimes do not merge a parameter, as for example, $K$ in the second line of the above chain of inequalities, to indicate what we used to obtain a bound.     
	
		\medskip
	
		\noindent \textit{Part 2}: $\sigma_{a}(\Q))$.	
			\begin{eqnarray}
				&&\hspace{-40pt} \| \sigma_{a}(\Q^{(1)}) - \sigma_{a}(\Q^{(2)}) \|_{H^1H^1} \nonumber \\
				&&\hspace{10pt}\leq K\| \Q^{(1)} \Delta\Q^{(1)} - \Q^{(2)} \Delta\Q^{(2)} \|_{H^1H^1} + K\| \Delta\Q^{(1)} \Q^{(1)} - \Delta\Q^{(2)} \Q^{(2)} \|_{H^1H^1}. \label{sigma_a_lip}
			\end{eqnarray}

				\noindent Applying \eqref{est_1_1_2} for the first term in the right hand side of \eqref{sigma_a_lip}, one can get
				\begin{eqnarray}
					&& \hspace{-30 pt}\| \Q^{(1)}\Delta\Q^{(1)} - \Q^{(2)} \Delta\Q^{(2)} \|_{H^1H^1} \nonumber\\
					&& \leq  \| (\Q^{(1)}_h - \Q^{(2)}_h) \Delta(\Q^{(1)}_h + \Q_{\text{os}}) \|_{H^1H^1} + \|(\Q^{(2)}_h + \Q_{\text{os}}) \Delta(\Q^{(1)}_h - \Q^{(2)}_h  ) \|_{H^1H^1} \nonumber\\
					&& \leq  C(T) \| \Q^{(1)}_h - \Q^{(2)}_h \|_{X_\Q} \cdot \| \Delta(\Q^{(1)}_h + \Q_{\text{os}})\|_{H^{1}H^{1}}   \nonumber\\
					&& \hspace{60pt}+ C(T) \| \Q^{(2)}_h + \Q_{\text{os}} \|_{X_\Q} \| \Delta(\Q^{(1)}_h - \Q^{(2)}_h) \|_{H^{1}H^{1}} \nonumber\\
					&&\leq  C(T)(R + 1) \|\Q^{(1)}_h - \Q^{(2)}_h\|_{X_\Q}.   \label{sigma_a_lip_step_2}
				\end{eqnarray}
				Applying same arguments for the second term in the right-hand side of \eqref{sigma_a_lip}, one can obtain
				\begin{equation*}
				\|\sigma_{a}(\Q^{(1)}) - \sigma_{a}(\Q^{(2)})  \|_{H^1H^1} \leq  KC(T, R) \|\Q^{(1)}_h - \Q^{(2)}_h\|_{X_\Q}.
				\end{equation*}

			\medskip

				\noindent \textit{Part 3}: $\sigma_{s}^1(\Q)$.
	\begin{eqnarray}
		&&\hspace{-30pt}\| \sigma_{s}^1(\Q^{(1)}) - \sigma_{s}^1(\Q^{(2)})  \|_{H^1H^1} =\dfrac{2\xi}{d}\| \H(\Q^{(1)}) - \H(\Q^{(2)})\| _{H^1H^1} \nonumber \\
		&&\hspace{10pt}\leq  \dfrac{2\xi K}{d}\| \Delta\Q^{(1)} - \Delta\Q^{(2)} \| _{H^1H^1}+ \dfrac{2\xi|a|}{d}\| \Q^{(1)} - \Q^{(2)}\| _{H^1H^1} \nonumber\\
		&&\hspace{50pt} +\dfrac{2\xi|c|}{d} \| \Q^{(1)}\text{Tr}((\Q^{(1)})^2) - \Q^{(2)}\text{Tr}((\Q^{(2)})^2) \| _{H^1H^1}. \label{sigma_s_1_lip_1}
		\end{eqnarray}
		The first two terms in the right-hand side of \eqref{sigma_s_1_lip_1} are bounded as follows:
		\begin{equation*}
		\dfrac{2\xi K}{d}\| \Delta\Q^{(1)} - \Delta\Q^{(2)} \| _{H^1H^1}+ \dfrac{2\xi|a|}{d}\| \Q^{(1)} - \Q^{(2)} \| _{H^1H^1}\leq C \|\Q^{(1)} - \Q^{(2)}\|_{X_\Q}.
		\end{equation*}
		Next, we bound the third (cubic) term  in the right-hand side of \eqref{sigma_s_1_lip_1}. Note  
		\begin{align}
			\| \Q^{(1)}\text{Tr}((\Q^{(1)})^2) - \Q^{(2)}\text{Tr}((\Q^{(2)})^2) \| _{H^1H^1}& \leq\| (\Q^{(1)} - \Q^{(2)})\text{Tr}(\Q^{(1)}\Q^{(1)})\| _{H^1H^1} \nonumber\\
		 &\hspace{10pt}+\| \Q^{(2)}\text{Tr}((\Q^{(1)} - \Q^{(2)})\Q^{(1)})\| _{H^1H^1}\nonumber\\& 
		 \hspace{10pt}+ \| \Q^{(2)}\text{Tr}(\Q^{(2)}(\Q^{(1)} - \Q^{(2)}))\| _{H^1H^1} \label{sigma_s_1_lip_2}
		\end{align}
		We show how to bound the first term in the right-hand side of \eqref{sigma_s_1_lip_2}. Other terms are bounded in the same way. 
		Apply \eqref{est_1_1_2_prime} twice to obtain
		\begin{align}
			&\| (\Q^{(1)} - \Q^{(2)})\text{Tr}(\Q^{(1)}\Q^{(1)})\| _{H^1H^1} \nonumber \\ 
			&\hspace{20pt}\leq  C(T)\|\Q^{(1)}_h - \Q^{(2)}_h\|_{X_\Q} \||\Q^{(1)}||\Q^{(1)}|\|_{H^1H^1} \nonumber\\
			&\hspace{20pt} \leq C(T)\|\Q^{(1)}_h - \Q^{(2)}_h\|_{X_\Q}\left(\|\Q^{(1)}_h\|_{X_\Q}+\|\Q_{\text{os}}\|_{L^2H^3}\right)\left(\|\Q^{(1)}_h\|_{H^1H^1}+\|\Q_{\text{os}}\|_{ H^1H^1}\right)\nonumber\\
			&\hspace{20pt} \leq  C(T)(R+1)^2\|\Q^{(1)}_h - \Q^{(2)}_h\|_{X_\Q}. \label{sigma_s_1_lip_3}
		\end{align}
		Thus, we have 
		\begin{align}
		\| \H(\Q^{(1)}) - \H(\Q^{(2)}) \| _{H^1H^1}\leq C(T,R) \|\Q^{(1)}_h - \Q^{(2)}_h\|_{X_\Q}, \label{lip_H}
		\end{align}
		which, in view of \eqref{sigma_s_1_lip_1}, implies
		\begin{align}
			\| \sigma_{s}^1(\Q^i) - \sigma_{s}^1(\Q^j) \|_{H^1H^1} \leq C(T,R) \|\Q^{(1)}_h - \Q^{(2)}_h\|_{X_\Q}. \label{lip_sigma_s_1}
		\end{align}
		
		\noindent \textit{Part 4}: $ \sigma_{s}^2(\Q)$. This is part, 
we will need the following bound
		\begin{align}
		\|\H(\Q^{(i)})\|_{H^1H^1} \leq  C(T,R) , \quad i=1,2. \label{bounds_for_H}
		\end{align}
		which can be obtained by applying same arguments as in  
\eqref{sigma_s_1_lip_1}-\eqref{sigma_s_1_lip_3} for $i=1,2$
		\begin{align}
			\|\H(\Q^{(i)})\|_{H^1H^1} &\leq C(\|\Delta \Q^{(i)}\|_{H^1H^1} + \|\Q^{(i)}\|_{H^1H^1} + \|\Q^{(i)}\text{Tr}((\Q^{(i)})^2)\|_{H^1H^1})\nonumber \\
			& \leq C(T) (\|\Q^{(i)}_h\|_{X_{\Q}} + \|\Q^{(i)}_h\|_{X_{\Q}}^3+1)\leq C(T,R). \nonumber
		\end{align}

Now, we can estimate,
		\begin{align}
			&\| \sigma_{s}^2(\Q^{(1)}) - \sigma_{s}^2(\Q^{(2)}) \|_{H^1H^1} \leq  \xi\|\Q^{(1)}\H(\Q^{(1)}) - \Q^{(2)}\H(\Q^{(2)})\|_{H^1H^1}\nonumber\\
			&\hspace*{6.5cm}  + \xi\|\H(\Q^{(1)})\Q^{(1)} - \H(\Q^{(2)})\Q^{(2)}\|_{H^1H^1} \nonumber\\
			& \hspace*{6.5cm} +\dfrac{2\xi}{d}\|\text{Tr}(\Q^{(1)}\H(\Q^{(1)}) - \Q^{(2)}\H(\Q^{(2)}))\|_{H^1H^1}\nonumber \\
			&\hspace*{5cm}\leq C\|\Q^{(1)}\H(\Q^{(1)}) - \Q^{(2)}\H(\Q^{(2)})\|_{H^1H^1}.
			 \label{sigma_s_2_lip}
			\end{align}
		Next, we use the triangle inequality, \eqref{est_1_1_2_prime},  \eqref{lip_H}, and \eqref{bounds_for_H}: 
			\begin{align}
				&\|\Q^{(1)}\H(\Q^{(1)}) - \Q^{(2)}\H(\Q^{(2)})\|_{H^1H^1}\nonumber \\
				& \hspace{40pt} \leq C(T) \| \Q^{(1)}_h - \Q^{(2)}_h \|_{X_\Q} \| \H(\Q^{(1)}) \|_{H^{1}H^{1}}   \nonumber\\
					& \hspace{60pt}+ C(T) \| \Q^{(2)}_h + \Q_{\text{os}} \|_{X_\Q}\| \H(\Q^{(1)}) - \H(\Q^{(2)}) \|_{H^{1}H^{1}}\nonumber \\
					&\hspace{40pt} 
					\leq C(T,R)\| \Q^{(1)}_h - \Q^{(2)}_h \|_{X_\Q}. \nonumber
			\end{align}
Therefore, 
\begin{align}
&\| \sigma_{s}^2(\Q^{(1)}) - \sigma_{s}^2(\Q^{(2)}) \|_{H^1H^1}\leq 	 C(T,R)\| \Q^{(1)}_h - \Q^{(2)}_h \|_{X_\Q}.
\end{align}

\noindent \textit{Part 5}: $ \sigma_{s}^3(\Q)$.
	\begin{align}
	&\| \sigma_{s}^3(\Q^{(1)}) - \sigma_{s}^3(\Q^{(2)}) \|_{H^1H^1} \leq  2\xi\|\Q^{(1)}\text{Tr}(\Q^{(1)}\H(\Q^{(1)})) - \Q^{(2)}\text{Tr}(\Q^{(2)}\H(\Q^{(2)}))\|_{H^1H^1}\nonumber\\
	& \hspace{140pt}\leq C\| (\Q^{(1)} - \Q^{(2)})\text{Tr}(\Q^{(1)}\H(\Q^{(1)}))\| _{H^1H^1} \nonumber\\
	&\hspace{160pt}+ C\| \Q^{(2)}\text{Tr}((\Q^{(1)} - \Q^{(2)})\H(\Q^{(1)}))\| _{H^1H^1}\nonumber\\& 
	\hspace{160pt}+ C\| \Q^{(2)}\text{Tr}(\Q^{(2)}(\H(\Q^{(1)}) - \H(\Q^{(2)})))\|_{H^1H^1}.
	\end{align}
	Next, applying the same arguments as in \eqref{sigma_s_1_lip_3} with bounds \eqref{lip_H} and \eqref{bounds_for_H} we get 
	\begin{align}
	&\| \sigma_{s}^3(\Q^{(1)}) - \sigma_{s}^3(\Q^{(2)}) \|_{H^1H^1}\leq 	 C(T,R)\| \Q^{(1)}_h - \Q^{(2)}_h \|_{X_\Q}.
	\end{align}

\noindent STEP 3. Here we establish the following inequality: 
\begin{eqnarray}
	&&\|S^2(\nabla \u^{(1)}, \Q^{(1)})-S^2(\nabla \u^{(2)}, \Q^{(2)})\|_{H^1H^{1}} \nonumber \\
	&& \hspace*{1cm} \leq C(T)\left(\|\Q^{(1)}_h-\Q^{(2)}_h\|_{X_\Q} + \|\u^{(1)}_h-\u^{(2)}_h\|_{X_\u} + \|\boldsymbol{\omega}^{(1)} - \boldsymbol{\omega}^{(1)}\|_{H^2(0,T)} \right) \label{what_we_want_in_step_3}.
\end{eqnarray}

	To this end, similar to how we treated $\sigma_s$ in STEP 2, we split $S(\nabla \u, \Q)$ into three parts: 
	\begin{eqnarray}
		S(\nabla \u,\Q) &&=(\xi \D + \A)\left(\Q+\frac{\mathbb I}{d}\right)+\left(\Q+\frac{\mathbb I}{d}\right)(\xi \D- \A)-2\xi\left(\Q+\frac{\mathbb I}{d}\right)\text{tr}(\Q\nabla \u) \nonumber \\
		&& = S^1 + S^2 + S^3, 
		\end{eqnarray}
		where 
\begin{align*}
	& S^1 (\nabla \u)= \frac{2\xi}{d} \D \\
	& S^2 (\nabla \u,\Q)= \xi(\D\Q + \Q\D) + (\A\Q - \Q\A) - \frac{2\xi}{d}\text{tr}(\Q\nabla \u)\\
	& S^3 (\nabla \u,\Q)= - 2\xi\Q\text{tr}(\Q\nabla \u)
\end{align*}		
		are correspondingly the linear, bilinear and trilinear part of $S$. First note that using \eqref{stability_u_os_2}, \eqref{stability_u_os_2.1} and \eqref{uniform_u_offset}, one gets
				\begin{align}
					\|\u^{(1)} - \u^{(2)}\|_{H^1H^2 \cap H^2L^2} &\leq \|\u^{(1)}_h - \u^{(2)}_h\|_{X_\u} + \|\u^{(1)}_{\text{os}} - \u^{(2)}_{\text{os}}\|_{H^1H^2 \cap H^2L^2} \nonumber \\
					& \leq \|\u^{(1)}_h - \u^{(2)}_h\|_{X_\u} + C\eta^{-1}(1+T)\|\boldsymbol{\omega}^{(1)} - \boldsymbol{\omega}^{(2)}\|_{H^2(0,T)} \label{Lipshitz_u}
				\end{align} 
				and
				\begin{align}
					\|\u^{(i)}\|_{H^1H^2 \cap H^2L^2} &\leq \|\u^{(i)}_h\|_{X_\u} + \|\u^{(i)}_{\text{os}}\|_{H^1H^2 \cap H^2L^2} \nonumber \\
					& \leq \|\u^{(i)}_h\|_{X_\u} + C\eta^{-1}(\|\OMEGA^{(i)}\|_{H^2(0,T)}+1), \quad i=1,2. \label{Bound_u}
				\end{align}

		\medskip
		\noindent \textit{Part 1}: $S^1(\nabla \u, \Q)$.

		Since $\D(\u) = \frac{1}{2}(\nabla \u + (\nabla \u)^{T})$, using \eqref{stability_u_os_2.1} one gets
		\begin{align}
			& \| \D(\u^{(1)}) - \D(\u^{(1)})  \|_{H^1H^1} \leq \|\u^{(1)} - \u^{(2)}\|_{H^1H^2}\nonumber\\ 
			& \hspace*{3cm}\leq \|\u^{(1)}_h - \u^{(2)}_h\|_{H^1H^2}+\|\u^{(1)}_{\text{os}} - \u^{(2)}_{\text{os}}\|_{H^1H^2}\nonumber \\
			& \hspace*{3cm} \leq \|\u^{(1)}_h - \u^{(2)}_h\|_{X_\u} + C(T)\eta^{-1}\|\boldsymbol{\omega}^{(1)} - \boldsymbol{\omega}^{(2)}\|_{H^2(0,T)}. \label{needed_for_sigma_hydro}
		\end{align}
		Then
		\begin{eqnarray}
			&&\|S^1(\nabla \u^{(1)}, \Q^{(1)})-S^1(\nabla \u^{(2)}, \Q^{(2)})\|_{H^1H^{1}} \nonumber \\
			&& \hspace*{2cm} \leq C(T)\left(\|\Q^{(1)}_h-\Q^{(2)}_h\|_{X_\Q} + \|\u^{(1)}_h-\u^{(2)}_h\|_{X_\u} + \|\boldsymbol{\omega}^{(1)} - \boldsymbol{\omega}^{(2)}\|_{H^2} \right).
		\end{eqnarray}

		\medskip 
	\noindent \textit{Part 2}: $S^2(\nabla \u, \Q)$.
	\begin{eqnarray}
	&&\|S^2(\nabla \u^{(1)},\Q^{(1)})-S^2(\nabla \u^{(2)},\Q^{(2)})\|_{H^1H^1}\leq 
C \|\nabla \u^{(1)}\Q^{(1)}-\nabla \u^{(2)}\Q^{(2)}\|_{H^1H^1}\nonumber\\
&&\hspace{60 pt}\leq C\|(\nabla \u^{(1)}-\nabla \u^{(2)}) \Q^{(1)}\|_{H^1H^1}+C
\|\nabla \u^{(2)}(\Q^{(1)}-\Q^{(2)})\|_{H^1H^1} \label{S_2_1}
		\end{eqnarray}

Apply \eqref{est_1_1_2}, \eqref{Q_i_bounded}  and \eqref{Lipshitz_u} to obtain
\begin{eqnarray}
	&&\|(\nabla \u^{(1)}-\nabla \u^{(2)}) \Q^{(1)}\|_{H^1H^1}\nonumber\\
	&&\hspace{70pt} \leq C(T)\|\nabla \u^{(1)}-\nabla \u^{(2)}\|_{H^1H^1} \left(\| \Q^{(1)} \|_{H^1H^3} +\| \Q^{(1)} \|_{H^2H^1} \right) \nonumber \\
	&&\hspace{70pt}\leq C(T)\|\u^{(1)}-\u^{(2)}\|_{H^1H^2} \left(\| \Q^{(1)} \|_{H^1H^3} + \| \Q^{(1)} \|_{H^2H^1} \right)\nonumber\\
	&& \hspace{70pt}\leq C(T) (R + 1) \left( \|\u^{(1)}_h - \u^{(2)}_h\|_{X_\u} + \|\boldsymbol{\omega}^{(1)} - \boldsymbol{\omega}^{(2)}\|_{H^2(0,T)} \right) \nonumber  
\end{eqnarray}
 
Applying the similar arguments to the second term in the right-hand side of \eqref{S_2_1}, we obtain
\begin{eqnarray}
	&&\hspace{-20pt}\|S^2(\nabla \u^{(1)}, \Q^{(1)})-S^2(\nabla \u^{(2)}, \Q^{(2)})\|_{H^1H^{1}} \nonumber \\
	&& \hspace*{1cm} \leq C(T)\left(\|\Q^{(1)}_h-\Q^{(2)}_h\|_{X_\Q} + \|\u^{(1)}_h-\u^{(2)}_h\|_{X_\u} + \|\boldsymbol{\omega}^{(1)} - \boldsymbol{\omega}^{(2)}\|_{H^2(0,T)} \right).
\end{eqnarray}

\medskip 
\noindent \textit{Part 3}: $S^3(\nabla \u, \Q)$.
\begin{eqnarray}
	&&\hspace{-20pt}\|S^3(\nabla \u^{(1)},\Q^{(1)})-S^3(\nabla \u^{(2)},\Q^{(2)})\|_{H^1H^1}\nonumber \\
	&&\hspace{60pt}\leq 
2\xi \|\Q^{(1)}\text{tr}\left(\nabla \u^{(1)}\Q^{(1)}\right)-\Q^{(2)}\text{tr}\left(\nabla \u^{(2)}\Q^{(2)}\right)\|_{H^1H^1}\nonumber\\
&&\hspace{60 pt}\leq 2\xi\|\left(\Q^{(1)} - \Q^{(2)}\right)\text{tr}\left(\nabla \u^{(1)}\Q^{(1)}\right)\|_{H^1H^1}\nonumber \\
&&\hspace{70pt} + 2\xi\|\Q^{(2)}\text{tr}\left(\nabla \left(\u^{(1)} - \u^{(2)}\right)\Q^{(1)}\right)\|_{H^1H^1} \nonumber\\ 
&&\hspace{70pt}+ 2\xi\|\Q^{(2)}\text{tr}\left(\nabla \u^{(2)}\left(\Q^{(1)} - \Q^{(2)}\right)\right)\|_{H^1H^1}. \label{S_3_1}
		\end{eqnarray}
Using same arguments as in \eqref{sigma_s_1_lip_3} and taking into account \eqref{uniform_u_offset} and \eqref{Q_i_bounded}, we obtain
\begin{eqnarray}
&&\|\left(\Q^{(1)} - \Q^{(2)}\right)\text{tr}\left(\nabla \u^{(1)}\Q^{(1)}\right)\|_{H^1H^1} \nonumber \\
&&\hspace{60pt}\leq C(T) \|\Q^{(1)} - \Q^{(2)}\|_{H^1H^3 \cap H^2H^1} \|\Q^{(1)}\|_{H^1H^3 \cap H^2H^1}\|\nabla \u^{(1)}\|_{H^1H^1}\nonumber\\
&&\hspace{60pt}\leq C(T)(R^2+1)\|\Q^{(1)}_h - \Q^{(2)}_h\|_{X_\Q}.\nonumber 
\end{eqnarray}

Applying similar arguments for the other two terms in the right-hand side of \eqref{S_3_1}, we obtain 
\begin{eqnarray}
	&&\hspace{-20pt}\|S^3(\nabla \u^{(1)}, \Q^{(1)})-S^3(\nabla \u^{(2)}, \Q^{(2)})\|_{H^1H^{1}} \nonumber \\
	&& \hspace{1cm} \leq C(T,R)\left(\|\Q^{(1)}_h-\Q^{(2)}_h\|_{X_\Q} + \|\u^{(1)}_h-\u^{(2)}_h\|_{X_\u} + \|\boldsymbol{\omega}^{(1)} - \boldsymbol{\omega}^{(2)}\|_{H^2(0,T)} \right).
\end{eqnarray}

				\bigskip
			
				\noindent STEP 4. 
Finally, we show the Lipschitz properties of all the remaining  terms in $\mathcal{J}$. 		
				
				\medskip 
				
				\noindent \textit{Part 1}: $P_{\sigma}(\nabla \cdot (\u \otimes \u))$.

				We again use the fact that $P_{\sigma}\nabla \cdot: H^{1}(\Omega) \rightarrow L^2_{\sigma}(\Omega)$ is a bounded operator:
				\begin{eqnarray}
					&&\| P_{\sigma}(\nabla \cdot (\u^{(1)} \otimes \u^{(1)})) - P_{\sigma}(\nabla \cdot (\u^{(2)} \otimes \u^{(2)})) \|_{H^1L^2_\sigma} \nonumber \\
					&& \hspace*{3.5cm}\leq \| \u^{(1)} \otimes \u^{(1)} - \u^{(2)} \otimes \u^{(2)} \|_{H^1H^1} \nonumber\\
					&& \hspace*{3.5cm} =\| (\u^{(1)} - \u^{(2)}) \otimes \u^{(1)} - \u^{(2)} \otimes (\u^{(1)} - \u^{(2)}) \|_{H^1H^1}.
				\end{eqnarray}
				We apply the same arguments in \textit{Part 1} of STEP 2, that is, apply \eqref{est_1_1_1}, along with \eqref{Lipshitz_u} and \eqref{Bound_u}, to get: 
				\begin{align}
					& \hspace{-10pt}\| \u^{(1)}\otimes (\u^{(1)} - \u^{(2)}) \| _{H^1H^1} +\| (\u^{(1)} - \u^{(2)})\otimes  \u^{(2)} \| _{H^1H^1} \nonumber \\
					 &\quad \leq  C(T)\left( \|\u^{(1)} - \u^{(2)}\|_{H^1H^2} + \|\u^{(1)} - \u^{(2)}\|_{H^2L^2} \right) \nonumber \\
					&\hspace*{2cm}\times\left[\left( \|\u^{(1)}\|_{H^1H^2} + \|\u^{(1)}\|_{H^2L^2} \right) + \left( \|\u^{(2)}\|_{H^1H^2} + \|\u^{(2)}\|_{H^2L^2} \right)\right] \nonumber\\
					& \quad\leq C(T) \left(\|\u^{(1)}_h - \u^{(2)}_h\|_{X_\u} + C\left(\|\boldsymbol{\omega}^{(1)} - \boldsymbol{\omega}^{(2)}\|_{H^2(0,T)}\right) \right) \nonumber \\
					&\hspace*{2cm}\times\left[\left( \|\u^{(1)}_h\|_{X_\u} + C(\|\OMEGA^{(1)}\|_{H^2(0,T)}+1) \right) + \left( \|\u^{(2)}_h\|_{X_\u} + C(\|\OMEGA^{(2)}\|_{H^2(0,T)}+1) \right) \right] \nonumber \\
					& \quad \leq C(T) (R + 1) \left(\|(\u^{(1)}_h, \Q^{(1)}_h, \boldsymbol{\omega}^{(1)}, \V^{(1)}) - (\u^{(2)}_h, \Q^{(2)}_h, \boldsymbol{\omega}^{(2)}, \V^{(2)})\|_{X} \right). \nonumber 
				\end{align}
			Thus, we obtained
				\begin{eqnarray}
			&&\| P_{\sigma}(\nabla \cdot (\u^{(1)} \otimes \u^{(1)})) - P_{\sigma}(\nabla \cdot (\u^{(2)} \otimes \u^{(2)})) \|_{H^1L^2_\sigma} \nonumber \\
			&& \hspace*{3.5cm}\leq C(T,R) \left(\|(\u^{(1)}_h, \Q^{(1)}_h, \boldsymbol{\omega}^{(1)}, \V^{(1)}) - (\u^{(2)}_h, \Q^{(2)}_h, \boldsymbol{\omega}^{(2)}, \V^{(2)})\|_{X} \right).\nonumber
			\end{eqnarray}

			\medskip

		\noindent \textit{Part 2}: $\u \cdot \nabla \Q$.
				\begin{eqnarray}
					&&\|  \u^{(1)} \cdot \nabla \Q^{(1)} -  \u^{(2)} \cdot \nabla \Q^{(2)} \|_{H^1H^1} \nonumber \\
					&& \hspace*{3.5cm} \leq \| (\u^{(1)} - \u^{(2)}) \cdot \nabla \Q^{(1)} \|_{H^1H^1} + \| \u^{(2)} \cdot \nabla (\Q^{(1)} - \Q^{(2)}) \|_{H^1H^1}.
				\end{eqnarray}
			Apply the same arguments as in \textit{Part 1} of STEP 2, that is, apply \eqref{est_1_1_1}, : 
				\begin{align}
					& \| \u^{(1)}\cdot \nabla (\Q^{(1)} - \Q^{(2)})\|_{H^1H^1} + \| (\u^{(1)} - \u^{(2)})\cdot  \nabla \Q^{(2)} \|_{H^1H^1} \nonumber \\
					& \quad \leq C(T)\left( \|\u^{(1)}\|_{H^1H^2} + \|\u^{(1)}\|_{H^2L^2} \right) \left( \|\nabla \Q^{(1)}_h - \nabla \Q^{(2)}_{h}\|_{H^1H^2} + \|\nabla \Q^{(1)}_h - \nabla \Q^{(2)}_h\|_{H^2L^2} \right) \nonumber\\ 
					& \hspace*{1cm} + C(T)\left( \|\u^{(1)} - \u^{(2)}\|_{H^1H^2} + \|\u^{(1)} - \u^{(2)}\|_{H^2L^2} \right)\left( \|\nabla \Q^{(2)}\|_{H^1H^2} + \|\nabla \Q^{(2)}\|_{H^2L^2}\right) \nonumber \\
					& \quad \leq C(T)\left( \|\u^{(1)}\|_{H^1H^2} + \|\u^{(1)}\|_{H^2L^2} \right) \|\Q^{(1)}_h - \Q^{(2)}_{h}\|_{X_\Q} \nonumber\\ 
					& \hspace*{1cm} + C(T)\left( \|\u^{(1)} - \u^{(2)}\|_{H^1H^2} + \|\u^{(1)} - \u^{(2)}\|_{H^2L^2} \right)\left( \|\Q^{(2)}\|_{H^1H^3} + \|\Q^{(2)}\|_{H^2H^1}\right). 
				\end{align}
				Using \eqref{Lipshitz_u}, \eqref{Bound_u}, and \eqref{Q_i_bounded} one gets
				\begin{align}
					& C(T)\left( \|\u^{(1)}\|_{H^1H^2} + \|\u^{(1)}\|_{H^2L^2} \right) \|\Q^{(1)}_h - \Q^{(2)}_{h}\|_{X_\Q} \nonumber\\ 
					& \hspace*{1cm} + C(T)\left( \|\u^{(1)} - \u^{(2)}\|_{H^1H^2} + \|\u^{(1)} - \u^{(2)}\|_{H^2L^2} \right)\left( \|\Q^{(2)}\|_{H^1H^3} + \|\Q^{(2)}\|_{H^2H^1}\right) \nonumber\\
					& \quad \leq C(T)\left( \|\u^{(1)}_h\|_{X_\u} + C(\|\OMEGA^{(1)}\|_{H^2(0,T)}+1) \right) \left(\|\Q^{(1)}_h - \Q^{(2)}_{h}\|_{X_\Q} \right)\nonumber\\
					& \hspace*{2cm} + C(T)\left(\|\u^{(1)}_h - \u^{(2)}_h\|_{X_\u} + C\left(\|\boldsymbol{\omega}^{(1)} - \boldsymbol{\omega}^{(2)}\|_{H^2(0,T)}\right) \right) \left( \|\Q^{(2)}\|_{H^1H^3} + \|\Q^{(2)}\|_{H^2H^1}\right)\nonumber\\
					& \quad \leq C(T) (R + 1) \|(\u^{(1)}_h, \Q^{(1)}_h, \boldsymbol{\omega}^{(1)}, \V^{(1)}) - (\u^{(2)}_h, \Q^{(2)}_h, \boldsymbol{\omega}^{(2)}, \V^{(2)})\|_{X}.
				\end{align}
				Thus, we obtained 
				\begin{eqnarray}
				&&\|  \u^{(1)} \cdot \nabla \Q^{(1)} -  \u^{(2)} \cdot \nabla \Q^{(2)} \|_{H^1H^1} \nonumber \\
				&& \hspace*{3.5cm} \leq C(T,R)\|(\u^{(1)}_h, \Q^{(1)}_h, \boldsymbol{\omega}^{(1)}, \V^{(1)}) - (\u^{(2)}_h, \Q^{(2)}_h, \boldsymbol{\omega}^{(2)}, \V^{(2)})\|_{X}. \nonumber 
				\end{eqnarray}
				
				\medskip

				\noindent \textit{Part 3}: $\hat{\H}(\Q)$. Here, we need to show
				\begin{eqnarray}
				&&\nonumber \| \H(\Q^{(1)}) - \H(\Q^{(2)}) \| _{H^1H^1}  \leq C(T,R)\|(\u^{(1)}_h, \Q^{(1)}_h, \boldsymbol{\omega}^{(1)}, \V^{(1)}) - (\u^{(2)}_h, \Q^{(2)}_h, \boldsymbol{\omega}^{(2)}, \V^{(2)})\|_{X}. \nonumber 
				\end{eqnarray}
				This bound follows directly from the proof on \textit{Part 3} of STEP 2. 
			
		\medskip
			\noindent \textit{Part 4:} $\dfrac{\text{d}\V}{\text{d}t}$.
			\begin{eqnarray}
			&	\hspace{-40pt}\|\dfrac{\text{d}\V^{(1)}}{\text{d}t} - \dfrac{\text{d}\V^{(2)}}{\text{d}t}\|_{H^1(0,T)}& \leq \|\V^{(1)} - \V^{(2)}\|_{H^2(0,T)}\nonumber  \\ 
			&&\leq  C(T) (\|(\u^{(1)}_h, \Q^{(1)}_h, \V^{(1)}, \boldsymbol{\omega}^{(1)}) - (\u^{(2)}_h, \Q^{(2)}_h, \V^{(2)}, \boldsymbol{\omega}^{(2)})\|_{X}).
			\end{eqnarray}

			\medskip
			
			\noindent \textit{Part 5:} $F_{\text{ext}}(\Q,\Q_{\infty})$.
			We first note that for both definitions \eqref{2d_f_ext} and \eqref{3d_f_ext}, $F_{\text{ext}}(\Q,\Q_{\infty})$ is a quadratic function of $\Q$: 
			\begin{equation*}
			F_{\text{ext}}(\Q,\Q_{\infty})=\mathcal{B}(\Q,\Q)=\sum\limits_{k,l,m,n=1}^{d}b_{klmn}Q_{kl}Q_{mn},
			\end{equation*}
			where coefficients $b_{klmn}$ depend on $\Q_{\infty}$. Then use of the triangle inequality and \eqref{est_1_1_2_prime}: 
			\begin{eqnarray}
			&&\|F_{\text{ext}}(\Q^{(1)},\Q_{\infty})-F_{\text{ext}}(\Q^{(2)},\Q_{\infty})\|_{H^1H^1}\leq \|\mathcal{B}(\Q^{(1)}-\Q^{(2)},\Q^{(1)})\|_{H^1H^1}\nonumber\\
			&&\hspace{280pt}+\|\mathcal{B}(\Q^{(2)},\Q^{(1)}-\Q^{(2)})\|_{H^1H^1}\nonumber\\
			&&\hspace{170pt}\leq C(T)\|\Q^{(1)}-\Q^{(2)}\|_{X_\Q}(\|Q^{(1)}\|_{H^1H^1}+\|Q^{(2)}\|_{H^1H^1})\nonumber\\
			&&\hspace{170pt}\leq C(T,R)\|\Q^{(1)}_h-\Q^{(2)}_h\|_{X_\Q}.\nonumber
			\end{eqnarray}

			\medskip 
			
			\noindent \textit{Part 6:}  $\frac{1}{I} \int_{\partial \P} \x \times \sigma \boldsymbol{\nu} + \ell \,dS_\x$.
			
			Since $\sigma = \sigma_{\text{hydro}} + \sigma_{\text{ela}} $ we use \eqref{what_we_want_in_step_2_modified} and \eqref{needed_for_sigma_hydro}
			\begin{align}
				&\| \sigma(\nabla \u^{(1)},\Q^{(1)}) - \sigma(\nabla \u^{(1)},\Q^{(1)})\|_{H^1H^{1}} \nonumber\\
				& \hspace{2cm} \leq C(T, R) (\|(\u^{(1)}_h, \Q^{(1)}_h, \V^{(1)}, \boldsymbol{\omega}^{(1)}) - (\u^{(2)}_h, \Q^{(2)}_h, \V^{(2)}, \boldsymbol{\omega}^{(2)})\|_{X}. \nonumber
			\end{align}
			Using trace theorem, we get
			\begin{align}
				& \left\|\frac{1}{I} \int_{\partial \P} \x \times \sigma(\Q^{(1)}) \boldsymbol{\nu} dS_\x - \frac{1}{I} \int_{\partial \P} \x \times \sigma(\Q^{(2)}) \boldsymbol{\nu} dS_\x \right\|_{H^{1}(0,T)} 
				\leq  C\|\sigma(\Q^{(1)}) - \sigma(\Q^{(2)})\|_{H^1H^1}
				\nonumber \\
				&\hspace{100pt}\leq C(T, R) \|(\u^{(1)}_h, \Q^{(1)}_h, \V^{(1)}, \boldsymbol{\omega}^{(1)}) - (\u^{(2)}_h, \Q^{(2)}_h, \V^{(2)}, \boldsymbol{\omega}^{(2)})\|_{X}. \label{integral_balance_bound_1}
			\end{align}
		To estimate the term with $\ell$, recall its simplified form \eqref{simplified_ell}:
		\begin{eqnarray}
		&&\left\|\int\limits_{\partial \mathcal{P}}\ell (\Q^{(1)})\,\text{d}S_x-\int\limits_{\partial \mathcal{P}}\ell(\Q^{(2)})\,\text{d}S_x\right\|_{H^1(0,T)}
		\leq C \|\Q^{(1)}-\Q^{(2)}\|_{H^1(0,T;L^1(\partial \mathcal{P}))}\nonumber \\
		&&\hspace{200pt}\leq C\|\Q^{(1)}_h-\Q^{(2)}_h\|_{H^1H^1}\nonumber\\
		&&\hspace{200pt}\leq C(T)\|\Q^{(1)}_h-\Q^{(2)}_h\|_{H^2H^1}\nonumber\\
		&&\hspace{80pt}\leq C(T) \|(\u^{(1)}_h, \Q^{(1)}_h, \V^{(1)}, \boldsymbol{\omega}^{(1)}) - (\u^{(2)}_h, \Q^{(2)}_h, \V^{(2)}, \boldsymbol{\omega}^{(2)})\|_{X}.\label{integral_balance_bound_2}
		\end{eqnarray}
			
		\medskip
		
		\noindent \textit{Part 7}: $\frac{1}{m} \int_{\partial \P} \sigma \boldsymbol{\nu} dS_\x$. The same argument for \textit{Part 6} also works for $\frac{1}{m} \int_{\partial \P} \sigma \boldsymbol{\nu} dS_\x$.

			\medskip
		\noindent \textit{Part 8}: $\partial_t \u_{\text{os}}$.
		Using \eqref{stability_u_os_2.1}, we have
			\begin{align}
				\|\partial_t \u_{\text{os}}^{(1)} - \partial_t \u_{\text{os}}^{(2)}\|_{H^{1}L^{2}_\sigma} \leq C(T) \|\boldsymbol{\omega}^{(1)} - \boldsymbol{\omega}^{(2)}\|_{H^{2}(0,T)}. 
			\end{align}
			
			\medskip 
			
Now, collecting all bounds from STEPS 2 - 4, we have \eqref{lipshitz_S} and thus Lemma~\ref{prop_lipshitz_J} is proven. 
\end{proof}

\subsection{Proof of Theorem~\ref{theorem-finite-time} (local in time existence)} \label{local_in_time_theorem}
\noindent In this section we prove the well-posedness of the time-dependent
problem. The equation  \eqref{lu_equal_to_ju} can be rewritten as $\mathcal{K}\mathcal{U}_h=\mathcal{U}_h$ where $\mathcal{K}:=\mathcal{L}^{-1}\mathcal{J}:X\to X$ and $\mathcal{U}_h=(\u_h,\Q_h,\OMEGA,\boldsymbol{V})$. The inverse linear operator $\cL^{-1}$ is bounded, as stated in the following proposition.
\begin{proposition}\it \label{prop_solvability_L}
	For all $(\boldsymbol{f}_\u, \boldsymbol{f}_\Q, \boldsymbol{f}_{\OMEGA}, \boldsymbol{f}_\V) \in Y$, and time $T \in (0, 1]$, linear system 
	\begin{eqnarray}
		\cL(\u_h, \Q_h, \boldsymbol{\omega}, \V) = (\boldsymbol{f}_\u, \boldsymbol{f}_\Q, \boldsymbol{f}_{\OMEGA}, \boldsymbol{f}_\V) \label{linear_L}
	\end{eqnarray}
	has a unique solution such that $\u_h|_{t=0}=0$, $\Q_h|_{t=0}=0$, $\OMEGA|_{t=0}=\boldsymbol{V}|_{t=0}=0$ and 
	\begin{eqnarray}
		\|(\u_h, \Q_h, \boldsymbol{\omega}, \V)\|_{X} \leq C\|(\boldsymbol{f}_\u, \boldsymbol{f}_\Q, \boldsymbol{f}_{\OMEGA}, \boldsymbol{f}_\V)\|_{Y}, \label{solvability}
	\end{eqnarray}
	where the constant $C$ is independent of time $T$ and choice of $(\boldsymbol{f}_\u, \boldsymbol{f}_\Q, \boldsymbol{f}_{\OMEGA}, \boldsymbol{f}_\V)$. 
\end{proposition}

To prove this proposition, one can follow \cite{AbeDolLiu2013}. Specifically, for the first two components, $\u_h$ and $\Q_h$, of system \eqref{linear_L}, we adapt the proof from \cite[Proposition 4.2]{AbeDolLiu2013}. For the last two components, which are not present in \cite{AbeDolLiu2013}, the statement naturally follows from the classical ODE theory.

\smallskip 

Next, 
according to propositions \ref{prop_lipshitz_J} and \ref{prop_solvability_L}, we have that
\begin{eqnarray}
	&&\|\mathcal{K}(\u_h^{(1)}, \Q_h^{(1)}, \boldsymbol{\omega}^{(1)}, \V^{(1)}) - \mathcal{K}(\u_h^{(2)}, \Q_h^{(2)}, \boldsymbol{\omega}^{(2)}, \V^{(2)})\|_{X} \nonumber\\
	&&\hspace*{2cm}\leq  C(T) \|(\u_h^{(1)}, \Q_h^{(1)}, \boldsymbol{\omega}^{(1)}, \V^{(1)}) - (\u_h^{(2)}, \Q_h^{(2)}, \boldsymbol{\omega}^{(2)}, \V^{(2)})\|_{X}.
\end{eqnarray} 
Recall that $C(T)$ depends on $T$ in such a way that $C(T) \rightarrow 0$ when $T \rightarrow 0$. Choose $T$ such that $C(T) < 1$. Then using Banach's fixed point theorem, we obtain that there exists a unique fixed point $(\u_h, \Q_h, \boldsymbol{\omega}, \V)$ of operator $\mathcal{K}$. Next, define $\u_{\text{os}}$ via \eqref{def_off_set_u_1}-\eqref{def_off_set_Q_3}. Finally, the tuple $(\u_h + \u_{\text{os}}, \Q_h + \Q_{\text{os}}, \boldsymbol{\omega}, \V)$ is a solution to the original system \eqref{eq:time_lc_1_rewrite}, \eqref{eq:time_lc_2}-\eqref{eq:time_lc_6} with force and torque balances \eqref{eq:time_ft_1}, \eqref{eq:time_ft_2_rewrite} for $0\leq t\leq T$. Theorem~\ref{theorem-finite-time} is proven.

\section{Homogenization: two-scale expansion}
\label{sec:homo_expansion}
In this section, we will perform formal two-scale expansion for \eqref{eq:pre_homo_Q_eq_re}-\eqref{eq:pre_homo_u_bc_re}. To this end, we introduce fast variable $\y=\ve^{-1}\x$ and represent the unknowns as  
\begin{equation}
\left\{\begin{array}{rcl}
\u(\x;\ve)&=& \bar{\u}(\x,\y)=\u^{(0)}(\x,\y)+\ve \u^{(1)}(\x,\y)+\cdots\\
p(\x;\ve)&=&\bar{p}(\x,\y)= p^{(0)}(\x,\y)+\ve p^{(1)}(\x,\y)+\cdots \\
\Q(\x;\ve)&=&\bar{\Q}(\x,\y)=\Q^{(0)}(\x,\y)+\ve \Q^{(1)}(\x,\y)+\cdots
\end{array}\right.\label{two-scale-ansatz}
\end{equation}
We will frequently use following identities for $f(\x,\ve)=\bar{f}(\x,\y)$ with $\y=\ve^{-1}\x$: 
\begin{eqnarray}
\nabla f&=&\nabla_{\x} \bar{f}+\ve^{-1}\nabla_{\y}\bar{f},\\
\Delta f&=&\Delta_{\x} \bar{f} +2\ve^{-1}\nabla_\y\cdot \nabla_{\x}\bar{f}+\ve^{-2}\Delta_{\y}\bar{f}.
\end{eqnarray}

\smallskip 

The derivation of the homogenized limit consists of the following steps.

\noindent{STEP 1.} {\it Show that $\u^{(0)}=0$}. Substitute two-scale representations \eqref{two-scale-ansatz} for $\u,\Q,$ and $p$ into \eqref{eq:pre_homo_u_eq_re} and $\nabla\cdot \u=0$. We get that at level $\ve^{-1}$:
\begin{equation*}
-\tilde{\eta}\Delta_{\y}\u^{(0)}+\nabla_{\y}p^{(0)}=0 \text{ and }\nabla_{\y}\cdot\u^{(0)}=0, 
\end{equation*}
with the boundary condition $\u^{(0)}=0$ on $\partial\mathcal{P}_{\ve}$. Thus, 
we can conclude that $\u^{(0)}=0$. 

\medskip 

\noindent{STEP 2.} {\it Find an equation for $\Q^{(0)}$}. In this step we expand equations \eqref{eq:pre_homo_Q_eq_re} and \eqref{eq:pre_homo_Q_bc_re} in $\ve$. To this end, we write the weak formulation of these two equations for 
arbitrary test function $\Phi \in H^1(\Omega_\ve;\mathbb{R}^{d\times d})$: 
\begin{eqnarray}
&&- \gamma\ve \int\limits_{\Omega_\ve} \nabla \Q \cdot \nabla \Phi \,\text{d}\x+\ve\tilde{W}\int\limits_{\partial \mathcal{P}_\ve}(\Q_{\text{pref}}-\Q):\Phi\,\text{d}S_{\x}\nonumber\\
&& \hspace{80 pt} + \int\limits_{\Omega_{\ve}}\left[\tilde{a}\Q-\tilde{c}\Q\text{Tr}(\Q^2)\right]:\Phi\;\text{d}\x+
\int\limits_{\Omega_\ve}S(\nabla \tilde{\u},\Q):\Phi\,\text{d}\x\nonumber \\
&& \hspace{90pt}- \int\limits_{\Omega_\ve}(\tilde{\u}\cdot \nabla)\Q:\Phi\,\text{d}\x+\tilde{\zeta}\int\limits_{\Omega_\ve}\tilde{F}_{\text{ext}}:\Phi\,\text{d}\x=\int\limits_{\Omega_\ve}{\bf G}:\Phi\,\text{d}\x.\label{Q_prelim_weak_formulation}	 
\end{eqnarray}
Introduce $\Omega_1=\ve^{-1}\Omega_\ve$ and $\mathcal{P}_1=\ve^{-1}\mathcal{P}_\ve$. We now consider two-scale representation for the test function $\Phi$:
\begin{equation*}
\Phi (\x;\ve)=\bar{\Phi}(\x,\y) =\Phi^{(0)}(\x, \y) + \ve \Phi^{(1)}(\x, \y) + \cdots.
\end{equation*}
\noindent Rewrite the first integral in \eqref{Q_prelim_weak_formulation} in domain $\Omega_1$:
\begin{eqnarray}
&&\hspace{-25pt}-\gamma\ve\int\limits_{\Omega_\ve}\nabla \Q \cdot \nabla \Phi \,\text{d}\x = -\gamma\ve^{1+d}\int\limits_{\Omega_1}\nabla \Q(\ve \y)\cdot \nabla \Phi(\ve \y)  \,\text{d}\y\nonumber \\
&&\hspace{10pt}= -\gamma\ve^{1+d}\int\limits_{\Omega_1}\left[\nabla_{\x}\bar{\Q}(\ve \y,\y)+\ve^{-1}\nabla_{\y} \bar{\Q}(\ve \y,\y)\right]\cdot \left[\nabla_{\x} \bar{\Phi}(\ve \y,\y)+\ve^{-1}\nabla_{\y} \bar{\Phi}(\ve \y,\y)\right]  \,\text{d}\y\nonumber \\
&& \hspace{10pt}= -\Gamma K \ve^{1+d}\left\{\ve^{-2}\int\limits_{\Omega_1} \nabla_{\y}\Q^{(0)}\cdot \nabla_{\y} \Phi^{(0)}\,\text{d}\y+ 
\right.\nonumber \\
&&\hspace{15pt}\left.+\ve^{-1}\int\limits_{\Omega_1}\left[\nabla_{\x}\Q^{(0)}+\nabla_{\y}\Q^{(1)}\right]\cdot\nabla \Phi^{(0)}+\nabla \Q^{(0)}\cdot \left[\nabla_{\x}\Phi^{(0)}+\nabla_{\y}\Phi^{(1)}\right]\,\text{d}\y
+\cdots\right\}\hspace{-3pt}.
\end{eqnarray}

Expanding analogously other terms in \eqref{Q_prelim_weak_formulation} and using  that $\u^{(0)}=0$ we get at level $\ve^{d-1}$:
\begin{equation}
	 \int\limits_{{\Omega_1}} \nabla_\y \Q^{(0)}\cdot\nabla_\y\Phi^{(0)}\, \text{d}\y = 0, \nonumber
\end{equation}
which implies, together with periodicity in $\y$, that $\Q^{(0)}(\x, \y) = \Q^{(0)}(\x)$.

At level $\ve^{d}$, accounting for $\nabla_{\y}\Q^{(0)}=0$, we have
\begin{eqnarray}
	&&-\gamma\int_{{\Omega_1}} \left[\nabla_\x \Q^{(0)} + \nabla_\y \Q^{(1)}\right]\cdot \nabla_\y \Phi^{(0)}\, \text{d}\y +  \tilde{W} \int_{\partial {\P_1}} (\Q_{\text{pref}} - \Q^{(0)}):\Phi_0\, \text{d}S_\y \nonumber\\
		&&\hspace{2.5cm}+ \int\limits_{\Omega_1} \left[-\tilde{a}\Q^{(0)}+\tilde{c}\Q^{(0)}\text{Tr}((\Q^{(0)})^2)\right]:\Phi^{(0)}\, \text{d}\y \nonumber\\
		&& \hspace{2.5cm} + \int\limits_{\Omega_1} S(\nabla_\y \u^{(1)}, \Q^{(0)})\cdot\Phi^{(0)} \,\text{d}\y +  \tilde{\zeta}\int\limits_{\Omega_1}\tilde{F}_{\text{ext}}(\Q^{(0)},\Q_{\infty}):\Phi^{(0)}\,\text{d}\y \nonumber\\
		&& \hspace{2cm} = \int\limits_{\Omega_1} {\bf G}(\x): \Phi^{(0)}\, \text{d}\y. \label{Q_1_weak_form}
\end{eqnarray}
Note that the above integral relation is the weak formulation for the following boundary-value problem: 
\begin{eqnarray}
\left\{ \begin{array}{ll}
-\gamma\Delta_\y \Q^{(1)} = \boldsymbol{f}_1, &\y \text{ in } {\Pi_1}, \\ & \\
\gamma(\nabla_\y \Q^{(1)}) \cdot \boldsymbol{\nu}_\y = \boldsymbol{g}_1, & \y \text{ on } \partial \P_{1}.
\end{array}
\right.\label{eq:laplace_Q_1}
\end{eqnarray}
Here 
\begin{eqnarray}
\boldsymbol{f}_1&=&\gamma\Delta_{\y\x} \Q^{(0)}  - \tilde{a}\Q^{(0)} + \tilde{c}\Q^{(0)}\text{Tr}((\Q^{(0)})^2) + S(\nabla_\y \u^{(1)}, \Q^{(0)})+\tilde{\zeta}\tilde{F}_{\text{ext}}(\Q^{(0)},\Q_{\infty}) - {\bf G}(\x),\nonumber \\ 
\boldsymbol{g}_1&=&-\gamma\nabla_\x \Q^{(0)}\cdot \boldsymbol{\nu}_\y + \tilde{W}(\Q_{\text{pref}} - \Q^{(0)}),\nonumber
\end{eqnarray}
 and $\Delta_{\y\x}\boldsymbol{h}=\nabla_\y\cdot \nabla_{\x}\boldsymbol{h}$ for arbitrary $\boldsymbol{h}$.

Next, we have the solvability condition for \eqref{eq:laplace_Q_1} given as:
\begin{equation}
\int\limits_{\partial\mathcal{P}_1}\boldsymbol{g}_1\,\text{d}S_\y=\int\limits_{\Omega_1}\boldsymbol{f}_1\,\text{d}\y. \label{solvability-general}
\end{equation}
To evaluate the right-hand side, we use the fact that $\Q^{(0)}$ is 
independent of $\y$ and 
\begin{equation}
\int\limits_{\Omega_1}\nabla_{\y}\u\,\text{d}\y=\int\limits_{\partial \mathcal{P}_1}u_{\text{sq}}\boldsymbol{\tau}\boldsymbol{\nu}^{\text{T}}\,\text{d}S_\y. 
\end{equation}
Hence, we have
\begin{eqnarray}
\int\limits_{\Omega_1}\boldsymbol{f}_1\,\text{d}\y
&=&|\Omega_1|\left( - \tilde{a}\Q^{(0)} + \tilde{c}\Q^{(0)}\text{Tr}((\Q^{(0)})^2)+\tilde{\zeta}\tilde{F}_{\text{ext}}(\Q^{(0)},\Q_{\infty}) - {\bf G}(\x)\right)\nonumber \\
&&
+S(\int\limits_{\partial \mathcal{P}_1}u_{\text{sq}}\boldsymbol{\tau}\boldsymbol{\nu}^{\text{T}}\,\text{d}S_\y,\Q^{(0)}),\label{simplified_Q_1_1}\\
\int\limits_{\partial\mathcal{P}_1}\boldsymbol{g}_1\,\text{d}S_\y
&=&
\tilde{W}\int\limits_{\partial\mathcal{P}_1}\Q_{\text{pref}}\,\text{d}S_\y-\tilde{W}|\partial \mathcal{P}_1|\Q^{(0)}.\label{simplified_Q_1_2}
\end{eqnarray}

Substituting \eqref{simplified_Q_1_1} and \eqref{simplified_Q_1_2} into \eqref{solvability-general} we get the equation for $\Q^{(0)}$:
\begin{eqnarray}
&&\hspace{-20pt}-\left[\tilde{a}-\dfrac{\tilde{W}|\partial\mathcal{P}_1|}{|{\Omega_1}|}\right]\Q^{(0)}+\tilde{c}\Q^{(0)}\text{Tr}((\Q^{(0)})^2)\nonumber \\
&& \hspace{100pt}+\tilde{\zeta}\tilde{F}_{\text{ext}}(\Q^{(0)},\Q_{\infty})+S({\bf G}_{\text{sq}},\Q^{(0)})={\bf G}(x)-\dfrac{\tilde{W}}{|\Omega_1|}\overline{\Q}_{\text{pref}}.\label{eq_for_Q_h_appendix}
\end{eqnarray}
Here, we denote ${\bf G}_{\text{sq}}=\int\limits_{\partial \mathcal{P}_1}u_{\text{sq}}\boldsymbol{\tau}\boldsymbol{\nu}^{\text{T}}\,\text{d}S_\y$ and $\overline{\Q}_{\text{pref}}=\int\limits_{\partial\mathcal{P}_1}\Q_{\text{pref}}\,\text{d}S_\y$. The function $\Q^{(0)}$ is the limit of $\Q$ as $\ve \to 0$, thus, the algebraic equation \eqref{eq_for_Q_h_appendix} determines $\Q^{(h)}=\Q^{(0)}$.

\medskip 

\noindent{STEP 3.} {\it Find an equation for $\u^{(1)}$}. At level $\ve^0$ in the expansion of \eqref{eq:pre_homo_u_eq_re}, accounting for that $\u^{(0)}=0$, $\Q^{(0)}$ is independent of $\y$, and 
\begin{eqnarray*}
\ve^2\kappa\nabla\cdot (\nabla {\Q} \astrosun \nabla {\Q} + {\Q}\Delta {\Q} - \Delta {\Q} {\Q})=\ve^0\kappa \nabla_{\y}\cdot\left({\Q}^{(0)}\Delta_{\y} {\Q}^{(1)} - \Delta_{\y} {\Q}^{(1)} {\Q}^{(0)}\right) +\left[\begin{array}{c}\text{higher}\\\text{order}\\\text{terms}\end{array}\right]
\end{eqnarray*}
we get 
\begin{eqnarray}
&&\hspace{-35pt}-\tilde{\eta}\Delta_{\y}\u^{(1)}+\nabla_{\x}p^{(0)} +\nabla_{\y}p^{(1)}=\kappa \nabla_{\y}\cdot\left({\Q}^{(0)}\Delta_{\y} {\Q}^{(1)} - \Delta_{\y} {\Q}^{(1)} {\Q}^{(0)}\right)+{\bf F}(\x), \, \y\text{ in }\Omega_1,\label{eq_for_u_1}\\
&&\hspace{-35pt} \u^{(1)}=u_{\text{sq}}\boldsymbol{\tau}, \, \y\text{ on }\partial\mathcal{P}_1.
\end{eqnarray}

Next, we aim to remove $\Q^{(1)}$ from the right-hand side of \eqref{eq_for_u_1}. To this end, we notice that due to \eqref{eq:laplace_Q_1} we have that $-\gamma \Delta_{\y}\Q^{(1)}=\boldsymbol{f}_1$ and 
\begin{eqnarray}
&&\Delta_{\y\x}\Q^{(0)}=0\text{ and }
\Q^{(0)},\,\Q^{(0)}\text{Tr}(\Q^{(0)}),\,\tilde{\zeta}\tilde{F}_{\text{ext}}(\Q^{(0)},\Q_{\infty}),\,{\bf G}(\x)\text{ are independent of $\y$.}\nonumber
\end{eqnarray}
Thus, we can rewrite \eqref{eq_for_u_1} as 
\begin{eqnarray}
&&\hspace{-30pt}-\tilde{\eta}\Delta_{\y}\u^{(1)}+\nabla_{\x}p^{(0)} +\nabla_{\y}p^{(1)}\nonumber \\
&& \hspace{20pt}=-\kappa \gamma^{-1} \nabla_{\y}\cdot\left({\Q}^{(0)}S(\nabla_\y \u^{(1)}, \Q^{(0)}) - S(\nabla_\y \u^{(1)}, \Q^{(0)}) {\Q}^{(0)}\right)+{\bf F}(\x).
\end{eqnarray}
The above can be rewritten in component-wise form as 
\begin{equation}
\sum\limits_{m,j,l=1}^{d}\eta_{klmj}u_{m,jl}^{(1)}+\partial_{x_k}p^{(0)}+\partial_{y_k}p^{(1)}=F_{k}(\x), \quad k=1,..,d,\label{u1-component-wise}
\end{equation}
where 
\begin{eqnarray}
u_{m,jl}^{(1)}
&=&\dfrac{\partial^{2}u_{m}^{(1)}}{\partial y_{j}\partial y_{l}}\nonumber\\
\text{and}\quad
\eta_{klmj}&=&-\tilde{\eta}\delta_{km}\delta_{jl}+ \dfrac{\kappa\xi}{2\gamma}\sum\limits_{n=1}^{d}\left[Q_{kn}^{(0)}Q_{nm}^{(0)}\delta_{jl}-Q_{mn}^{(0)}Q_{nl}^{(0)}\delta_{jk}-Q_{jn}^{(0)}Q_{nl}^{(0)}\delta_{km}\right]\nonumber\\
&&+\dfrac{\kappa\xi}{d\gamma}\left[Q_{km}^{(0)}\delta_{jl}-Q_{ml}^{(0)}\delta_{jk}-Q_{jl}^{(0)}\delta_{km}\right]+\dfrac{\kappa}{\gamma}\left[Q_{kj}^{(0)}Q_{ml}^{(0)}-Q_{km}^{(0)}Q_{jl}^{(0)}\right]\nonumber\\
&&+\dfrac{\kappa}{\gamma}\sum\limits_{n=1}^{d}\left[Q_{kn}^{(0)}Q_{nm}^{(0)}\delta_{jl}-Q_{mn}^{(0)}Q_{nl}^{(0)}\delta_{jk}+Q_{jn}^{(0)}Q_{nl}^{(0)}\delta_{km}\right].
\end{eqnarray}
Next, rewrite \eqref{u1-component-wise} in a vectorial form: 
\begin{equation}
 -\eta(\Q^{(0)})\nabla_\y^2 \u^{(1)} + \nabla_\y p^{(1)} = {\bf F}(\x) - \nabla_\x p^{(0)}. 
\end{equation}
Taking into account the boundary condition $\u^{(1)}(\x,\y)=\tilde{u}_{\text{sq}}\boldsymbol{\tau}$ on $\mathcal{P}_1$, we obtain the following representation for $\u^{(1)}$:
\begin{equation}
\u^{(1)}=\mathcal{A}_{\eta(\Q^{(0)})}(\y)\left[{\bf F}(\x)-\nabla_{\x}p^{(0)}\right]+\overline{\u}_{\text{sq}},
\end{equation}
where $\mathcal{A}_{\eta(\Q^{(0)})}(\y)$ is a $\y$-dependent $d\times d$ matrix such that $\u(\y)=\mathcal{A}_{\eta(\Q^{(0)})}(\y)\boldsymbol{e}_i$  ($\boldsymbol{e}_i$ is the $i$th basis vector) is the solution of the following cell problem:
\begin{equation}
\left\{
\begin{array}{l}
-\eta(\Q^{(0)})\nabla_\y^2 \u + \nabla_\y p = \boldsymbol{e}_i, \text{ in } \Omega_1,\\
\nabla\cdot \u=0, \\ 
\u = 0, \text{ on }\mathcal{P}_1, \\ 
\u \text{ is 2-periodic}.
\end{array}
\right.
\end{equation}
The term $\overline{\u}_{\text{sq}}$ is defined as the solution of 
\begin{equation}
\left\{
\begin{array}{l}
-\eta(\Q^{(0)})\nabla_\y^2 \overline{\u}_{\text{sq}} + \nabla_\y p = 0, \text{ in } \Omega_1,\\
\nabla\cdot \overline{\u}_{\text{sq}}=0, \\ 
\overline{\u}_{\text{sq}} = \tilde{u}_{\text{sq}}\boldsymbol{\tau}, \text{ on }\mathcal{P}_1, \\ 
\overline{\u}_{\text{sq}} \text{ is 2-periodic}.
\end{array}
\right.
\end{equation}
Finally, we define the homogenized function $\u^{(h)}$ by averaging $\u^{(1)}$ 
and using the fact that $\Q^{(h)}=\Q^{(0)}$:
\begin{eqnarray}
\u^{(h)} = \mathcal{B}_{\eta(\Q^{(h)})} \left[{\bf F}({\x}) - \nabla_\x p\right]  + \fint_{{\Omega}_1} \bar{\u}_{\text{sq}}\,\text{d}\y,
\label{u_represent_Darcy}
\end{eqnarray} 
where $\mathcal{B}_{\eta} = \fint_{\Omega_1} \mathcal{A}_{\eta}(y)\, \text{d}\y$ and the pressure $p$ can be found from the divergence-free condition $\nabla\cdot\u=0$. 

To conclude, we have derived a system of homogenized equations in the form 
of \eqref{eq_for_Q_h_appendix}, an algebraic equation for $\Q$ 
and \eqref{u_represent_Darcy} for $\u$ in the form
of Darcy's law.

\bigskip 

\bigskip 

\section{Concluding remarks}

In this work, we initiated the theoretical justification of the active microswimmer model, also known as a squirmer, in a liquid crystal. This model has been recently developed to explore a non-trivial response of the microswimmer to surrounding environment possessing a liquid crystalline structure. 
As computational studies \cite{LinWurStr2017,chi2020surface} clearly show that the squirmer eventually converges to an equilibrium, both time-dependent solutions and steady states are important and were in the focus of our work. In investigating well-posedness of the corresponding equations, we started with combining techniques from the two fields, the squirmer in Newtonian fluids and the Beris-Edwards model of liquid crystal. However, such a combination is not straightforward. As explained in Remark~\ref{remark:dissipation}, one of the main difficulties, besides that our model is complex and highly nonlinear, is that it is not dissipative: there is a permanent energy input (not necessarily constant) 
coming from the activity of the squirmer. It makes application of a priori energy bounds established for the Beris-Edwards model not possible here. Therefore, the considered model requires novel approaches for its analysis. 
For the steady state problem, using suitable offset functions, 
we first proved the existence of a steady state for a truncated system via 
Galerkin approximations and careful energy bounds using specific properties of the Beris-Edwards system (see, e.g., \eqref{a_priori_bp3}), 
and then extended the well-posedness from the truncated system to the original 
one using the $L^{\infty}$ result formulated in Lemma~\ref{L_inf_Q}. 
For the time-dependent problem, in order to exploit the contraction mapping 
principle, we considered higher regularity solutions (instead of weak ones) 
which allowed us to obtain all the necessary bounds including the one for 
integrals where activity of squirmer enters as well as the force and torque 
balances for the squirmer (see, e.g., \eqref{integral_balance_bound_1} \& \eqref{integral_balance_bound_2}). 
Periodic settings, in which we considered our model, allowed us to pose a question of homogenization limit which would be a model describing a colony of synchronously moving squirmers. We found a scaling which, on the one hand, is consistent with experimental data (see Appendix~\ref{sec:homo_rescaling}) and, on the other hand, allows for a non-trivial two-scale expansions so that the homogenized limit takes the form of Darcy's law perturbed by an algebraic expression for the liquid crystal order parameter (see equations \eqref{eq_for_Q_h_appendix} and \eqref{u_represent_Darcy}).     

Natural extensions of our work are:  
\begin{itemize}
	\item [{\it (i)}] {\it Stability analysis of steady states.} Namely, we would like to find conditions on parameters when a steady state corresponding to swimming either parallel or perpendicular to the liquid crystal is stable. This analytical result will be compared with the main observation from \cite{chi2020surface} on bifurcations with respect to anchoring strength parameter $W$. It would be also important to show that there is no steady state other than corresponding to swimming parallel or perpendicular to the liquid crystal.
	
	\item [{\it (ii)}] {\it Force-velocity relation for steady swimming.} Though in the squirmer's frame and periodic conditions, squirmer's velocity is not well-defined, we can consider the so-called superficial velocity \cite{zick1982speres} $\boldsymbol{V}=-\fint_{\Omega}\u\,\text{d}\x$, which can be understood as the velocity of the squirmer with respect to the surrounding flow, and show how it depends on propulsion force entering the problem via $u_{\text{sq}}$. Specifically, given the profile of the active slip velocity $u_{\text{sq}}$ (with all other physical parameters fixed), what is the resulting velocity $\boldsymbol{V}$? This question is related to the evaluation of the squirmer efficiency in Stokes fluid as a function of  $u_{\text{sq}}$ \cite{michelin2010efficiency,guo2021optimal}. 
	
	\item [{\it (iii)}] {\it Rigorous justification of the homogenization limit.} We plan to justify the two-scale limit formally derived in Section~\ref{sec:homo_expansion} and in more general stochastic settings based on techniques developed for Newtonian fluids \cite{duerinckx2021corrector, duerinckx2022effective, duerinckx2021continuum, duerinckx2020einstein, duerinckx2022sedimentation}. 
\end{itemize}

\section*{Acknowledgment}
\addcontentsline{toc}{section}{Acknowledgement}
The work of L. Berlyand was supported by NSF grant DMS-2005262 and the work of H.~Chi was partially supported by NSF grant DMS-2005262.


\bibliographystyle{ieeetr}
\bibliography{llc_analytics}

\bigskip 
 
\appendix
\noindent{\huge \bf Appendix}

\section{Proof of Lemma~\ref{lemma_existence_of_q_m}}
\label{sec:aux_estimate}

\begin{proposition}[A Poincar\'{e}-type estimate]
	There exists $c_P>0$ such that
	\begin{equation}
	\|\Q\|^{2}_{L^2(\Omega)}\leq c_P \left(\dfrac{K}{2}\|\nabla \Q\|^2_{L^2(\Omega)} + \dfrac{W}{2}\|\Q\|^2_{L^2(\partial \P_{\text{st}})}\right). \label{poincare}
	\end{equation}
	for all $\Q\in H^1_{\text{per}}(\Omega)$.
\end{proposition}
\begin{proof}
Here, we first show that there exists $\hat{c}_P>0$ such that for arbitrary $\Q\in H^1_{\text{per}}(\Omega)$ the following inequality holds
\begin{equation}\label{poincare_without_coefs}
\|\Q\|^2_{L^2(\Omega)}\leq \hat{c}_P(\|\nabla \Q\|_{L^2(\Omega)}^2+\|\Q\|^2_{L^2(\mathcal{P}_{\text{st}})}).
\end{equation}
By contradiction, we assume that there exists a sequence $\{\Q_n\}_{n=1}^{\infty}$: 
\begin{equation}
\|\Q_n\|^2_{L^2(\Omega)}=1 \text{ and }\|\nabla \Q_n\|_{L^2(\Omega)}^2+\|\Q_{n}\|^2_{L^2(\mathcal{P}_{\text{st}})}\leq \dfrac{1}{n}.\label{contradiction-assumption}
\end{equation} 
From boundedness of $\Q_n$ in $L^2(\Omega)$ it follows that the sequence $\{\Q_n\}_{n=1}^{\infty}$ possesses a weakly converging sub-sequence in $L^2(\Omega)$. Consider any such weakly converging sub-sequence $\{\Q_{n_k}\}_{k=1}^{\infty}$, $\Q_{n_k}\rightharpoonup \Q^{*}$ and a function $\psi\in H^1_{\text{per}}(\Omega)$:
\begin{equation*}
\int\limits_{\Omega}\Q_{n_k}\partial_{x_i}\psi\,\text{d}x=
\int\limits_{\partial\mathcal{P}_{\text{st}}}\Q_{n_k}\psi\nu_i\,\text{d}S_x-\int\limits_{\Omega}\partial_{x_i}\Q_{n_k}\psi\,\text{d}x \to 0,\quad 1\leq i\leq d. 
\end{equation*} 
The convergence to $0$ follows from strong convergences of $\Q_{n_k}$ in $L^2(\partial\mathcal{P}_{\text{st}})$ and $\nabla\Q_{n_k}$ in $L^2(\Omega)$ which in turn follow from \eqref{contradiction-assumption}. Then 
\begin{equation*}
\int\limits_{\Omega}\Q^{*}\partial_{x_i}\psi\,\text{d}x=0 \quad \forall \psi \in H^1_{\text{per}}(\Omega).
\end{equation*}
Using integration by parts we get 
\begin{equation*}
\int\limits_{\partial\mathcal{P}_{\text{st}}}\Q^{*}\psi\nu_i\,\text{d}S_x-\int\limits_{\Omega}\partial_{x_i}\Q^{*}\psi\,\text{d}x=0\quad \forall \psi \in H^1_{\text{per}}(\Omega).
\end{equation*} 
By taking first $\psi \in C^{\infty}_{c}(\Omega)$ we get $\partial_{x_i}\Q^{*}\equiv 0$ so the second integral in the equality above vanishes. Next, we get that the first integral is zero as well by taking $\psi$ with various non-zero traces. We conclude $\Q^{*}\equiv 0$. Moreover, since for any weakly converging sub-sequence the limit is $0$, then entire sequence $\{\Q_n\}_{n=1}^{\infty}$ weakly converges to $0$. 

Note that $H^1_{\text{per}}(\Omega)\subset H^1(\Omega)$ and thus $H^1_{\text{per}}(\Omega)$ is compactly embedded in $L^2(\Omega)$. Hence, $\Q_{n}$ is strongly converging in $L^2(\Omega)$ and since the weak limit is $0$, we conclude $\Q_n\to 0$ strongly in $L^2(\Omega)$. However, it contradicts to \eqref{contradiction-assumption} since it implies that if $\Q_n\to \Q^*$ strongly in $L^2(\Omega)$, then $\|\Q^*\|_{L^2(\Omega)}^2=1$. Thus, inequality \eqref{poincare_without_coefs} is shown.

Finally, to prove \eqref{poincare} take $c_P=\dfrac{2\hat{c}_P}{\min\{K,W\}}$.
\end{proof}

\smallskip 
 
Next we turn to the proof of Lemma~\ref{lemma_existence_of_q_m}. We note that an equivalent way to define $\Q_m$ from \eqref{aux_elliptic} is via minimization of the following energy functional:
\begin{equation}
\mathcal{E}(\Q)=\int\limits_{\Omega}\dfrac{K}{2}|\nabla \Q|^2 + \hat{\mathcal{F}}_M(\Q)\,\text{d}x+\dfrac{W}{2}\int\limits_{\partial \mathcal{P}_{\text{st}}}|\Q_{\text{pref}}-\Q|^2 \,\text{d}S_x + \int\limits_{\Omega} \H_m: \Q \,\text{d}x
\end{equation} 
among functions $\Q\in H^1_{\text{per}}(\Omega)$. The minimizer $\Q$ of the energy functional $\mathcal{E}(\Q)$ exists.  

From \eqref{poincare}, the Cauchy inequality, and $\|\Q-\Q_{\text{pref}}\|^2\geq \dfrac{1}{2}\|\Q\|^2-\|\Q_{\text{pref}}\|^2$ we obtain that $\mathcal{E}(\Q)$ is bounded from below:
\begin{eqnarray}
\mathcal{E}(\Q)&\geq& \dfrac{K}{2}\|\nabla \Q\|^2_{L^2(\Omega)} + \dfrac{W}{2}\|\Q_{\text{pref}}-\Q\|^2_{L^2(\partial \mathcal{P}_{\text{st}})}-\dfrac{1}{2c_P}\|\Q\|^2_{L^2(\Omega)}-\dfrac{c_P}{2}\|\H_m\|^2_{L^2(\Omega)}\nonumber
\\
&\geq& -\dfrac{W}{2}\|\Q_{\text{pref}}\|_{L^2(\partial \mathcal{P}_{\text{st}})}-\dfrac{c_P}{2}\|\H_m\|^2_{L^2(\Omega)}.\nonumber
\end{eqnarray} 
Thus, there exists a minimizing sequence $\Q^{(\ell)}$ weakly converging in $H^1_{\text{per}}(\Omega)$. Then using the $\liminf$ property of a weakly converging sequences we get the existence of minimizer $\Q=\Q_m$.  

From $\mathcal{E}(\Q_m)\leq \mathcal{E}(\Q_\infty)$ and \eqref{poincare} we get that there exists $C>0$ such that 
\begin{eqnarray}
\dfrac{K}{8}\|\nabla \Q_m\|^2_{L^2(\Omega)} +\dfrac{1}{8c_P}\|\Q_m\|_{L^2(\Omega)}^2+\dfrac{W}{4}\|\Q_{\text{pref}}-\Q_m\|^2_{L^2(\partial \P_{\text{st}})}\leq 3c_P\|\H_m\|^2_{L^2(\Omega)}+C.
\label{gradient_estimate}
\end{eqnarray}
This shows \eqref{lemma32_H1}. Next, due to the elliptic regularity result (see Appendix~\ref{sec:elliptic})
\begin{equation}
\left\{
\begin{array}{l}
\Delta \Q= K^{-1}(\H_m-\hat{\H}_M(\Q_m))\text{ in }\Omega, \\
\partial_\nu \Q|_{\partial \P_{\text{st}}}=WK^{-1}(\Q_{\text{pref}}-\Q)
\end{array}
\right.
\nonumber
\end{equation}
we have 
\begin{equation}
\quad \|\Q\|_{H^2(\Omega)}^2\leq C \left(\dfrac{(W+K)^2}{W^2}\|\H_m-\hat{\H}_M(\Q_m)\|_{L^2(\Omega)}^2+\dfrac{W+K}{K}\|\Q_{\text{pref}}\|_{C^1}^2\right).
\end{equation}
Next, using \eqref{max_bound_on_H_M} we get \eqref{lemma32_H2} and it completes the proof of Lemma~\ref{lemma_existence_of_q_m}.

\section{Elliptic regularity for the squirmer boundary conditions}
\label{sec:elliptic}
Here, we consider the following auxiliary elliptic problem: 
\begin{equation}
\left\{
\begin{array}{l}
\Delta q= F\text{ in }\Omega, \\
\partial_\nu q|_{\partial \mathcal{P}_{\text{st}}}=\beta(\gamma(\x) -q),\\
q \text{ is  }\Pi\text{-periodic.}
\end{array}
\right.\label{aux_elliptic_problem}
\end{equation}
Here $F=F(\x)$, $\gamma=\gamma(\x)$ and $\beta>0$ are given. 
We aim to prove the elliptic regularity for \eqref{aux_elliptic_problem}:

\begin{thm}\it 
Let $q$ be the solution of \eqref{aux_elliptic_problem}. Then  
\begin{equation}
\|q\|_{H^2(\Omega)}^2\leq C \left(\frac{(1+\beta)^2}{\beta^2}\|F\|_{L^2(\Omega)}^2+ (1+\beta)\|\gamma\|_{C^1}^2\right) \label{classic_regularity}
\end{equation}
for some constant $C$ independent of $\beta$ and $\gamma$. 
\end{thm}

This result is well-known from PDE textbooks \cite[Theorem 8.12]{GilTru1983} and \cite[Theorem 4 in \S 6.3.2]{evans1998pde} for the Dirichlet boundary conditions. However, we need to re-visit this result due to our specific boundary conditions for which the afore-mentioned results are not applicable. For the sake of clarity, the proof below is written for two-dimensional case, $d=2$. 

\medskip 

\begin{proof}
We first address a priori estimates for regions near the boundary $\partial \mathcal{P}_{\text{st}}$ of the squirmer. Choose any point $\x_0$ on $\partial \mathcal{P}_{\text{st}}$. Suppose its vicinity on the boundary can be described by equation $x_2=\varphi(x_1)$ so the domain $x_2>\varphi(x_1)$ is the interior of domain $\Omega$. Then introduce change of variables 
\begin{equation}
\y=\Phi(\x) \Leftrightarrow 
\left\{
\begin{array}{l}
y_1=x_1,\\
y_2=x_2-\varphi(x_1).
\end{array}\right.\label{change_of_variables_y_x}
\end{equation} 
In variable $\y$, the problem \eqref{aux_elliptic_problem} has the form:
\begin{equation*}
\left\{
\begin{array}{l}
\nabla_{\y}\cdot (L(\y)\nabla_{\y}q)=F,\quad y_2>0\\
L(\y)\nabla_{\y}q\cdot \nu_\y=\beta\sqrt{1+(\varphi')^2}\,(\gamma- q),\quad y_2=0,
\end{array}
\right.
\end{equation*}
where
\begin{equation*}
L=
\left[
\begin{array}{cc}
1 & -\varphi' \\
-\varphi' & (1+(\varphi')^2)
\end{array}
\right], \quad
\nu_\y=
\left[
\begin{array}{c}
0\\1
\end{array}
\right].
\end{equation*}
We note that $L$ is a positive definite symmetric matrix with the smallest eigenvalue
\begin{equation}
\lambda_{\text{min}}(\y)=\dfrac{1}{2}\left(2+(\varphi')^2-\sqrt{(2+(\varphi')^2)^2-4}\right)\geq 1.
\end{equation}
Thus, $L$ is uniformly positive definite: 
\begin{equation}
\text{$(L(\y)\u\cdot \u)\geq |\u|^2$ for all $\y,\u\in \mathbb R^2$.}
\label{uniform_positivity}
\end{equation} 
\begin{lemma} 
	\it Let $q$ be the solution of 
	\begin{equation}
	\left\{
	\begin{array}{l}
	\nabla_\y \cdot (L(\y)\nabla_\y q) = \hat{F}, \quad \y\in \mathbb R_+^2=\left\{(y_1,y_2):y_2>0\right\}\\
	L(\y)\partial_{\y}q\cdot \nu_\y=\hat{\beta}(\hat{\gamma}- q), \quad y_2=0, 
	\end{array}
	\right.
	\label{aux_problem_weak}
	\end{equation}
	for some $f\in L^2_{loc}(\mathbb R_{+}^2)$, $\hat{\gamma}= \hat{\gamma}(y_1)\in H^1(\mathbb R)$ and $\hat{\beta}(y_1)\geq\beta_0>0$ and matrix $L(\y)$ satisfying uniform positivity condition \eqref{uniform_positivity}. Denote also $U=B_1(0)\cap\left\{y_2>0\right\}$ and $V=B_{1/2}(0)\cap\left\{y_2>0\right\}$. 
	
	Then we have the following bound: 
	\begin{equation}
	\|q\|_{H^2(V)}^2\leq C\left(\|\hat{F}\|_{L^2(U)}^2+\|q\|_{H^1(U)}^2+\|q\|_{L^2(U_0)}^2+\beta^2 \|\gamma\|_{C^1}^2\right).
	\label{ractified_regularity}
	\end{equation}
	Here $U_0=\{y_2=0\}\cap U$.
\end{lemma}

\begin{proof}
	We adapt arguments from \cite{evans1998pde}. All gradients $\nabla$ in the proof of this lemma are taken with respect to variable $\y$. First, we write the weak formulation of \eqref{aux_problem_weak} for all $v\in H^1(U)$:
	\begin{equation}
	\int\limits_{\tilde{\Pi}}L\nabla q \cdot \nabla v \,\text{d}y-\int\limits_{\{y_2=0\}\cap \tilde{\Pi}}\hat{\beta}(\hat{\gamma} -  q)v\,\text{d}S_y=-\int\limits_{\tilde{\Pi}} \hat{F}v\,\text{d}y.\label{weak_aux_problem_weak}
	\end{equation}
	Here $\tilde{\Pi}$ is the image of $\Pi\setminus \mathcal{P}_{\text{st}}$ under transformation \eqref{change_of_variables_y_x}. Next we introduce $\zeta(\y)$ such that $\zeta \in C^{\infty}$ and $\zeta\equiv 1$ in $V$ and $\zeta \equiv 0$ outside of $W=B_{{3/4}}(0)\cap\left\{y_2>0\right\}$. Take test function $v=D_{1}^{-h}(\zeta^2D_1^{h}q)$, where $D_1^{h}$ is the difference quotient operator: 
	\begin{equation}
	D_1^{h}g=\dfrac{g(y_1+h,y_2)-g(y_1,y_2)}{h}.
	\end{equation} 
	Integration by parts for the difference quotient allows us to rewrite the first integral in \eqref{weak_aux_problem_weak} as follows: 
	\begin{eqnarray}
	\int\limits_{\tilde{\Pi}}L\nabla q \cdot \nabla v \,\text{d}y&=&-\int\limits_{U}\zeta^2\,(LD^h_1(\nabla q)\cdot D^h_1(\nabla q)) \,\text{d}y-2\int\limits_{U}\zeta D_1^h q (\nabla \zeta \cdot LD^h_1(\nabla q))\,\text{d}y\nonumber\\
	&&\nonumber-\int\limits_{U}\zeta^2 (\left[D^h_1L\right](\nabla q) \cdot D^h_1(\nabla q))\,\text{d}y
	-2\int\limits_{U}\zeta D_1^h q (\nabla \zeta \cdot \left[D^h_1L\right](\nabla q))\,\text{d}y\nonumber\\
	&&\leq-\dfrac{1}{2}\int\limits_{U}\zeta^2|D^h_1(\nabla q)|^2 \,\text{d}y+C\int\limits_{W}|D^h_1q|^2\,\text{d}y+C\int\limits_{U}|\nabla q|^2\,\text{d}y.\label{aux_est_1}
	\end{eqnarray}
	Here, to obtain the estimate we used \eqref{uniform_positivity}, uniform boundedness of $\nabla \zeta$ and $[D^h_1L]$, and Cauchy-Schwarz inequality.
		
	Similarly, we rewrite the second integral in \eqref{weak_aux_problem_weak}:
	\begin{eqnarray}
	-\int\limits_{\{y_2=0\}\cap\tilde{\Pi}}\hat{\beta}(\hat{\gamma} -  q)v\,\text{d}S_y &=& -\int\limits_{\mathbb R}\hat{\beta}(\hat{\gamma} - q)D_{1}^{-h}(\zeta^2D_1^{h}q)\,\text{d}y_1\nonumber\\
	&=&\int\limits_{\mathbb R}\,D_1^{h}(\hat{\beta}\hat{\gamma} - \hat{\beta} q)\,(D_1^{h}q) \zeta^2\,\text{d}y_1\nonumber\\
	&=&\int\limits_{\mathbb R}\,\zeta^2 \left[D_1^{h}(\hat{\beta}\hat{\gamma})\right] D_1^{h}q\,\text{d}y_1-\int\limits_{\mathbb R}\zeta^2 q (D^h_1\hat{\beta}) D^h_1 q\,\text{d}y_1 -\int\limits_{\mathbb R}\hat{\beta}\zeta^2|D^h_1q|^2\,\text{d}y_1 \nonumber\\
	&\leq& -\dfrac{\beta_0}{2}\int\limits_{\mathbb R}\zeta^2|D^h_1q|^2\,\text{d}y_1+C\int\limits_{\mathbb R}\zeta^2 q^2\,\text{d}y_1+ C\int\limits_{\mathbb R}\zeta^2|D_1^h(\hat{\gamma}\hat{\beta})|^2\,\text{d}y_1. \label{aux_est_2}
	\end{eqnarray}

	Finally, we estimate the third integral in \eqref{weak_aux_problem_weak} using \cite[Theorem 3 from \S 5.8.2]{evans1998pde} as follows: 
	\begin{eqnarray}
	-\int\limits_{\tilde{\Pi}}\hat{F}v\,\text{d}y&&\leq C\left(\int\limits_{U}|\hat{F}|^2 \text{d}y\right)^{1/2}
	\left(\int\limits_{U}|\nabla(\zeta^2 D^h_1q)|^2\,\text{d}y\right)^{1/2}\nonumber 
	\\&&\leq C\int\limits_{U}|\hat{F}|^2 \,\text{d}y+\dfrac{1}{8}\int\limits_{U}|\nabla(\zeta^2D_1^hq)|^2\,\text{d}y\nonumber\\
	&&\leq C\int\limits_{U}|\hat{F}|^2 \,\text{d}y+\dfrac{1}{4}\int\limits_{U}\zeta^2|\nabla(D_1^hq)|^2\,\text{d}y+\dfrac{1}{4}\int\limits_{U}|\nabla(\zeta)^2\cdot D_1^hq|^2\,\text{d}y\nonumber\\
	&&\leq C\int\limits_{U}|\hat{F}|^2 \,\text{d}y+\dfrac{1}{4}\int\limits_{U}\zeta^2|\nabla(D_1^hq)|^2\,\text{d}y+C\int\limits_{W}|D_1^hq|^2\,\text{d}y.\label{aux_est_3}
	\end{eqnarray}
	Combining \eqref{weak_aux_problem_weak} with estimates \eqref{aux_est_1}, \eqref{aux_est_2}, and \eqref{aux_est_3}, we get: 
	\begin{eqnarray*}
	&&\dfrac{1}{4}\int\limits_{U}\zeta^2|D^h_1(\nabla q)|^2 \,\text{d}y+\dfrac{\beta_0}{2}\int\limits_{\mathbb R}\zeta^2|D^h_1q|^2\,\text{d}y_1 \nonumber \\
	&&\hspace{60 pt} \leq C\left[\int\limits_{U}|\hat{F}|^2 \,\text{d}y+\int\limits_{W}|D^h_1q|^2\,\text{d}y+\int\limits_{U}|\nabla q|^2\,\text{d}y\right.\nonumber\\
	&&\hspace{120pt}\left.+\int\limits_{\mathbb R}\zeta^2 q^2\,\text{d}y_1+ \int\limits_{\mathbb R}\zeta^2|D_1^h(\hat{\gamma}\hat{\beta})|^2\,\text{d}y_1\right].
	\end{eqnarray*}
	Using \cite[Theorem 3(i) from \S 5.8.2]{evans1998pde} we get
	\begin{eqnarray*}
	&&\dfrac{1}{4}\int\limits_{V}|D^h_1(\nabla q)|^2 \,\text{d}y+\dfrac{\beta_0}{2}\int\limits_{\mathbb R \cap V}|D^h_1 q|^2\,\text{d}y_1 \nonumber \\
	&&\hspace{60 pt} \leq C\left[\int\limits_{U}|\hat{F}|^2 \,\text{d}y+\int\limits_{U}|\nabla q|^2\,\text{d}y+\int\limits_{\mathbb R}\zeta^2 q^2\,\text{d}y_1+ \int\limits_{\mathbb R}\zeta^2|\partial_{y_1} (\hat{\gamma}\hat{\beta})|^2\,\text{d}y_1\right].
	\end{eqnarray*}
	Then \cite[Theorem 3(ii) from \S 5.8.2]{evans1998pde} implies
	\begin{eqnarray}
	&&\dfrac{1}{4}\int\limits_{V}|\partial_{y_1}(\nabla q)|^2 \,\text{d}y+\dfrac{\beta_0}{2}\int\limits_{\mathbb R \cap V}|\partial_{y_1} q|^2\,\text{d}y_1 \nonumber \\
	&&\hspace{60 pt} \leq C\left[\int\limits_{U}|\hat{F}|^2 \,\text{d}y+\int\limits_{U}|\nabla q|^2\,\text{d}y+\int\limits_{\mathbb R}\zeta^2 q^2\,\text{d}y_1+ \int\limits_{\mathbb R}\zeta^2|\partial_{y_1} (\hat{\gamma}\hat{\beta})|^2\,\text{d}y_1\right].
	\end{eqnarray}
	Now write 
	$\partial_{y_1} (\hat{\gamma}\hat{\beta})=\beta\left(\dfrac{\varphi''\varphi'}{\sqrt{1+(\varphi')^2}}\gamma+\sqrt{1+(\varphi')^2}(\gamma_{x_1}+\gamma_{x_2}\varphi')\right)$. Thus,
	\begin{equation*}
	\int\limits_{\mathbb R}\zeta^2|\partial_{y_1} (\hat{\gamma}\hat{\beta})|^2\,\text{d}y_1\leq C\beta^2 \|\gamma\|_{C^1}^2.
	\end{equation*}

	Analogous estimate is valid for $\partial^2_{x_2}q$ since
	\begin{equation*}
	\partial^2_{y_2}q=-\dfrac{1}{1+(\varphi')^2}\left[\hat{F}-\partial_{y_1}^2q+2\varphi'\partial^2_{y_1y_2}q-\varphi''\partial_{y_2}q+2\varphi'\varphi''\partial_{y_2}q\right]
	\end{equation*} 
	and then 
	\begin{equation*}
	|\partial_{y_2}^2q|\leq C\left(|\hat{F}|+|\partial_{y_1}(\nabla q)|+|\nabla q|\right)
	\end{equation*}
	and thus \eqref{ractified_regularity} is proven. 
\end{proof}

\medskip 

Next from from \eqref{ractified_regularity} and interior regularity \cite[Theorem 8.8]{GilTru1983} we have that for and $\tilde{\Omega} \subset \subset \Omega$: 
\begin{equation}
\|q\|_{H^2(\tilde{\Omega})}\leq C\left(\|F\|_{L^2(\Omega)}+\|q\|_{L^2(\Omega)}\right).
\end{equation}
To obtain a bound on $\|q\|_{L^2(\Omega)}$, we will use that $q$ from \eqref{aux_elliptic_problem} minimizes the energy functional 
\begin{equation}
\mathcal{E}_0(q)=\dfrac{1}{2}\int\limits_{\Omega}|\nabla q|^2 \,\text{d}x+\dfrac{\beta}{2}\int\limits_{\Omega}|\gamma-q|^2\,\text{d}S_x+\int\limits_{\Omega}Fq\,\text{d}x.
\end{equation}
From $\mathcal{E}_0(q)\leq \mathcal{E}_0(0)$ we get for all $\delta>0$
\begin{equation*}
\int\limits_{\Omega}|\nabla q|^2 \,\text{d}x+\beta\int\limits_{\Omega}|q|^2\,\text{d}S_x\leq C\beta\|\gamma\|_C^2+C\delta^{-1}\|F\|^{2}_{L^2(\Omega)}+\delta\|q\|_{L^2(\Omega)}^2. 
\end{equation*}
Next, we use Poincar\'e estimate \eqref{poincare_without_coefs}: 
\begin{eqnarray*}
\|q\|_{L^2(\Omega)}^2\leq C\left(1+\dfrac{1}{\beta}\right)\left(\beta\|\gamma\|_C^2+C\delta^{-1}\|F\|^{2}_{L^2(\Omega)}+\delta\|q\|_{L^2(\Omega)}^2\right).
\end{eqnarray*}
Take $\delta:=\dfrac{1}{2C}\left(1+\dfrac{1}{\beta}\right)^{-1}$ and we get 
\begin{eqnarray*}
\|q\|_{L^2(\Omega)}^2\leq C(\beta+1)\|\gamma\|_C^2+C\left(1+\dfrac{1}{\beta}\right)^2\|F\|^{2}_{L^2(\Omega)}.
\end{eqnarray*}

Finally, we conclude that 
\begin{eqnarray*}
\|q\|^2_{H^2(\Omega)}\leq C(\beta+1)\|\gamma\|_C^2+C\left(1+\dfrac{1}{\beta}\right)^2\|F\|^{2}_{L^2(\Omega)}.
\end{eqnarray*}
\end{proof}


\section{Rescaling}
\label{sec:homo_rescaling}
In this Appendix we present non-dimensionalization of the steady state problem, 
showing how the scalings in \eqref{eq:pre_homo_Q_eq_re}-\eqref{eq:pre_homo_u_bc_re} arise. We will assume that all quantities are in their physical dimensions. Representative values of physical parameters can be found in Table~\ref{table:parameters}.

\begin{table}[h]
\begin{center}
	\rowcolors{2}{gray!25}{white!50}
	\begin{tabular}{|c c c c|} 	
		\hline 
		physical parameter & value & unit & representation in $\delta$'s \\ [0.5ex] 
		\hline
		$K$ & $10^{-8}$ & $\text{N}$ & $10^{-8}\,\, \delta_f$\\ 
		$\eta$ & $1$ & $\text{N}\cdot\text{s}/\text{m}^2$ & $10^{-4} \,\,\delta_T \delta_f / \delta_L^2$\\
		$W$ & $10^{-6}$ & $\text{N}/\text{m}$  & $10^{-8}\,\, \delta_f / \delta_L$\\
		$\Gamma$ & $1$ & $\text{m}^2/(\text{N} \cdot \text{s})$ & $10^4\,\,\delta_L^2/\delta_f\delta_T$  \\
		$a$ & $0.4$ & $\text{N}/\text{m}^2$ & $4 \times 10^{-5}  \delta_f / \delta_L^2$\\
		$c$ & $0.8$ & $\text{N}/\text{m}^2$ & $8 \times 10^{-5}  \delta_f / \delta_L^2$\\ 
		$\rho$ & $1.0$ & $\text{g}/\text{mL}$ & $10^{-5}\,\, \delta_f\delta_T^2/\delta_L^4$ \\
		$\zeta$ & $2.0$ & $\text{s}^{-1}$ & $2\,\, \delta_T^{-1}$ \\
		$v_{\text{prop}}$ &  $10^{-6}$& $\text{m}/\text{s}$ &  $10^{-4}\,\,\delta_L/\delta_T$\\
		\hline
	\end{tabular}
\caption{Values of physical parameters, taken from \cite{chi2020surface,genkin2017defects}.}
\label{table:parameters}
\end{center}
\end{table}

Introduce characteristic length $\delta_{L} = 10^{-2}\,\text{m}$,  time $\delta_T = 1\,\text{s}$, 
and force $\delta_f = 1\,\tN$. Non-dimensional flow velocity and pressure are  
\begin{eqnarray}
\u = \tilde{\u}\, \frac{\delta_L}{\delta_T}\text{ and }
 p = \tilde{p}\, \frac{\delta_f}{\delta_L^2}. 
\end{eqnarray}
Note that tensor order parameter $\Q$ is non-dimensional and does not require a non-dimensio\-nal\-ization. We also represent the external alignment field $F_\text{ext}$ as 
\begin{equation*}
F_{\text{ext}}=\zeta \tilde{F}_{\text{ext}},
\end{equation*}
where $\tilde{F}_{\text{ext}}$ is non-dimensional.

Using $\nabla_\x=\delta_L^{-1}\nabla_{\tilde{\x}}$, PDEs \eqref{eq:stst_lc_1} and \eqref{eq:stst_lc_4} reduce to
\begin{eqnarray}
&&\hspace{-20pt}\dfrac{\delta_T\Gamma K}{\delta_L^2}\Delta\Q+\delta_T\Gamma a\, \Q-\delta_T\Gamma c\, \Q\text{Tr}(\Q^2)+S(\nabla \tilde{\u}, \Q) - \tilde{\u} \cdot \nabla \Q+\delta_T\zeta\, \tilde{F}_\text{ext}=0,\label{rescaled-Q}\\
&&\hspace{-20pt}\dfrac{\rho\delta_L^4}{\delta_T^2\delta_f}(\tilde\u\cdot\nabla)\tilde{\u}-\dfrac{\eta\delta_L^2}{\delta_T\delta_f}\Delta \tilde{\u} + \nabla\tilde{p} = \dfrac{K}{\delta_f} \nabla\cdot(\nabla {\Q} \astrosun \nabla {\Q} + {\Q}\Delta {\Q} - \Delta {\Q} {\Q}).\label{rescaled-u}
\end{eqnarray}
Here and below in this section all spatial derivatives are taken with respect to $\tilde{\x}$. 

Boundary condition \eqref{eq:stst_lc_3} and \eqref{eq:stst_lc_6} becomes
\begin{equation}
\tilde{\u}=\dfrac{\delta_T}{\delta_L}v_{\text{prop}}{u}^{(p)}_{\text{sq}}\boldsymbol{\tau}\quad \text{and}\quad \partial_\nu \Q=\dfrac{W\delta_L}{K}(\Q_{\text{pref}}-\Q).
\end{equation}
Here we represented $u_{\text{sq}}=v_{\text{prop}}{u}^{(p)}_{\text{sq}}$ where ${u}^{(p)}_{\text{sq}}$ is the profile of the propulsion such that  $\max{{u}^{(p)}_{\text{sq}}}=1$ and $v_{\text{prop}}$ is the propulsion strength. 
%
Introduce rescaled parameters: 
\begin{eqnarray}
&\varepsilon=\dfrac{L}{\delta_L},\quad \gamma=\dfrac{\delta_T\Gamma K}{L\delta_L}, \quad \tilde{a}=\delta_T\Gamma a,\quad \tilde{c}=\delta_T\Gamma c,\quad \tilde{\zeta}=\delta_T\zeta,&\nonumber \\
&\tilde{\rho}=\dfrac{\rho\delta_L^5}{L\delta_T^2\delta_f}, \quad \tilde{\eta}=\dfrac{\eta \delta_L^3}{L\delta_T\delta_f},\quad  \kappa=\dfrac{K\delta_L^2}{L^2\delta_f}, \quad \tilde{W}=\dfrac{W}{\delta_L}{K},\quad \tilde{v}_{\text{prop}}=\dfrac{\delta_T}{L}v_{\text{prop}}.& \label{non-dim-parameters} 
\end{eqnarray}
Specific values of these parameters can be found in Table~\ref{table:parameters-rescaled}.

\begin{table}[h]
	\begin{center}
		\rowcolors{2}{gray!25}{white!50}
		\begin{tabular}{|c c|} 
			\hline
			rescaled parameter & value  \\ [0.5ex] 
			\hline
			$\ve$ & $10^{-4}$\\ 
			$\gamma$ & $1$ \\
			$\tilde{a}$ & $0.4$\\
			$\tilde{c}$ & $0.8$ \\
			$\tilde{\zeta}$ & $2.0$ \\
			$\tilde{\rho}$ & $0.1$ \\ 
			$\tilde{\eta}$ & $1.0$  \\
			$\kappa$ & $1.0$ \\
			$\tilde{W}$ & $1.0$ \\
			$\tilde{v}_{\text{prop}}$ & $1.0$ \\
			\hline
		\end{tabular}
		\caption{Values of non-dimensional parameters introduced in \eqref{non-dim-parameters} corresponding to values of physical parameters from Table \eqref{table:parameters}.}
		\label{table:parameters-rescaled}
	\end{center}
\end{table}

Then PDEs \eqref{rescaled-Q} and \eqref{rescaled-u} become
\begin{eqnarray*}
&&\varepsilon\gamma\Delta\Q+\tilde{a}\, \Q-\tilde{c}\, \Q\text{Tr}(\Q^2)+S(\nabla \tilde{\u}, \Q) - \tilde{\u} \cdot \nabla \Q+\tilde{\zeta}\, \tilde{F}_\text{ext}=0\text{ in }\Omega_{\ve},\\
&&\varepsilon\tilde{\rho}(\tilde\u\cdot\nabla)\tilde{\u}-\varepsilon\tilde{\eta}\Delta \tilde{\u} + \nabla\tilde{p} = \varepsilon^2 \kappa\nabla \cdot(\nabla {\Q} \astrosun \nabla {\Q} + {\Q}\Delta {\Q} - \Delta {\Q} {\Q}) \text{ in }\Omega_{\ve},
\end{eqnarray*}
with boundary conditions 
\begin{eqnarray*}
&&\tilde{\u}=\ve \tilde{u}_{\text{sq}}\,\boldsymbol{\tau} \text{ and }\partial_\nu \Q=\tilde{W}(\Q_{\text{pref}}-\Q) \text{ in } \partial\mathcal{P}_\ve,\\
&& \tilde{\u} \text{ and } \Q \text{ are }2\varepsilon-\text{periodic}.  
\end{eqnarray*}
Here, $\tilde{u}_{\text{sq}}=\tilde{v}_{\text{prop}}{u}^{(p)}_{\text{sq}}$, $\Omega_{\ve}=\delta_L^{-1}\Omega$ and $\mathcal{P}_\ve=\delta_L^{-1}\mathcal{P}$. 

\end{document}